\documentclass[12pt,reqno]{amsart}

\usepackage{latexsym}
\usepackage{amssymb}

\usepackage{epsfig}
\newcommand{\PR}{{\bf Prob}}

\newcommand{\szego}{Szeg\"o\ }

\newcommand{\kahler}{K\"ahler\ }

\newcommand{\PP}{{\mathbb P}}
\newcommand{\R}{{\mathbb R}}
\newcommand{\C}{{\mathbb C}}

\newcommand{\Z}{{\mathbb Z}}

\newcommand{\CP}{\C\PP}

\newcommand{\dbar}{\bar\partial}
\newcommand{\ddbar}{\partial\dbar}

\renewcommand{\H}{{\mathbf H}}
\newcommand{\half}{{\frac{1}{2}}}

\renewcommand{\phi}{\varphi}

\newcommand{\acal}{\mathcal{A}}
\newcommand{\bcal}{\mathcal{B}}

\newcommand{\dcal}{\mathcal{D}}
\newcommand{\ecal}{\mathcal{E}}
\newcommand{\fcal}{\mathcal{F}}
\newcommand{\gcal}{\mathcal{G}}
\newcommand{\hcal}{\mathcal{H}}
\newcommand{\jcal}{\mathcal{J}}
\newcommand{\lcal}{\mathcal{L}}
\newcommand{\pcal}{\mathcal{P}}
\newcommand{\mcal}{\mathcal{M}}
\newcommand{\ncal}{\mathcal{N}}

\newcommand{\ocal}{\mathcal{O}}

\newcommand{\zcal}{\mathcal{Z}}

\newtheorem{theo}{{\sc Theorem}}

\newtheorem{cor}[theo]{{\sc Corollary}}

\newtheorem{lem}[theo]{{\sc Lemma}}
\newtheorem{prop}[theo]{{\sc Proposition}}

\newenvironment{rem}{\medskip\noindent{\it Remark:\/} }{\medskip}

\newtheorem{maintheo}{{\sc Theorem}}
\newtheorem{mainprop}{{\sc Proposition}}
\newtheorem{maindefin}{{\sc Definition}}
\title[Large deviations of empirical  measures of zeros  on Riemann
surfaces
 ] {Large deviations of empirical  measures of zeros on Riemann surfaces}

\author{Steve Zelditch}

\address{Department of Mathematics, Northwestern University,
Evanston IL,  60208-2730, USA}

\thanks{Research partially supported by NSF grant
  DMS-0904252}

\begin{document}

\begin{abstract} We determine an LDP (large deviations principle) for the empirical measure $$ \tilde{Z}_s: = \frac{1}{N} \sum_{\zeta: s(\zeta) = 0}
 \delta_{\zeta}, \;\;\; (N: = \# \{\zeta: s(\zeta) = 0)\})$$
 of zeros of  random holomorphic sections $s$ of random  line
 bundles $L \to X$  over a Riemann surface $X$ of genus $g \geq 1$. In a previous
 article \cite{ZZ},  O. Zeitouni and the author proved such an LDP
 in the $g = 0$  case  of  $\CP^1$ using an explicit formula for the JPC (joint probability
 current) of zeros of Gaussian random polynomials. The main purpose of this article
 is to define Gaussian type measures on the ``vortex moduli  space" of all holomorphic sections of all
 line bundles of degree $N$ and to calculate   its JPC  as a volume form on the configuration
 space $X^{(N)}$ of $N$ points of $X$.   The calculation involves
 the higher genus analogues of Vandermonde determinants, the prime
 form and bosonization.

 \end{abstract}

\bigskip

\bigskip

\maketitle

\tableofcontents

In a recent article,  O. Zeitouni and the author \cite{ZZ} proved
 an LDP (large deviations
principle) for the empirical measure \begin{equation}\label{ZN}
\tilde{Z}_s: = d\mu_{\zeta}: =  \frac{1}{N} \sum_{\zeta: s(\zeta)
= 0}
 \delta_{\zeta}, \;\;\; N: = \# \{\zeta: s(\zeta) = 0\} \end{equation}
  of zeros of Gaussian
random polynomials $s$ of degree $N$ in one complex variable
(where $\delta_{\zeta}$ is the Dirac point measure at $\zeta $.)
The purpose of this continuation is to generalize the LDP to
holomorphic sections of line bundles  $L \to X$ of degree $N$ over
a compact Riemann surface $X$ of genus $g \geq 1$. Roughly
speaking, the LDP determines the
 asymptotic probability that a
configuration $\{P_1, \dots, P_N\}$ of $N$ points  is the zero set
of a random holomorphic section $s$ of a line bundle  $L$ of
degree  $N$ as $N \to \infty$.
 The essentially new aspect of
higher genus Riemann surfaces is that $L$ is not unique but rather
varies over the $g$-dimensional Picard variety $\mbox{Pic}^N$ of
holomorphic line bundles of degree $N$. As recalled in \S
\ref{BACKGROUND},
 the space $\mbox{Pic}^N$ of line bundles of degree $N$ is a compact complex torus of dimension $g$.
 The space of all holomorphic sections of all line bundles of
 degree $N$ is therefore
  the total space
 \begin{equation} \label{ECALDEF} \ecal^N : = \bigcup_{\xi \in Pic^N}  H^0(X, \xi) \end{equation}
of the complex holomorphic vector bundle (the Picard bundle),
\begin{equation} \label{ECALVB} \pi_N: \ecal^N \to \mbox{Pic}^N
\end{equation}
 whose fiber $\ecal^N_{\xi}$ over $\xi$ is the space $H^0(X, \xi)$ of holomorphic sections of $\xi$.
Since $\ecal_N$ is a vector bundle rather than a vector space, it
does not carry a Gaussian measure as in the genus zero case of
\cite{ZZ}.  But  it does carry  closely related types of
Gaussian-like measures introduced in   Definitions
\ref{FSHAARDEF}- \ref{LARGEVS}. Since we are interested in zeros,
it is natural to identify sections which differ by a constant
multiple, i.e. to projectivize $H^0(X, \xi)$ to $\PP H^0(X, \xi)$,
and to push forward the Gaussian type measures to Fubini-Study
type probability measures on the $\CP^{N-g}$- bundle
\begin{equation} \label{PECALDEF}
\PP \ecal^N : = \bigcup_{\xi \in Pic^N}  \PP H^0(X, \xi) \to
\mbox{Pic}^N.
\end{equation}
We endow the total space with Fubini-Study volume forms along the
fibers and Haar measure along the base.  The resulting probability
measure is called the {\it Fubini-Study-fiber ensemble}
(Definition \ref{FSHAARDEF}). In addition we will define a
Fubini-Study fiber ensemble over the configuration space $X^{(g)}$
of $g$ points (Definition \ref{FSTILDEDEF}) and   a more linear
{\it projective linear ensemble} by embedding $\PP \ecal^N$ in a
higher dimensional projective space (Definition \ref{LARGEVS}).

The projectivized Picard bundle  $\PP \ecal_N$ is analytically
equivalent to
 the configuration space
 \begin{equation} \label{CONFIG} X^{(N)} = Sym^{N} X :=  \underbrace{X\times\cdots\times
X}_N /S_{N} \end{equation}
 of $N$ points of  $X$ under the
 `zero set' or
divisor map
  \begin{equation} \label{ZEROSMAPg} \dcal: \PP \ecal_N \to X^{(N)}, \;\;
  \;\;\;\dcal(\xi, s)= \zeta_1 + \cdots + \zeta_N, \;\; Z_s = \{\zeta_j\}.\end{equation}
   Here, $S_N$ is the symmetric
  group on $N$ letters.
This analytic equivalence explains why it is natural to view the
line bundle as well as the section as a random variable. Any
configuration $\{\zeta_1, \dots, \zeta_N\}$ is the zero set of
some section of some line bundle of degree $N$, but the  possible
zero sets of sections $s \in H^0(X, L)$ of a fixed $L \in
\mbox{Pic}^N$ lie on a codimension $g$ submanifold of to
$X^{(N)}$. (As will be recalled in \S \ref{BACKGROUND},  the
submanifold is  a fiber of the Abel-Jacobi map $ A_N: X^{(N)} \to
\mbox{Jac}(X).$) $ \PP \ecal_N $ is also the  moduli space of
abelian vortices of Yang-Mill-Higgs fields of vortex number $N$
(\cite{Sam,MN}), and is therefore also called the vortex moduli
space.

A Gaussian type probability measure on $ \ecal_N$ weights a
section in terms of its $L^2$ norm with respect to an inner
product, or equivalently in terms of its coefficients relative to
an orthonormal basis. Under $\dcal$, it induces a probability
measure $\vec K^N$ on $X^{(N)}$, which is called the {\it joint
probability distribution} JPD of zeros of the random sections $s
\in \PP \ecal_N$. We also refer to it as the JPC (joint
probability current) since it is naturally a (possibly singular)
volume form on $X^{(N)}$.  Its integral over an open set $U
\subset X^{(N)}$ is the probability that the zero set of the
random section
 lies in $U$. The first objective of this
article is to calculate it explicitly for the Gaussian-type
probability measures on $\ecal_N$ mentioned above in terms of
products of Green's functions and prime forms.  This calculation
is much more involved than in the genus zero case, and is based on
  bosonization formulae on higher genus Riemann surfaces in
\cite{ABMNV,VV,F,DP,W}

 We then use  the explicit formula for $\vec K^N$ to obtain a rate function for the
 large deviations principle, which (roughly speaking) gives
 the asymptotic  probability as $N \to \infty$  that a  configuration of $N$ points is the zero
 sets
 of  a random section.  Since the spaces $X^{(N)}$ change with $N$, we
encode a configuration by its empirical measure,
\begin{equation} \{\zeta_1, \dots, \zeta_N\} \to \mu_{\zeta}: =
\frac{1}{N} \sum_{j = 1}^N \delta_{\zeta_j} \in \mcal(X).
\end{equation} We denote the map from configurations to empirical
measures by
\begin{equation} \label{MUDDEF} \mu: X^{(N)} \to \mcal(X),
\end{equation}
where $\mcal(K)$ denotes the convex set of probability
 measures on a  set $K$.
Under  the  map $\mu$ the JPD (or JPC) pushes forward to
 a probability  measure
\begin{equation} \label{LDPNa} \PR_N =  \mu_* \vec K^N \end{equation}  on
 $\mcal(X)$.

  Our main result is the analogue in higher genus of
 the LDP in genus zero. To state the result, we need some further
 notation. As in \cite{ZZ}, the input for our probability measures
 on sections is a pair $(\omega, \nu)$ where $\omega$ is a real
 $(1,1)$ form and where $\nu$ is a probability measure on $X$
 satisfying the two rather small technical hypotheses (\ref{BM}) and
 (\ref{REGULAR}) carried over from  \cite{ZZ}.
We denote by $G_{\omega}$ the Green's function of $X$ with respect
to $\omega$, and by
\begin{equation}\label{Green's potential} U_{\omega}^{\mu} (z) =
\int_X G_{\omega}(z,w) d\mu(w) \end{equation} the Green's
potential of a probability measure $\mu \in \mcal(X)$ (see \S
\ref{GRF} for background).  We define the
 Green's energy of $\mu \in \mcal(X)$  by
\begin{equation}\label{GEN} \ecal_{\omega}(\mu) = \int_{X \times X} G_{\omega}(z,w)
d\mu(z) d\mu(w). \end{equation}

As in \cite{ZZ}, we have:

\begin{maindefin}
{\it The LD  rate functional is defined by,
\begin{equation} \label{IGREEN}  I^{\omega, K} (\mu) =
- \frac{1}{2} \ecal_{\omega}(\mu) + \sup_K U^{\mu}_{\omega},\quad
\mu\in \mcal(X)
\end{equation}}

\end{maindefin}

  We also let
\begin{equation}
    \label{eq-ofer1}
    E_0(\omega)=\inf_{\mu\in \mcal(X)} I^{\omega,K}(\mu), \quad
    \tilde I^{\omega,K}=I^{\omega,K}-E_0(\omega)\,.
\end{equation}
It is proved in \cite{ZZ} that the infimum $\inf_{\mu\in \mcal(X)}
I^{\omega,K}(\mu)$ is achieved at the unique Green's equilibrium
measure $\nu_{\omega, K}$ with respect to $(\omega, K)$, and
$E_0(\omega)= \frac12 \log \mbox{Cap}_{\omega} (K)$, where
$\mbox{Cap}_{\omega}(K)$ is the Green's capacity.  By the Green's
equilibrium measure we mean the minimizer of $- \ecal_{\omega}$ on
$\mcal(K)$.

 \begin{maintheo}\label{LD} Assume
that $d\nu \in  \mcal(X)$ satisfies the Bernstein-Markov property
(\ref{BM}) and that its support is nowhere thin
 (\ref{REGULAR}). Then, for the Fubini-Study-fiber ensembles (see  Definitions
\ref{FSHAARDEF}- \ref{FSTILDEDEF}), $\tilde I^{\omega,K}$ of
\eqref{eq-ofer1} is a convex rate function and the sequence of
probability measures $\{{\bf Prob}_N\}$ on $\mcal(X)$ defined by
(\ref{LDPNa}) satisfies a large deviations principle with speed
$N^2$ and rate function $\tilde I^{\omega, K} $, whose unique
minimizer $\nu_{h, K} \in \mcal(X)$ is the Green's equilibrium
measure of $K$ with respect to $\omega$.
\end{maintheo}

Roughly speaking an LDP with speed $N^2$ and rate function $I$
states that, for any Borel subset $E \subset \mcal(X)$,
$$\frac{1}{N^2} \log \PR_N \{\sigma \in \mcal: \sigma \in E\}
\to - \inf_{\sigma \in E} I(\sigma). $$ The rate function is the
same as in the genus zero case in \cite{ZZ}. The  analytical
problems involved in proving the LDP from the formula for the JPC
are similar to those in the genus zero case of \cite{ZZ}. The
principal new feature in this work is the derivation of an
explicit formula for the JPC for our ensembles.

Roughly speaking, the existence of the LDP is due to the explicit
relation between two structures on $X^{(N)}$:

\begin{itemize}

\item The fiber bundle structure coming from the Abel-Jacobi
fibrations $X^{(N)} \to \mbox{Jac(X)}$
(\ref{ECALDEF})-(\ref{PECALDEF}). This gives rise to Fubini-Study
metrics along the fibers, which are the source of the probability
measures we define, based on the idea that the term `random
section' should refer to random coefficients relative to a fixed
basis.

\item The product structure of $X^{(N)}$: Although it is the
quotient of the Cartesian product $X^N$ by the symmetric group
$S_N$, the order $N!$ of $S_N$ is negligible when taking the
$\frac{1}{N^2} \log $ limit, so we may think of $X^{(N)}$ as
essentially a product.   As such, it carries the exterior product
measures $\pi_1^*\omega \boxtimes \cdots \boxtimes \pi_N^*\omega$
where $\omega$ is any \kahler form on $X$ and $\pi_j: X^N \to X$
is the projection onto the $j$th factor.  In practice, we use the
local coordinate volume form $\prod_{j = 1}^N d \zeta_j \wedge d
\bar{\zeta}_j$ where $\zeta_j$ is the local uniformizing
coordinate on $X$. The main step in the proof of Theorem \ref{LD}
is to express the probability measures on $\PP \ecal_N$ defined by
the fiber bundle structure  in terms of product measures.
Estimates of the product measure in Lemma \ref{ZEROa}  are then
crucial to the proof of Theorem \ref{LD}.

\end{itemize}

The relation between the fiber bundle structure and the product
structure of $X^{(N)}$  permeates the proof of Theorem \ref{LD},
sometimes in an implicit way. For instance, our use of
bosonization in calculating the JPC is essentially to relate these
structures.

 \subsection{\label{PLSFSH} The Projective Linear Ensemble and the Fubini-Study-fiber ensemble}

We now define the basic probability measures on $\PP \ecal_N$ of
this article. As discussed in \cite{ZZ,SZ} (and elsewhere),
Gaussian (or Fubini-Study) measures on a vector space $V$ (or
projective space $\PP V$) correspond to a choice of Hermitian
inner product on $V$. When $V = H^0(X, \xi)$, inner products may
be defined by choosing a Hermitian metric $h$ on $\xi$ and a
probability measure $\nu$ on $X$.  The data $(h, \nu)$ induces the
inner product
\begin{equation} \label{DEFGN} ||s||^2_{G(h, \nu)} := \int_X
|s(z)|^2_{h} d\nu(z).
\end{equation}
The  Hermitian inner product $G(h, \nu)$ in turn  induces a
complex Gaussian measure $\gamma(h, \nu)$ on  $H^0(X, \xi)$: If
$\{S_j\}$ is an orthonormal basis for $G(h, \nu)$, then the
Gaussian measure is given in coordinates with respect to this
basis by,
\begin{equation}\label{gaussian}d\gamma(s):=\frac{1}{\pi^m}e^
{-|c|^2}dc\,,\quad s=\sum_{j=1}^{d}c_jS_j\,,\quad
c=(c_1,\dots,c_{d})\in\C^{d}\,.\end{equation} Here, $d = \dim_{\C}
H^0(X, \xi)$. As in \cite{ZZ} we assume throughout that $\nu$ is a
Bernstein-Markov measure whose support is non-thin at all of its
points (see \S \ref{HIP} for background).

Since we are studying zeros, it is   natural to push a Gaussian
measure $\gamma_{G(h, \nu)}$  on $H^0(X, \xi)$  under the natural
projection $H^0(X, \xi)^* \to \PP H^0(X, \xi)$ to obtain a
Fubini-Study measure $dFS_{G(h, \nu)}$  induced by $(h, \nu)$. In
fact,  our important maps factor through the projective space, so
it simplifies things to  use  Fubini-Study volume forms from the
start.  We refer to \cite{ZZ}, \S 3.2, for further discussion of
this step.

The fact that $\ecal_N$ is a vector bundle rather than a vector
space complicates this picture in two ways: first, we need to
define a   family $G_N(\xi)$ of Hermitian inner products on the
spaces $H^0(X, \xi)$ as  $\xi$ varies over $\mbox{Pic}^N$. If we
are given a family $G_N(\xi)$ of Hermitian metrics on $H^0(X,
\xi), \xi \in \mbox{Pic}^N$, we denote by $\gamma_{G_N(\xi)}$ the
associated family of Gaussian measures on $H^0(X, \xi)$ and by
$dV^{FS}_{G_N(\xi)}$ the associated family of Fubini-Study
measures on $\PP H^0(X, \xi)$ (the pushforwards of the
$\gamma_{G_N(\xi)}$). Second, Gaussian measures along the fibers
of $\ecal_N$ do not define a probability measure on $\ecal_N$; we
also need a probability measure on the `base', $\mbox{Pic}^N$. For
expository simplicity, we assume that the base measure is
normalized  Haar measure $d \theta$  on $\mbox{Pic}^N \simeq
\mbox{Jac}(X)$.

\begin{maindefin}  \label{FSHAARDEF}   Given a family of Hermitian inner products
$G_N(\xi)$ for $\xi \in \mbox{Pic}^N$, and a volume form $d
\sigma$ on $\mbox{Pic}^N$,
  the associated  {\it Fubini-Study-fiber} measure
 $d\tau_N^{FSH}$ on $\PP \ecal^N $ is the fiber-bundle product measure,
  \begin{equation} \label{gammadef} \int_{\PP \ecal^N} F(s) d\tau_N^{FSH} (s)
  := \int_{Pic^N}\; \{\int_{\PP H^0(X, \xi)} F(s) dV^{FS}_{G_N(\xi)} \} \;d\sigma(\xi), \end{equation}
where $dV^{FS}_{G_N(\xi)}$ are the fiber Fubini-Study volume
  forms defined by the Hermitian inner products on $H^0(X, \xi)$
  induced by $G_N(\xi)$. We also denote by $\hat{\tau}_N$
  the analogous construction for Gaussian-fiber  measures. In the
  special case where $d \sigma$ is normalized Haar measure, we
  call $d\tau_{FSH}$ Fubini-Study-Haar measure.
 \end{maindefin}

In genus zero, where the only line bundle of degree $N$ is
$\ocal(N)$,  the inner products $G_N(\omega, \nu)$ on $H^0(\CP^1,
\ocal(N))$ were those induced from tensor powers $h^N$ of a fixed
Hermitian metric $h$ with curvature form $\omega$  on $\ocal(1)$
and from $\nu$. In higher genus, we would like to induce a family
of   inner products $G_N(\omega, \nu)$ on all $H^0(X, \xi)$ from
the data $(\omega, \nu)$. It is a slight but useful change in
viewpoint to regard the basic data as $(\omega, \nu)$ rather than
$(h, \nu)$
 since $h$ is determined by $\omega$ only  up to a constant.

 The data $(\omega, \nu)$ can be used to determine a family
 $G_N(\omega, \nu)$ of Hermitian inner products on $H^0(X, \xi)$
 in at least two natural ways. The first is to choose a family of
 Hermitian metrics $h_N(\xi)$ on the family $\xi$ such that the
 curvature $(1,1)$ forms of the $h_N(\xi)$ equal $\omega$. We will
 refer to such a family as $\omega$-admissible.
 Unfortunately, such a family is only determined up to a function
 on $\mbox{Pic}^N$. The function may be fixed up to an overall
 constant using the so-called Faltings' metric on the determinant
 line bundle $\bigwedge^{\mbox{top}} H^0(X, \xi)$. This is a
 natural and attractive approach, but requires a certain amount
 of background to define. In the end, the overall constant (or
 function) turns out to be  irrelevant to the LDP.

There is a second approach which is less obvious but is simpler
and asymptotically equivalent as $N \to \infty$. It arises from an
auxiliary Hermitian line bundle $\lcal_{N +g} \to X$ of degree $N
+ g$, whose   space $H^0(X, \lcal_{N + g})$ of holomorphic
sections we call the {\it large vector space} (see \S \ref{LVS}
for background). Since  $\dim H^0(X, \lcal_{N + g}) = N + 1$ and
 $\dim X^{(N)} = \dim \PP H^0(X, \lcal_{N + g}) $, we regard  $\PP H^0(X, \lcal_{N + g})$ as a kind of linear model
for $X^{(N)}$. As discussed below, there exists a finite branched
holomorphic cover $X^{(N)} \to  \PP H^0(X, \lcal_{N + g})$ (away
from a certain Wirtinger subvariety) which restricts on each fiber
of the Abel-Jacobi map to an embedding of projective spaces.
 The choice of $\lcal_{N + g}$ is arbitrary but may be uniquely fixed
by picking a base point $P_0$ and defining $$\lcal_{N + g} =
\ocal((N + g) P_0)$$ to be the line bundle of the divisor $(N + g)
P_0$ (see \S \ref{BACKGROUND} for the notation). We recall (from
\cite{Gu2}, p. 107) that for each $\xi \in \mbox{Pic}^N$,  there
is an embedding
\begin{equation} \label{CANONICALIOTA} \iota_{\xi,\lcal_{N + g}} : H^0(X, \xi) \to \{s \in H^0(X,
\lcal_{N + g}): \dcal(s) \geq A_{\lcal_{N + g}, g}(\xi) \}
\end{equation}  where $A_{\lcal_{N + g}, g}: \mbox{Pic}^{N +g}  \to
X^{(g)}$ is the Abel-Jacobi map associated to $\lcal_{N + g}$ (see
\S \ref{BACKGROUND}).

We use this to show that $\ecal_N$ is analytically equivalent
(away from a certain Wirtinger subvariety) to the holomorphic
vector bundle $\tilde{\ecal}_N \to X^{(g)}$ with fiber
\begin{equation}\label{TILDEECALDEF} \tilde{\ecal}_{P_1 + \cdots + P_g} : = \{s \in H^0(X,
\lcal_{N + g}): \dcal(s) \geq P_1 + \cdots + P_g\}. \end{equation}
To be more precise, the Abel-Jacobi map $X^{(g)} \to \mbox{Jac}(X)
\simeq \mbox{Pic}^N$ is a branched cover, with branch locus along
a  Wirtinger subvariety (cf. \S \ref{AJSECT}). It turns out to be
more convenient to use $\tilde{\ecal}_N$ rather than $\PP \ecal_N$
since there is a simpler map from $\PP \tilde{\ecal}_N \to \PP
H^0(X, \lcal_{N + g})$. From a probability point of view, this
model is equivalent to that of Definition \ref{FSHAARDEF}, but for
emphasis we give it a separate definition.

\begin{maindefin}  \label{FSTILDEDEF}   Given a family of Hermitian inner products
$G_N(\xi)$ for $\xi \in \mbox{Pic}^N$, and a volume form $d
\sigma$ on $X^{(g)}$,
  the associated  {\it Fubini-Study-$X^{(g)}$} measure
 $d\tau_N^{FSg}$ on $\PP \tilde{\ecal}^N $ is the fiber-bundle product measure,
  \begin{equation} \label{gammadeftilde} \int_{\PP \tilde{\ecal}^N} F(s) d\tau_N^{FSg} (s)
  := \int_{X^{(g)}}\; \{\int_{\PP H^0(X, \xi)} F(s) dV^{FS}_{G_N(\xi)} \} \;d\sigma(P_1 + \cdots + P_g), \end{equation}
where $dV^{FS}_{G_N(\xi)}$ are the fiber Fubini-Study volume
  forms defined by the Hermitian inner products on $H^0(X, \xi)$
  induced by $G_N(\xi)$. We also denote by $\hat{\tau}_N$
  the analogous construction for Gaussian-Haar measures.
 \end{maindefin}

We now use this set-up to define a family $G_N(\omega, \nu)$ of
Hermitian inner products on $H^0(X, \xi)$.  We first  choose a
Hermitian metric $h_0$ on $\ocal(P_0)$ with curvature $(1,1)$-form
$\omega$. This  induces Hermitian metrics $h_0^{N + g}$ on all the
line bundles $\lcal_{N + g}$. Together with the positive measure
$\nu$ on $X$, we obtain the inner product $G_{N + g}(h, \nu)$
(\ref{DEFGN})  on $H^0(X, \lcal_{N + g})$ and  by restriction, on
the  subspaces $\iota_{\xi,\lcal_{N + g}} H^0(X, \xi)$. Since the
embeddings are isomorphisms, we obtain inner products on $H^0(X,
\xi)$ as $\xi $ varies over $\mbox{Pic}^N$ (see \S \ref{EMBEDIP}).
We refer to this family of inner products as an $\omega-${\it
admissible} family of Hermitian inner products (see \S
\ref{ADMHERM}). We then endow $H^0(X, \lcal_{N + g})$ and the
fibers $H^0(X, \xi), \xi \in \mbox{Pic}^N$ with the Gaussian
measures $\gamma_{G_{N + g}(h, \nu)}, $ resp. $\gamma_{G_N(h,
\nu)}$, associated to the inner products. The Gaussian measures on
$H^0(X, \xi)$ are (up to identification by the embedding)  the
conditional Gaussian measures of $\gamma_{G_{N + g}(h, \nu)}.$

 We then  projectivize to obtain Fubini-Study metrics associated
 to the inner products on the fibers $\PP H^0(X, \xi)$ and on $\PP
 (X, \lcal_{N + g})$. The induced maps (cf.
 (\ref{CANONICALIOTA})),
 \begin{equation} \label{psi} \left\{ \begin{array}{l} \psi_{\lcal_{N + g}} : \PP \tilde{\ecal}_N \to \PP H^0(X,
\lcal_{N + g}), \\ \\
\psi_{\lcal_{N + g}}[s] = [\iota_{\xi, \lcal_{N + g}}(s)],
\;\;\mbox{where} \;\; s \in H^0(X, \xi)
\end{array} \right.
\end{equation} are by definition isometric
 along the subspaces $\PP H^0(X, \xi)$.  In Proposition \ref{CANONICAL}, we show that
this map is a branched covering map of degree ${N \choose g}$ away
from the exceptional (Wirtinger) locus.  Further $\psi_{\lcal_{N +
g}}$ pulls back the Hermitian hyperplane bundle of $\PP H^0(X,
\lcal_{N + g})$ (with its Fubini-Study metric) to the natural
hyperplane bundle over $\PP \tilde{\ecal}_N$ (with the Hermitian
metric above).

We use this construction to obtain our third ensemble:

  \begin{maindefin} \label{LARGEVS} We define the projective linear ensemble
 $(\PP \ecal_N, d\lambda_{PL})$ by
  $$d \lambda_{PL} = \frac{1}{{N \choose g}} \psi_{\lcal_{N +g}}^* dV^{FS}_{G_{N + g}(h_0, \nu))} $$
  where $dV^{FS}_{G_{N + g}(h_0, \nu))}$ is the Fubini-Study volume form on $\PP H^0(X,
  \lcal_{N + g})$ and $\psi_{\lcal_{N + g}}: \PP \tilde{\ecal}_N \to \PP H^0(X,
  \lcal_{N + g})$ is the map (\ref{psi}).
  \end{maindefin}

The Fubini-Study-Haar measure and the projective linear measure
are closely related since they are top exterior powers of
differential forms which agree along the fibers of the projection
$\pcal_N \to \mbox{Pic}^N$ to a base of fixed dimension $g$. They
are compared in detail in \S \ref{PLVSFSH}. A technical
complication in this approach is that the pulled-back Fubini-Study
probability measure is only a semi-positive volume form on
$X^{(N)}$ which vanishes along the branch locus.

\subsection{Joint probability current}

 Given
a probability measure $\tau_N$ on $\ecal^N$ (or  $\PP \ecal^N$)
the corresponding JPC is defined as follows:

\begin{maindefin} \label{JPCDEF}
\begin{equation} \vec K^N(\zeta_1, \dots, \zeta_N)
: =  \dcal_* \tau_N \; \in \; \mcal(X^{(N)}) \end{equation}
 of the measure $\tau_N$ under the divisor map $ \dcal$ (\ref{ZEROSMAPg}).
 \end{maindefin}
Since we are dealing with a number of measures $\tau_N$, we often
subscript $\vec K^N$ to indicate the associated measure $\tau_N$.

  The main task of this article is to express the JPC of the
ensembles in the previous section in terms of empirical measures
of zeros (\ref{ZN}), and to  extract an approximate rate function
from it. Less formally,  we need to express the JPC in terms of
`zeros coordinates', i.e.  the natural coordinates $\{\zeta_1,
\dots, \zeta_N\}$ of $X^{N}$.  Elementary symmetric functions of
the $\zeta_j$ define local coordinates on  $X^{(N)}$.
 What makes the expression difficult is that the underlying probability
measures on $H^0(X, \xi)$ and $\ecal^N$ are expressed in terms of
coefficients relative to a basis of sections. Thus we need to
`change variables' from coefficients to zeros. As mentioned above,
we are essentially changing from variables adapted to the
Abel-Jacobi fibration structure of $X^{(N)}$ to those adapted to
its product structure.

In the genus zero case, the inverse map from zeros to coefficients
is given by the Newton-Vieta formula,
\begin{equation} \label{VIETA}   \prod_{j = 1}^N (z - \zeta_j) =
 \sum_{k = 0}^N (-1)^k e_{N - k} (\zeta_1, \dots, \zeta_N)\;\;
 z^k,
\end{equation}
where  $e_j = \sum_{1 \leq p_1 < \cdots < p_j \leq N} z_{p_1}
\cdots z_{p_j}$ are the elementary symmetric functions. The
Jacobian of the map from coefficients to zeros is thus
$|\Delta(\zeta)|^2$ where
\begin{equation}\label{DELTAPROD}  \Delta(\zeta) = \det (\zeta_j^k) = \prod_{j < k}
(\zeta_j - \zeta_k) \end{equation} is the Vandermonde determinant.
In effect, we must generalize both formulae to higher genus.
Generalizing the Vieta formula is the subject of \S
\ref{VIETASECT}. Generalizing the Vandermonde determinant formula
is the subject of \S \ref{VDSECT}.

In particular, we need the analogue of $z - w$ in higher genus. As
is well-known, the analogue is the `prime form' $E(z,w)$ of $X$
\cite{M,F,ABMNV,R}. It is a holomorphic section of a certain line
bundle over $X \times X$ with divisor equal to $D, $ the diagonal
in $X \times X$. $E(z,w)$ has  a well-known expression in terms of
theta-functions, but for our purposes it may be characterized  as
follows: the space  $H^0(X \times X, \ocal(D)) $ of holomorphic
sections of $\ocal(D)$ is  one-dimensional \cite{R}. The bundle of
which $E(z,w)$ is a section is isomorphic to $\ocal(D) \to X
\times X$. Hence to specify $E(z,w)$ we need to specify an element
of $H^0(X \times X, \ocal(D))$. Roughly speaking, we do this by
specifying that $E(z,w) \sim z - w$ near the diagonal. This
depends on a choice of local coordinates, which we define by
uniformizing, i.e. by expressing $X = \tilde{X}/\Gamma$ where
$\tilde{X}$ is the universal holomorphic cover and $\Gamma$ is the
deck transformation group. We then select a fundamental domain
$\fcal$ for $\Gamma$ in $\tilde{X}$ and use the global coordinates
on $\tilde{X}$ as local coordinates on $X$. For more details, we
refer to \S \ref{PRIMEFORM}.

 As above, we denote by $ \ocal(P)$ the {\it point
line bundle}, i.e. the line bundle associated to the divisor
$\{P\}$. For $g \geq 1$, the space $H^0(X, \ocal(P))$ is one
(complex) dimensional, and may be identified  as the pullback
under the embedding $\iota_P: X \to X \times X, \iota_P(z) = (z,
P)$ of the line bundle $\ocal(D) \to X \times X$. The section
${\bf 1}_{\ocal(P)}(z)$ is defined to be $\iota_P^* E(z,\cdot) =
E(z, P). $ Using products of  the prime form, we define `canonical
sections' $S_{\zeta_1, \dots, \zeta_N}$  of high degree line
bundles with prescribed zeros (see \ref{SZETA}).

 We also need to define a certain Bergman kernel for the
 orthogonal projection onto $H^0(X, \lcal_{N + g})$ with respect to the inner
product $ G_{N + g} (h_0, \nu)$. It is not quite the standard one,
but rather projects onto the codimension one subspace
$H_{P_0}^0(X, \lcal_{N + g})$ of sections vanishing at the base
point $P_0$. This kind of Bergman kernel is studied in \cite{SZZ}
and is called there the {\it conditional \szego kernel}. Thus, we
put
\begin{equation} \label{BN} B_N = \mbox{ the Bergman kernel
on}\;\; H^0_{P_0}(X, \lcal_{N + g}) \;\; \mbox{w.r.t. the inner
product}\; G_{N + g} (h_0, \nu).
\end{equation}  As will be shown below,  $B_N$ is almost the same as  the Bergman kernel for
a slightly modified line bundle $E_{N + g - 1}$ of degree $N + g
-1$ (see (\ref{EDEF})). Their determinants differ by products of
${\bf 1}_{P_0}$.

Before stating the result, we emphasize a further notational
convention used throughout: when discussing sections of a line
bundle $L$, we fix a local frame $e_L$ and express a section by $s
= f e_L$ where $f$ is a locally defined function. As discussed in
\S \ref{BERGMANSECTION}, the orthogonal projection onto $H^0(X,
L)$ is a section of $L \otimes L^{-1}$ (the \szego kernel). When
we express it in terms of $e_L \otimes \bar{e}_L$, the local
function is the Bergman kernel $B_L$ above. In a similar way, we
use $E(z, P)$ to be the local expression for ${\bf
1}_{\ocal(P)}(z)$ relative to the canonical frame for $\ocal(P)$.
The former is standard, but the latter is not a standard
convention.

The following Theorem expresses the JPC in terms of coordinates on
$X^{(N)}$ or more precisely its further lift to $X^N$. The
coordinates are  products of uniformizing coordinates on the
factors. As mentioned above, and as  will be clarified during the
proof, the local expressions are understood to be given on the
complement of the exceptional Wirtinger loci.

\begin{maintheo} \label{JPDHG} (I)
 The JPC  of zeros in the projective  linear  Fubini-Study probability
space $\PP \ecal_N$ with measure in Definition \ref{LARGEVS} is
given in local coordinates  by
$$\begin{array}{lll} (I)\;\; \vec K^N_{PL}(\zeta_1, \dots, \zeta_N)  & = & \frac{1}{Z_N(h)} \frac{\fcal_N(\zeta_1, \dots, \zeta_N, P_0) \prod_{k = 1}^N \left|
\; \prod_{j = 1}^g E(P_j, \zeta_k)   \cdot \prod_{j: k \not= j}
  E(\zeta_j, \zeta_k)\right|^2}{ \; \det
  \left(B_N (\zeta_j, \zeta_k) \right)_{j, k = 1}^N} \prod_{j = 1}^N d\zeta_j \wedge d \bar{\zeta}_j  \\ && \\
  && \times \left(\int_X \left| \prod_{j = 1}^g E (P_j, z) \right|_{h_g}^2 \cdot \left| \prod_{j = 1}^N E (\zeta_j, z)\right|^2_{h^N} d\nu(z) \right)^{- N - 1}
  \end{array},$$
  where $ P_1 + \cdots + P_g = A_{\lcal_{N + g}} (\zeta_1 + \cdots + \zeta_N)$.
(defined in \S \ref{AJSECT}) and $Z_N(h)$ is a normalizing
constant so that $\vec{K}^N$ has mass one. Here, $\fcal_N(\zeta_1,
\dots, \zeta_N, P_0)$ is defined in (\ref{FCALNDEF}).

Furthermore,
$$
\begin{array}{lll} (II)\;\; \vec K_{PL}^N(\zeta_1, \dots, \zeta_N) \label{eq-030209d} & = &
\frac{1}{\hat{Z}_N(h)} \frac{\exp \left( \sum_{i < j}
G_{\omega}(\zeta_i, \zeta_j) \right)  \left(\prod_{j = 1}^N
|\rho_{\omega}(\zeta_j)|^2 d^2 \zeta_j \right)}{\left(\int_{X}
e^{N \int_{X} G_{\omega}(z,w) d\mu_{\zeta}} e^{ \int_{X}
G_{\omega}(z,w) d\mu_{P(\zeta)}} d\nu(z) \right)^{N+1}}
F_N(\zeta_1, \dots, \zeta_N),.
\end{array}$$
where $d\mu_{P(\zeta)} = \sum_{j = 1}^g \delta_{P_j(\zeta)}, $
where $\rho_{\omega}$ is defined in (\ref{rho}),  and where $F_N$
is  defined in (\ref{FNDEF}) \S \ref{JPCCALCULATIONII}.
\end{maintheo}

Note that $\prod_{j = 1}^N d^2 \zeta_j$ is not well-defined on
$X^{(N)}$ but only on the Cartesian product $X^N$; however $\vec
K_{PL}^N$ is  symmetric under permutations and descend to the
quotient. We say more about this in \S \ref{XN} (see \cite{A} for
background on holomorphic forms on symmetric powers). To explain
the first formula, we observe that the numerator and denominator
involve Hermitian inner products of sections. When expressed in
terms of the same local frame, the frames and metric factors
cancel.  This will be explained in a Remark after the statement of
Proposition \ref{BASICFORMULA}. The main difference between
expressions (I) and (II) is that the difficult Bergman determinant
$\det
  \left(B_N (\zeta_j, \zeta_k) \right)_{j, k = 1}^n$ in (I) has
  been expressed in terms of products of norms of the Green's
  function (or the prime form). It is an instance of the theme mentioned above
  of relating the fiber bundle structure of $X^{(N)} $ to
 its product structure.  The identity relating the Bergman
  determinant and products of the prime form is known as the bosonization
  formula on Riemann surfaces \cite{ABMNV,VV}. This complicated identity brings in
  other special factors which we call $F_N$. When we take
  $\frac{1}{N^2} \log$ this factor will evaporate in the limit.

Since $\vec{K}^N$ is the pull-back of a smooth form under a smooth
map (defined on the complement of a hypersurface), the zeros of
the denominator must be cancelled by those of the numerator. This
is discussed after the proof in \S \ref{BERGZEROS}.

The corresponding formula in the genus zero case  $\CP^1$ is given
in Proposition 3 of \cite{ZZ}:
\begin{eqnarray}
    \label{eq-030209b}
    \vec K^N(\zeta_1, \dots, \zeta_N) & = & \frac{1}{Z_N(h)}
\frac{|\Delta(\zeta_1, \dots, \zeta_N)|^2  d^2 \zeta_1 \cdots d^2
\zeta_N}{\left(\int_{\CP^1} \prod_{j = 1}^N |(z - \zeta_j)|^2 e^{-
N \phi(z)}  d\nu(z) \right)^{N+1}}.
\end{eqnarray}
Here, we abbreviate $d^2 \zeta = d \zeta \wedge d \bar{\zeta}$.
Thus, the expression
\begin{equation}\label{DELTAG} \Delta_g(\zeta_1, \dots, \zeta_N):=
\frac{\prod_{j = 1}^g \prod_{k = 1}^N  E(P_j, \zeta_k) \cdot
\prod_{j: k \not= j}
  E(\zeta_j, \zeta_k)}{ \sqrt{\det
  \left(B_N (\zeta_j, \zeta_k) \right)_{j, k = 1}^n}} \end{equation} is
  a higher genus generalization of the Vandermonde determinant in the sense that it
  is the Jacobian of the change of variables from coefficients to zeros.
  The same Jacobian arises in bosonization formulae(see e.g.  (4.15) in \cite{ABMNV}).

For the purposes of this article, Theorem \ref{JPDHG} is mainly
useful to obtain a similar formula for the JPC $\vec K^N_{FSH}$ of
the Fubini-Study-fiber ensembles of Definitions \ref{FSHAARDEF}
and \ref{FSTILDEDEF}. These are the ensembles for which we prove
Theorem \ref{LD}. We obtain formulae for $\vec K^N_{FSH}$ by
relating them explicitly to $\vec K^N_{PL}$ as given  in Theorem
\ref{JPDHG}. The formula for $\vec K^N_{FSH}$ is too complicated
to state in the introduction, and requires many considerations
from Abel-Jacobi theory. It is stated and proved in Theorem
\ref{MAINFSHPROP}.  We summarize it as follows:

\begin{maintheo} \label{JPDFSH}
 The JPC  of zeros in the    Fubini-Study-fiber probability
spaces $\PP \ecal_N$ (resp. $\PP \tilde{\ecal}_N$) with measure in
Definition \ref{FSHAARDEF} (resp. Definition \ref{FSTILDEDEF})
pulls back to  $X^N$  as the form,
$$\vec K^N_{FSH} (\zeta_1, \dots, \zeta_N) = \frac{\jcal_N (\zeta_1, \dots, \zeta_N)}{Z_N(\omega)}
    \frac{\exp \left(\frac{1}{2} \sum_{i \not=
j} G_{\omega}(\zeta_i, \zeta_j) \right) \prod_{j = 1}^N d^2
\zeta_j}{\left(\int_{X} e^{ N \int_{X} G_{\omega}(z,w)
d\mu_{\zeta}(w)} e^{ \int_{X} G_{\omega}(z,w) d\mu_{P(\zeta)}(w)}
d\nu(z) \right)^{N+1}}, $$ where $\jcal_N$ is defined in Theorem
\ref{MAINFSHPROP}.

\end{maintheo}

The main idea is that the  volume forms $\vec K^N_{PL}$ and $\vec
K^N_{FSH}$  on $\PP \ecal_N$ only differ in `horizontal'
directions with respect to connections on $X^{(N)} \to
\mbox{Jac}(X)$ resp. $\tilde{X}^{(N)} \to X^{(g)}$, which has a
fixed dimension $g$. Hence, they are equivalent up to small errors
as $N \to \infty$.

\subsection{\label{APPROX} Approximate rate function}

We use  Theorems \ref{JPDHG} and \ref{JPDFSH}  to derive an
approximate rate functional for the large deviations principle. We
need to introduce the following functionals:

\begin{maindefin}  \label{APPROXRF}  Let $\zeta \in X^{(N)}$ and let $\mu_{\zeta}$ be as in (\ref{ZN}).
Also let $\mu_{P(\zeta)}$ be the sum of point masses at the $g$
points in the image of $\zeta$ under the Abel-Jacbobi map.  Let
$D=\{(z,z): z\in X\}$ be the diagonal. Put:
$$ \left\{ \begin{array}{l} \ecal^{\omega}_N(\mu_{\zeta}) = \int_{X\times X \backslash D}
G_{\omega}(z,w) d\mu_{\zeta}(z) d\mu_{\zeta}(w), \\ \\
 J_N^{\omega, \nu} (\mu_{\zeta}) = \log ||e^{U_{\omega}^{\mu_{\zeta}}}  e^{ \int_{X}
G_{\omega}(z,w) d\mu_{P(\zeta)}(w)} ||_{L^N(\nu)}
 \end{array}
\right.$$ Then put $$ I_N^{\omega, \nu} (\mu_{\zeta}) =
-\frac{1}{2} \ecal^{\omega}_N(\mu_{\zeta}) + \frac{N+1}{N}
J_N^{\omega, \nu} (\mu_{\zeta}). $$
\end{maindefin}

We then have,

\begin{mainprop} \label{APPROXRATEa} Write $\vec K_{FSH}^N(\zeta_1, \dots, \zeta_N) =
D^N(\zeta_1, \dots, \zeta_N) \prod_{j = 1}^N d^2 \zeta_j$, and
similarly define $D^N_{FSH}$.  Then,
$$D^N(\zeta_1, \dots, \zeta_N) = \frac{1}{\hat{Z}_N(h)} e^{- N^2  \left(I_N^{\omega, \nu} (\mu_{\zeta}) +
O(\frac{1}{N})\right)}.
$$

\end{mainprop}

We observe that $D^N$ is precisely the coefficient relating
$\vec{K}^N$ (defined using the fiber bundle structure of $\PP
\ecal_N$) to the local Lebesgue product measure on   $X^{(N)}$. We
could use any product measure $\omega_0 \boxtimes \cdots \boxtimes
\omega_0$ here, but it is sufficient to use $\prod_j d^2 \zeta_j$
in local uniformizing coordinates. The proof of Theorem \ref{LD}
from Proposition \ref{APPROXRATEa} uses almost the same analysis
as in \cite{ZZ}. Hence the emphasis of this article is on the
proof of Theorems \ref{JPDHG} and \ref{JPDFSH} and of Proposition
\ref{APPROXRATEa}.

Finally, we thank R. Wentworth and S. Wolpert for helpful
conversations on the prime form and on  bosonization formulae.

\section{\label{BACKGROUND} Background }

The main purpose of this section is to  review  the Abel-Jacobi
theory we need. To the extent possible, we follow the notation of
\cite{ACGH,ABMNV,Gu1,Gu2}. We also briefly review some definitions
and notation regarding Gaussian and Fubini-Study random
holomorphic sections. We use throughout the same notation and
terminology as in \cite{ZZ}, and when the definitions are
essentially the same as in genus zero, we refer the reader to that
article.

We denote by $X$ a compact smooth Riemann surface of genus $g$.
Througout we fix a base point $P_0 \in X$.  By the uniformization
theorem, it may be expressed as $\tilde{X} \backslash \Gamma$
where $\tilde{X}$ is the universal conformal cover ($= \CP^1$ if
$g = 0$, $= \C$ if $g = 1$ and $= \hcal$ (the upper half plane) if
$g \geq 2$). Here, $\Gamma$ is the deck transformation group or
fundamental group, realized as conformal automorphisms of
$\tilde{X}$. We also denote by $A_1, \dots, A_g; B_1, \dots, B_g$
generators of $\Gamma$ defined as in \cite{Gu2} (page 4). They
depend on a choice of base point $P_0$ and a marking of $X$; we
refer to \cite{Gu2} for the background.

We further denote by $\omega_1, \dots, \omega_g$ a basis for the
holomorphic differential one forms of $X$. We assume the $A, B$
cycles and basis are canonical, i.e. $\int_{A_i} \omega_j =
\delta_{i j}$. For any holomorphic differential, we denote by
$w(z) = \int_{z_0}^z \omega$ its Abelian integral, a holomorphic
function on $\tilde{X}$ with $w(z_0) = 0$.  In the case of the
basis differentials $\omega_k$ the Abelian integrals are denoted
$w_k$.

We denote line bundles by $L, \xi$ or by divisors of holomorphic
sections $s \in H^0(X, \xi)$. The Chern class of $L$ is denoted as
usual by $c_1(L)$. We also put $h^0(X, L) = \dim H^0(X, L)$.

\subsection{\label{POINT} Point  line bundles over Riemann surfaces of genus $g > 0$}

 Given $P \in
X$, there exists a unique holomorphic line bundle $\ocal(P)$ with
the properties that $\dim H^0(X, \ocal(P)) = 1$ and $c_1(\ocal(P))
= 1$ and so that each non-zero section vanishes at (and only at)
the point $P$. These are the line bundles  defined by a point
divisor $\{P\}$, defined by an atlas of two charts, $U_0 = X
\backslash \{P\}$ and $U_1 = $ a small disc around $P$, and the
transition function $g_{0 1} = z - P$.

The line bundle $\ocal(P)$ has a canonical section   ${\bf
1}_{\ocal(P)}$ corresponding to the meromorphic function $1$ under
the correspondence between meromorphic functions and holomorphic
sections of $H^0(X, \ocal(P))$. It  is the section locally
represented by $1$ on $X \backslash \{P\}$.  Our notation follows
\cite{ABMNV}; the same section is called the `constant section'
and is denote by ${\bf 1}$ in \cite{Fal}.

These sections are canonical once we fix the coordinates and
atlas, but that choice itself is non-canonical. The choices can be
made consistently as $P$ varies by working on $X \times X$. As in
the introduction, we define local coordinates by uniformizing. We
let $\{U_{\alpha}, z_{\alpha}\}$ denote the corresponding atlas of
$X$.  We then define an atlas of local trivializations of
$\ocal(D)$ taking the cover $\{U_{\alpha} \times U_{\alpha},
U_0\}$ of $X \times X$ where $U_0 = X \times X \backslash D$. Then
the local holomorphic functions $f_{\alpha}(Q, P) = z_{\alpha}(P)
- z_{\alpha}(Q)$ are local defining functions of $D$ in
$U_{\alpha} \times U_{\alpha}$. Hence their ratios
$\frac{f_{\alpha}}{f_{\beta}}$ give transition functions for
$\ocal(D)$ and $\{f_{\alpha}\}$ is a section. If we fix $P$ in the
slice $X \times \{P\}$ we obtain a family of sections
$f_{\alpha}(z, P)$ of the family of line bundles $\ocal(P)$ which
pull back to ${\bf 1}_{\ocal(P)}$. We explain the relation to
theta functions in \S \ref{PRIMEFORM}.

\subsection{\label{XN} Configuration spaces}

 The  $r$-rold symmetric product $X^{(r)}$
of a Riemann surface $X$ is the quotient of the Cartesian product
$X^r$ by the action of the symmetric group $S_r$. The action has
non-trivial isotropy along the large diagonals $\zeta_j = \zeta_k$
($j \not= k$). However,  $X^{(r)}$  is a complex analytic manifold
of dimension $r$ and the natural projection
$$p_r: X^r \to X^{(r)}$$
is a complex analytic $r!$ sheeted branched covering map with
branch locus on the large diagonals. We briefly review this fact
and the associated local  coordinates on $X^{(r)}$; the
degeneracies of $p_r$ will be important later in explaining
cancellations of zeros and poles in $\vec K^N_{FSH}.$ We refer to
(\cite{Gu2}, Theorem 9; or \cite{GH}, p. 236).

Let $D = p_1 + \cdots + p_r \in X^{(r)}$. Let $U_i$ be a
neighborhood of $p_i$ and $z_i$ a local coordinate in $U_i$. We
assume that $U_i \not= U_j$ if $p_i \not= p_j$ and $U_i = U_j, z_i
= z_j$ if $p_i = p_j$. Let $\sigma_1, \dots, \sigma_r$ be the
elementary symmetric functions. Then
$$q_1 + \cdots + q_r \to (\sigma_1 \{z_i(q_i)\}, \dots, \sigma_r
\{z_i(q_i)\}) $$ is a local coordinate system on $p_r(U_1 \times
\cdots \times U_r)$ defining $X^{(r)}$ as a complex manifold. Away
from the large diagonals, $p_r$ is a covering map and we can use
$(z_1(p_1), \dots, z_r(p_r))$ as local coordinates.

A volume form on $X^{(r)}$ pulls back under $p_r$ to a smooth form
on $X^r$ which may be written in terms of local coordinate volume
form $$d \sigma_1 \wedge \cdots \wedge d \sigma_r = \Delta(\zeta)
d \zeta_1 \wedge \cdots \wedge d \zeta_r,$$ where $\Delta$ is the
Vandermonde determinant. As this shows, any smooth volume form on
$X^{(r)}$ lifts to a semi-volume form with zeros on the branch
locus of $p_r$. For more on holomorphic forms on $X^{(r)}$ we
refer to \cite{A}.

We also recall that $X^{(g)}$ is a projective \kahler manifold. To
see this, let $L \to X$ be an ample line bundle and let $\tilde{L}
= L_{z_1} \boxtimes \cdots \boxtimes L_{z_g}$ be the exterior
tensor bundle on $X^r$.  Then $\bigotimes_{\tau \in S_r} \tau^*
\tilde{L}$ defines an ample line bundle on $X^r$ which is
invariant under permutations and therefore descends to $X^{(r)}$.
Positively curved metrics on this bundle give  a supply of \kahler
forms on $X^{(r)}$ which can be used to specify the volume forms
on $X^{(g)}$ in Definition \ref{FSTILDEDEF}.

\subsection{Products of point line bundles and canonical sections}

Point bundles can be used to generate all other line bundles. We
denote by $\ocal(P_1 + \cdots + P_n) = \ocal(P_1) \otimes \cdots
\otimes \ocal(P_n)$ the tensor product of the point bundles. In
general, given two divisors $D, D'$, $\ocal(D) \ocal(D') = \ocal(D
+ D')$.

Products of point line bundles have a canonical section, depending
only on the local trivializing charts and coordinates used to
define $\ocal(P)$.

\begin{maindefin} \label{CANONICALSECTION} The canonical section of $\ocal(\zeta_1 + \cdots + \zeta_N)$
is defined by
$${\bf 1}_{\zeta_1 + \cdots + \zeta_N}(z) : =    \prod_{j = 1}^N
{\bf 1}_{\ocal(\zeta_j)}(z). $$
\end{maindefin}

By Riemann-Roch,  for $N \geq 2g - 1$ we have
\begin{equation} \label{DIMENSION} \dim H^0(X, \ocal(P_1 + \cdots + P_N)) = N + 1 - g.
\end{equation} This is the source of the difference between $g =
0$ and $g > 0$: the number of zeros of sections is  greater by $g$
than the dimension of the projective  space of sections.
 Indeed, the configurations of zeros of sections $s \in H^0(X, \xi)$ form the fiber of
 the Abel-Jacobi map  $ A_N: X^{(N)} \to \mbox{Jac}(X)$. This is the next object to
 review.

 \subsection{Jacobian and Picard varieties}

The Jacobian variety $\mbox{Jac}(X)$ may be identified with the
compact complex $g$-dimensional torus $\C^g/\pcal$ of unitary
characters $\chi$ of the fundamental group $\pi_1(X)$ or
equivalently of $H^1(X, \Z)$. Here, $\pcal$ is the period lattice
of $X$.

$\mbox{Pic}^0 = \mbox{Div}_0/\mbox{PDiv}$ is the space of divisor
classes  of degree zero, i.e. divisors of degree 0 modulo
principal divisors (divisors of meromorphic functions). We denote
the divisor class of $D$ by $[D]$. Thus, $\mbox{Pic}^0 $ is a
compact complex torus of dimension $g$. For general $r \in \Z$,
the Picard variety $\mbox{Pic}^r$ is the space of holomorphic line
bundles of Chern class $r$.
 $\mbox{Pic}^r$ is a
homogeneous space for $\mbox{Pic}^0$ under tensor product.  If we
  fix a base point $P_0$,  we obtain an identification  $\mbox{Pic}^r(X) \simeq \mbox{Pic}^0(X)$ by
 $\xi \in \mbox{Pic}^r(X) \to \xi \ocal(- r P_0).$
 We recall that the map from  divisors to line
bundles by
\begin{equation}\label{LBDIV}  \dcal = \sum_j \nu_j p_j \in X^{(r)} \to \ocal(\dcal) = \prod_j\;
 \ocal(p_j)^{\nu_j} \in Pic^{(r)} \end{equation} is surjective and that
principal divisors correspond to holomorphically  trivial line
bundles. Hence,  $\mbox{Pic}^0$ can be identified the space of
degree 0 line bundles modulo equivalence, or with flat line
bundles induced by unitary characters of $\pi_1(X)$, i.e.
$\mbox{Pic}^0 \simeq \mbox{Jac}(X)$. Specifically, the map is
defined by
\begin{equation}\label{JACOBI}  \begin{array}{l}  \ocal(p_1) \cdots  \ocal(p_r)  \ocal(q_1)^{-1} \cdots  \ocal(q_r)^{-1} \to \chi, \;\mbox{with}\\ \\
 \chi(A_k) = 1, \;\; \chi(B_k) = \exp 2 \pi i \sum_{j = 1}^r
\int_{q_j}^{p_j} \omega_k, \end{array}\end{equation} where the
path of integration  is a simple arc $\delta_j$ from $q_j \to p_j$
(see \cite{Gu2}, page 24).

\subsection{\label{AJSECT}  Abel-Jacobi maps, Picard varieties and Wirtinger varieties}

 As above we fix a basepoint $P_0$ and  the effective divisor  $g P_0$ of degree
 $g$. The Abel-Jacobi (or  Abelian sums) map $A_r: X^{(r)} \to
\mbox{Jac}(X)$ is defined by \begin{equation} \label{ABEL}
A_r(\zeta_1, \dots, \zeta_r) = \sum_{j = 1}^r \int_{P_0}^{\zeta_j}
 \left(\begin{array}{l} \omega_1 \\ \cdots \\ \omega_g
 \end{array}\right) \in \C^g/\pcal.
 \end{equation}
 Abel's theorem
states that the image is $0$ if $\sum_{j = 1}^r \zeta_j - r P_0$
is a principal divisor (the divisor of a meromorphic function).

When $r = g$ the Abelian sums map is an analytic isomorphism away
from a (Wirtinger) hypersurface. It may be re-formulated in terms
of Picard varieties as follows:
 The
map $P_1 + \cdots + P_g \to \ocal(g P_0 - (P_1 + \cdots P_g))$ is
an analytic diffeomorphism
$$A_{g}: X^{(g)} \simeq Pic^0(X)$$ outside of a codimension one
subvariety. That is,  any $\xi \in Pic^0(C)$ may be represented in
the form
$$\xi = \ocal(P_1 + \cdots + P_g) \ocal( g P_0)^{-1}, $$
for some $P_1 + \cdots + P_g$. The representation is unique when
$\dim H^0(X, \ocal(P_1 + \dots + P_g) = 1$, and this holds for
generic $P_1 + \cdots + P_g$ (\cite{Gu1} (pages 116-120)).

The exceptional points lie on Wirtinger varieties, which play an
important role in the formulae for the JPC. There are several
related definitions accordingly as we regard them as subsets of
$\mbox{Jac}(X)$ or $\mbox{Pic}^N(X)$ or $X^{(g)}$.  The Wirtinger
varieties $W_r \subset \mbox{Jac}(X)$ are defined by $W_r = W_1 +
\cdots + W_1 \subset \mbox{Jac}(X)$ where $W_1$ is the image of
$X$ under $A_1$. They depend on the choice of base point $P_0$,
and one has $\dim W_r = r$ for $1 \leq r \leq g$ and $W_g =
\mbox{Jac}(X)$. We also recall {\it Jacobi's inversion formula}:
The Abelian sums map $A_g: X^{(g)} \to \mbox{Jac X}$ is a
surjective holomorphic map with a codimension one singular set of
effective divisors $P_1 + \cdots + P_g$ for which $\dim H^0(X,
\ocal(P_1 + \cdots + P_g)) = 2$; equivalently, there exists a
different element  $Q_1 + \cdots + Q_g$ for which $\ocal(P_1 +
\cdots + P_g) = \ocal(Q_1 + \cdots +Q_g).$ The image of this set
under $A_g$ is $W^1_g$, which has dimension $g - 2$ by the
Brill-Noether formula (see \cite{ACGH}). Viewing $W^1_g \subset
\mbox{Pic}^g(X)$ as the subset with $\dim H^0(X, \xi) = 2$, one
has a map $X \times W^1_g \to W_{g - 1}$ with $(P_0, \xi) \to \xi
\ocal(- P_0)$, whose sections correspond to the sections of $\xi$
vanishing at $P_0$.  Fixing the base point,  this map embeds
$W_g^1 \subset W_{g -1}$ as a codimension one subset.

 Fix a line bundle $\lcal_{N + g}$ of degree $g + N$.
Let $\zeta_1 +\dots + \zeta_N \in X^{(N)}$.  Then there exists a
point $P_1 + \cdots + P_g$ so that $ \ocal(\zeta_1) \otimes \cdots
\otimes \ocal(\zeta_N) \simeq \lcal_{N + g} \ocal(P_1 + \cdots +
P_g)^{-1} . $ The  point is unique  when $\lcal_{N + g} \otimes
\ocal( - (\zeta_1 + \cdots + \zeta_N)) \notin W_g^1, $ i.e. lies
outside of the codimension one Wirtinger subvariety $W^1_g$ of
line bundles $\xi$ of degree $g$ with $\dim H^0(X,\xi) \geq 2$.

We  take $\lcal_{N + g}$ to be $(N + g) P_0$ throughout, and
further define
\begin{equation} \label{WIRT} X_{N + g}^{(N)} = \{\zeta_1 +
\cdots + \zeta_N \in X^{(N)}: (N + g) P_0 - \zeta_1 +\dots +
\zeta_N  \in W^1_g\},  \end{equation} and
\begin{equation} \label{ALCAL} A_{\lcal_{N + g}}: X^{(N)} \backslash X_{N + g}^{(N)}  \to X^{(g)}, \;\;\; A_{\lcal_{N + g}}(\zeta_1 + \cdots
+ \zeta_N)  = P_1 + \cdots + P_g.
\end{equation}  where as above
$ (N + g) P_0 - (\zeta_1 + \cdots + \zeta_N) = P_1 + \cdots + P_g.
$ for a unique $P_1 + \cdots + P_g.$ In the notation of
\cite{ACGH}, $X^{(N)}_{N + g}$ is the inverse image of $W^1_g$
under the Abel sums maps $X^{(N)} \to \mbox{Jac}(X).$

\begin{maindefin}\label{CANONFACTOR}  Let $\xi  \in \mbox{Pic}^N(X)$ and let $s \in \PP
H^0(X, \xi)$. If $\ocal((N + g) P_0 - \dcal(s)) \notin W^1_g$ we
define the {\it canonical factor} associated to $[s]$ to be the
factor $\prod_{j = 1}^g {\bf 1}_{\ocal(P_j)}$ where $A_{\lcal_{N +
g}} (\dcal(s)) = P_1 + \cdots + P_g$.
\end{maindefin}


\subsection{\label{PICBUN} Picard bundles, vortex moduli space and Poincar\'e
bundles}

In  the introduction, we defined the Picard bundles (\ref{ECALVB})
and their projectivizations  $\PP \ecal_N$, the vortex moduli
spaces. It is well-known that $\PP \ecal_N$ is analytically
equivalent to $X^{(N)}$ for $N \geq 2g - 1$. For the proof we
refer to \cite{Gu2}, Corollary 2 to Theorem 10; and Corollary 2 to
Theorem 15. See also \cite{ACGH} and also \cite{Mat,Sch} for the
original proofs.

In \cite{FL} it is proved that the Picard bundles are negative
complex vector bundles, i.e. that $\ecal_N^*$ is an ample vector
bundle. It is equivalent that $\ocal_{\PP}(1) \to \PP (\ecal_N)$
is an ample line bundle.  See also \cite{ACGH}, Ch. VII.


\subsection{\label{THETASTUFF} Theta functions, theta divisor and Riemann's vector
$\Delta$}

Riemann's theta function $\theta$  is a section of a certain line
bundle $\Theta_{\alpha} \to \mbox{Jac}(X)$.  Its divisor is also
denoted by $\Theta$. We refer the reader to
\cite{DP,F,Gu1,Gu2,ABMNV,VV} for background. It is also customary
to denote by $\Theta$  the set of line bundles $\xi \in
\mbox{Pic}^{g-1}(X)$ such that $h^0(X, \xi) = 1$, i.e. which have
a global section (as in \cite{Fal,ABMNV}).

As above, we  denote by $W_{g-1} \subset \mbox{Pic}^{g-1}(X)$ the
subvariety of effective divisors of degree $g-1$.  Riemann's
vanishing theorem states that there exists a point $\Delta \in
\mbox{Pic}^{g-1}(X)$ so that $\Theta = W_{g-1} + \Delta$ (i.e. it
is the translate of $W_{g-1}$ by this point of the torus). Thus,
$\theta(\sum_{j = 1}^{g -1} P_j - \Delta) \equiv 0$ and $\theta(z)
= 0$ if and only if $ z=\sum_{j = 1}^{g -1} P_j - \Delta$ for some
$P_j$. We also denote by  $\theta[\alpha]$ the theta function with
characteristic $\alpha$ (a choice of spin structure).

The  `map' $\zeta_1 + \cdots + \zeta_N \to P_1 + \cdots + P_g$ in
(\ref{ALCAL})  is the composition of the Abel sums map $X^{(N)}
\to \mbox{Jac}(X)$ with the Jacobi inversion `map' $\mbox{Jac}(X)
\to X^{(g)}$. The latter may be described as follows (\cite{ACGH},
p. 28): Let $A_1: X \to \mbox{Jac}(X)$ be the Abel embedding. Then
the  inversion map $\psi: X^{(g)} \to \mbox{Jac}(X)$ is given by
$\psi(\lambda) = A_1^* (\Theta_{\lambda + \Delta})$ as long as
$A_1(X) $ is not contained in the translate $\Theta_{\lambda}$ of
$\Theta$ by $\lambda$. Here, $ A_1^* (\Theta_{\lambda + \Delta})$
is the effective divisor of degree $g$ on $X$. Thus, we have
\begin{equation} \label{PSTUFF} P_1 + \cdots + P_g = A_1^* (\Theta_{A_N (\zeta_1 + \cdots + \zeta_N) +
\Delta}), \end{equation} where $A_N$ is (as above) the Abelian
sums map with basepoint $P_0$.

%

\subsection{\label{PRIMEFORM} Canonical section of a point line bundle and the prime
form}

The canonical sections  ${\bf 1}_{\ocal(P)}(z)$ are defined rather
abstractly. After making identifications, they   may be more
concretely expressed in terms of $\theta$-functions. The prime
form is a holomorphic differential form on $\tilde{X}$  of type
$(-1/2, -1/2)$ of the form $ \frac{z - w}{\sqrt{dz} \sqrt{dw}}. $
The degree of the canonical bundle is $2g - 2$, and that of the
square root spin bundles is $g - 1$. So $ \frac{z - w}{\sqrt{dz}
\sqrt{dw}} $  has a zero at $z = w$ plus $g - 1$ other zeros at
points independent of $w$. To remove the extra zeros,   one
defines the prime form by  (cf. \cite{F}, Definition 2.1)
\begin{equation} E(z,w) = \frac{\theta[\alpha](w - z)}{\sqrt{\omega_{\alpha}(z)} \sqrt{\omega_{\alpha}(w)}} \end{equation}
where $\alpha$ is an odd spin structure (i.e. a choice of square
root of the canonical bundle $K_X$) and $\sqrt{\omega_{\alpha}}$
is a certain holomorphic section of the corresponding spin bundle.
This  prime form vanishes only when $z = w$.  To tie this
discussion together with the abstract one in \S \ref{POINT}, we
have  (see e.g. \cite{DP}, (6.54)):
\begin{mainprop} Let $\iota_w: X \to X \times X$ be the map $\iota_w(z) = (z, w)$, let $\mu(z,w) = A_1(z - w)$, and let $\Theta$ be the
theta-divisor. Then
  each fixed $w,$

\begin{itemize}

\item $  \iota_w^* \mu^* \Theta \otimes K^{-1/2}$ is a line bundle
of degree $g - (g - 1) = 1$;

\item $  \iota_w^* \mu^* \Theta \otimes K^{-1/2} = \zeta_w $

\item
   ${\bf 1}_w(z) = E(z,w) \in H^0(X, \iota_w^* \mu^* \Theta \otimes K^{-1/2})$
   \end{itemize}

   \end{mainprop}

It is not important in this article whether we use the explicit
formula in terms of theta functions or the more abstract
definition in terms of point line bundles. We will be taking the
Hermitian norms of these sections and construct the metrics so
that the isomorphism from the point line bundle setting to the
theta function setting is isometric.  We denote a Hermitian metric
on $\ocal(P)$ by $h_P$ so that the norm of ${\bf 1}_{\ocal(P)}(z)$
is $||{\bf 1}_{\ocal(P)}(z)||_{h_P(z)}$. It is equivalent to
define a Hermitian metric on $\ocal(D) \to X \times X$ and pull it
back under $\iota_P$, so that $||{\bf 1}_{\ocal(P)}(z)||_{h_P} =
||E(z, P)||_{h(z) \boxtimes h(P)}. $

\subsection{Hermitian metrics and Chern classes}

Let $h$ be a smooth Hermitian metric on a holomorphic line bundle
$L \to X$. Its Chern form is defined by
\begin{equation}\label{curvature} c_1(h)= \omega_h : = -\frac{\sqrt{-1}}{2 \pi}\ddbar \log \|e_L\|_h^2\;,\end{equation} where $e_L$ denotes a local
holomorphic frame (= nonvanishing section) of $L$ over an open set
$U\subset M$, and $\|e_L\|_h=h(e_L,e_L)^{1/2}$ denotes the
$h$-norm of $e_L$. We say that $h$ is positive if the (real)
2-form $\omega_h $  is a positive $(1,1)$ form,  i.e.  defines a
\kahler metric.

For any smooth Hermitian metric $h$ and local frame $e_L$ for $L$,
we  write $\|e_L\|_h^2 = e^{-\phi} $ (or,  $h = e^{- \phi}$), and
\begin{equation} \label{DDCPHI} \omega_h = \frac{\sqrt{-1}}{2 \pi} \ddbar \phi  = dd^c
\phi. \end{equation} We  refer to $\phi = - \log ||e_L||_h^2$ as
the potential of $\omega_h$ in $U$, or as the \kahler potential
when $\omega_h$ is a \kahler form.
 As in \cite{ZZ}, we are interested in
general smooth metrics, not only those where $\omega_h$ is
positive.  The metric $h$ induces Hermitian metrics $h^N$ on
$L^N=L\otimes\cdots\otimes L$ given by $\|s^{\otimes
N}\|_{h^N}=\|s\|_h^N$.
 The
 $N$-dependent factor $e^{- N \phi }$ is then the local expression of   $h^N$
 in the local frame $e^N$.

In the reverse direction, suppose that we are given a smooth
$(1,1)$ form $\omega$ with $\int_X \omega = 1$ and a line bundle
$L$ of degree one. Then there exists a  Hermitian metric $h$ on
$L$, unique up to a multiplicative constant,   with $\omega_{h} =
\omega$. To see this, let $h_*$ be any Hermitian metric on $L$.
Then $\int_X (\omega - \omega_*) = 0$ so there exists a $\phi_0
\in C^{\infty}(X)$ orthogonal to the constant functions so that
$\omega - \omega_* = dd^c \phi_0$. It is unique up to an additive
constant but can be normalized to have integral zero with respect
to $\omega_*$.  Then $h = e^{- \phi_0} h_*$.

We adopt the following terminology from \cite{Fal,F} and
elsewhere: Given a    real $(1,1)$ form  $\omega_0 \in H^2(X,
\Z)$, a Hermitian metric on a line bundle $\xi \to X$ is called
{\it $\omega_0$-admissible} if its curvature $(1,1)$ form
(\ref{DDCPHI}) equals $\omega_0$.

\subsection{\label{HIP} Hermitian inner products on $H^0(X, \xi)$}

As mentioned in the introduction, and as reviewed below, our
Gaussian and Fubini-Study measures are induced by a choice of data
$(\omega_0, \nu)$ where $\omega_0 \in H^2(X, \Z)$ has $\int
\omega_0 = 1$ and where $\nu$ is a probability measure on $X$
satisfying the following two technical assumptions: First is the
weighted Bernstein-Markov condition (see \cite{B,ZZ}  or
\cite{BB}, Definition 4.3 and references):
\medskip

\noindent {\it For all $\epsilon
> 0$ there exists $C_{\epsilon} > 0$ so that
\begin{equation} \label{BM} \sup_K \|s(z)\|_{h^N} \leq C_{\epsilon}
    e^{ \epsilon N}
||s||_{G_N(h, \nu)}\,, \quad s\in H_0(\CP^1,\ocal(N)).
\end{equation}
} Here, and throughout this article, we write  \begin{equation}
\label{K} K = \;\; \mbox{ supp}\; \nu. \end{equation} Second is
the assumption that
\begin{equation} \label{REGULAR} K\;  {\it is\; non-thin\;at \;all \; of \; its\;
points}.
\end{equation}  We refer to \cite{ZZ} for further discussion of these
rather mild assumptions.

We only assume that $h$ is a $C^{\infty}$ metric as in
\cite{SZ,ZZ,B,Ber1}. In the local frame any holomorphic section
may be written $s = f e$ where $f \in \ocal(U)$ is a local
holomorphic function. The inner product (\ref{DEFGN}) then takes
the form,
\begin{equation} \label{INNERPRODUCTa} ||s||_{G_N(h, \nu)} =
\int_{\C} |f(z)|^2 e^{-  N \phi} d\nu(z).
\end{equation}

\subsection{\label{BERGMANSECTION} \szego projectors and Bergman kernels}

Let $(L, h)$ be any Hermitian holomorphic line bundle, let $\nu$
 be a probability measure, and let
$G = G(h, \nu)$ be the Hermitian inner product
(\ref{DEFGN})-(\ref{INNERPRODUCTa}). We denote by  $\Pi_{G}:
L^2(X, L) \to H^0(X, L)$  the orthogonal projection with respect
to $G$. It is a section of $L \otimes \bar{L} \to X \times X$. If
we choose a local frame $e_L$ for $L$, the it may be expressed as
$\Pi_G = B_G(z,w) e_L \otimes e_L^*$. The coefficient is the
Bergman kernel function relative to a frame $e_L$. If $S_j$ is an
orthonormal basis of $H^0(X, L)$ with respect to $G$, then locally
$S_j = f_j e_L$ where $f_j$ are local holomorphic functions and
$B_G (z,w) = \sum_{j = 1}^d f_j(z) \overline{f_j}(w)$ where $d =
\dim H^0(X, L)$. Thus, if $s = f e_L$ is any section,
$$\Pi_G s(z) = \left(\int_X B_G(z,w) f(w) e^{- \phi} d\nu\right) e_L(z), , \;\;\; (e^{-
\phi} = |e_L|^2_h) $$

\subsection{\label{COHERENTSTATE} Coherent states}

Given the inner product $G =G_{(h, \nu)}$ on $H^0(X, L)$ and  a
point $P \in X$, the associated coherent state is defined by
\begin{equation} \label{CS} \Phi_{G}^{P}: =
  \Pi_{G}(\cdot, P). \end{equation}
  Its important property is that
  \begin{equation} s(P) = \langle s, \Phi_{G}^P \rangle_G, \;\;\; s \in H^0(X, L). \end{equation}
  Thus, $\Phi^P_G$ represents the evaluation functional at $P$. It
  is also useful to a scalar valued evaluation functional by
  picking a non-zero $v_P \in L_P$ and tensoring $\Phi_G^P \otimes
  v_P^*$, i.e.
producing $\langle s(P), v_P \rangle_G$.

\subsection{\label{FSVOL} Fubini-Study volume form}

 Let $Z
\in \C^{d + 1}$ and let $||Z||^2 = \sum_{j = 1}^{d + 1} |Z_j|^2. $
In the open dense chart $Z_0 \not= 0$, and
 in affine coordinates $w_j =
\frac{Z_j}{Z_0}$, the Fubini-Study  volume form is given  by,
\begin{equation} \label{FSAFFINE} dVol_{I} = \frac{\prod_i dw_i \wedge d\bar{w}_i}{(1 +
||W||^2)^{d + 1}}. \end{equation}

We sometimes work with a basis which is not orthonormal for the
inner product, and therefore  need a more general formula where
the inner product $||Z||^2$ is replaced by any Hermitian inner
product $||A Z||^2$ on $\C^{d + 1}$ where $A \in GL(d + 1, \C)$.
The Fubini-Study metric for this inner product  is  $\ddbar \log
||A Z||^2$. On the affine hyperplane $\{(w, 1)\}$, the  matrix $A = \begin{pmatrix} \acal & \vec{b} \\ & \\
\vec{c} & d \end{pmatrix}$ `acts'   by the projective linear map
$A_{\PP} \cdot W = \frac{\acal W + \vec{b}}{\langle \vec{c}, W
\rangle + d} $; `acts' is in quotes because the group does not
preserve this hyperplane.  A general inner product on $\C^{d + 1}$
may be expressed in the form $||A Z||^2$ for some $A \in GL(N + 1,
\C)$. The associated Fubini-Study metric is $\frac{i}{2 \pi}
\ddbar \log ||A Z||^2 = A^* \frac{i}{2 \pi} \ddbar \log || Z||^2.
$ Hence, in homogeneous coordinates, the induced volume form in
the affine chart is
\begin{equation} \label{FSAFFINEA} dVol_{A} = \left( (w, 1)^* A^* \frac{i}{2 \pi} \ddbar \log || Z||^2 \right)^d=
 A_{\PP}^* \frac{\prod_i dw_i \wedge d\bar{w}_i}{(1 +
||W||^2)^{d + 1}}. \end{equation}

As in \cite{ZZ}, under the natural projection $\pi: \C^{d + 1} -
\{0\} \to \CP^d$,
\begin{equation}\label{VOLAc} \begin{array}{lll}
 \pi^* d\mbox{Vol}_{A}
& = &| \det A|^2 |(A Z)_0|^2 \\ && \\ && \cdot \left(
\frac{\partial}{\partial Z_0} \wedge \frac{\partial}{\partial
\bar{Z}_0} \vdash A^* (dZ_0 \wedge d\bar{Z}_0)  \right)^{-1}
\left( \frac{ \prod_{j = 1}^{d } dZ_j \wedge d \bar{Z}_j}{{||A
Z||^{2 (d + 1)}}}\right),
 \end{array} \end{equation}
where $\left( \frac{\partial}{\partial Z_0} \wedge
\frac{\partial}{\partial \bar{Z}_0} \vdash A^* (dZ_0 \wedge
d\bar{Z}_0)  \right)$ is the coefficient of $dZ_0\wedge d\bar Z_0$
in the form $d(A^*Z)_0\wedge d\overline{(A^* Z)_0}$.

\section{\label{EMBEDIP} Inner products induced by
the canonical embedding}

In this section, we provide background and definitions for the
objects introduced in \S \ref{PLSFSH}. In particular, we define
the large vector space $H^0(X, \lcal_{N + g})$  and  the
embeddings of the spaces $H^0(X, \xi)$ with $\xi \in
\mbox{Pic}^N(X)$ into it.

\subsection{\label{LVS} The large vector space $H^0(X, \lcal_{N + g})$}

Given a line bundle $L$ and a  divisor $D = \sum_i a_i P_i$, let
$\lcal(L; D)$ denote the vector space of meromorphic sections $s$
 satisfying $D + \dcal(s) \geq 0$. Let $s_0$ be a
global holomorphic  section of $[D]$ with $\dcal(s_0) = D$. We
recall that if $D$ is an effective divisor and  $s_0$ be a section
of $H^0(X, [D])$ with $\dcal(s_0) = [D]$, then   multiplication by
$s_0$ gives an identification $H^0(X, \ocal(L \otimes [- D]) )
\simeq \lcal(L;- D) . $

We make constant use of the following case (see \cite{Gu2} page
107):
  Let $\lcal_{N + g} = \ocal((N + g) P_0) \in \mbox{Pic}^{N + g}$, let $\zeta = \zeta_1 + \cdots + \zeta_N \in
X^{(N)} \backslash X^{(N)}_{N + g}$, i.e. assume that $\ocal((N +
g) P_0) \ocal(- (\zeta_1 + \cdots + \zeta_N)) \notin W^1_g$, and
let $P_1 + \cdots + P_g = A_{\lcal_{N + g}}(\zeta_1 + \cdots +
\zeta_N)$. Then multiplication by $\prod_{j = 1}^g {\bf
1}_{\ocal(P_j)}$ defines the isomorphism (\ref{CANONICALIOTA})
\begin{equation} \label{ISO2b} \begin{array}{lll} H^0(X, \ocal(\zeta_1 + \cdots + \zeta_N)) &\simeq & H^0(X,
\ocal((N + g) P_0 - (P_1 + \cdots + P_g)) \\ && \\ &\simeq & \{s
\in H^0(X, \ocal((N + g) P_0)): \dcal(s) \geq P_1 + \cdots +
P_g\}.
\end{array}
\end{equation}

  We now take the product of the sections defined in  Definitions \ref{CANONICALSECTION}-\ref{CANONFACTOR}  to produce canonical sections
   of $H^0(X, \lcal_{N + g})$.

\begin{maindefin} \label{SZETA} Assume that $\zeta = \zeta_1 + \cdots + \zeta_N \in
X^{(N)} \backslash X^{(N)}_{N + g}$.  Then we   define the section
$S_{\zeta_1, \dots, \zeta_N} \in H^0(X, \lcal_{N + g})$ by,

$$S_{\zeta_1, \dots, \zeta_N}(z) := \prod_{j =
1}^g {\bf 1}_{\ocal(P_j)}(z) \cdot \prod_{j = 1}^N {\bf
1}_{\ocal(\zeta_j)}(z), $$
\end{maindefin}


\subsection{\label{TILDEECALN} The vector bundle $\tilde{\ecal}_N$}

Let $\lcal_{N + g}$ be a line bundle of degree $N + g$, which as
always we take to be $\ocal((N + g) P_0)$.  We then define  a
vector bundle over $X^{(g)}$:

\begin{maindefin} \label{ETILDEDEF} $\tilde{\ecal}_N \to X^{(g)}$ is
the vector bundle with fiber,
\begin{equation} \tilde{\ecal}_{P_1 + \cdots + P_g} : =  H^0(X,
\lcal_{N + g} \ocal(-(P_1 + \cdots + P_g)). \end{equation}

We further define the extended configuration space
$$\begin{array}{lll}\tilde{X}^{(N)}: =  \PP \tilde{\ecal}_N & = & \{([s], P_1 + \cdots + P_g): \dcal(s) +
P_1 + \cdots + P_g = (N + g) P_0\} \\ && \\ & &  = \{(\zeta_1 +
\cdots + \zeta_N, P_1 + \cdots + P_g): \zeta_1 + \cdots + \zeta_N
+  P_1 + \cdots + P_g = (N + g) P_0\}.]\end{array}
$$

\end{maindefin}

Here as usual, equality means equality of divisor classes. There
is a natural map $\tilde{X}^{(N)} \to X^{(N)}$ which is an
analytic isomorphism away from $X^{(N)}_{N + g}$, i.e. on the
locus where there exists a unique $P_0 + \cdots + P_g$ so that $
\zeta_1 + \cdots + \zeta_N +  P_1 + \cdots + P_g = (N + g) P_0$.
In effect, we add extra points to the configuration space when the
representation is not unique. The purpose for doing this is to
analyze the behavior of volume forms along the bad set.

 There
exists a natural diagram of maps
\begin{equation}\label{DIAa}  \begin{array}{lll}  \PP \tilde{\ecal}_N = \tilde{X}^{(N)} & \stackrel{\dcal}{  \rightarrow}& \PP
\ecal_N = X^{(N)}
 \\ \rho \downarrow & & \downarrow \pi \\
X^{(g)} & \stackrel{\iota_{\lcal}}{  \rightarrow} & \mbox{Pic}^N
\simeq \mbox{Jac}(X),
\end{array}
\end{equation}
where the bottom arrow is the map $P_1 + \cdots + P_g \to \lcal
\otimes \ocal(-(P_1 + \cdots + P_g))$ and the top arrow is the
induced identification of  sections.  We need the diagram of
inverse maps and run into the usual Jacobi inversion problem, i.e.
that the the maps are not invertible along the
  the Wirtinger subvarieties. But they are analytic isomorphisms
  away from these subvarieties:

\begin{equation}\label{DIA}  \begin{array}{lll} X^{(N)} \backslash X^{(N)}_{N + g} \simeq \PP \ecal_N
\backslash W^N_{(N + g)P_0}  &
\stackrel{\iota_{\lcal}^*}{\rightarrow} & \PP \tilde{\ecal}_N
 \\ \pi \downarrow & & \downarrow \rho \\
\mbox{Pic}^N & \stackrel{\iota_{\lcal}}{  \leftarrow} & X^{(g)}
\end{array}
\end{equation}
The bottom map is singular (degenerate) along the hypersurface of
$W^1_g \subset X^{(g)}$ where $\dim H^0(X, \ocal(P_1 + \cdots +
P_g)) = 2$.  The image under $\pi$ of this set is the set of line
bundles $\xi \in \mbox{Pic}^N$ so that $\lcal_{N + g} \xi^{-1} \in
W^1_g$. So the singular sets of the maps are compatible with the
diagram.

A natural question is to compare $\PP H^0(X, \lcal_{N + g})$, $\PP
\ecal_N$ and
 and $\PP \tilde{\ecal}$ (\ref{TILDEECALDEF}).  As in (\ref{ISO2b}), there is a natural identification
\begin{equation} \tilde{\ecal}_{P_1 + \cdots + P_g} : = \{s \in H^0(X,
\lcal_{N + g}): \dcal(s) \geq P_1 + \cdots + P_g\}, \end{equation}
defined by multiplication by $\sigma_{\lcal}
 = \prod_{j = 1}^g {\bf 1}_{\ocal(P_j)}$. Since we understand the
 relation between $\PP
\ecal_N$ and
 and $\PP \tilde{\ecal}$, the main point is the following

\begin{mainprop} \label{CANONICAL} The fiberwise canonical map
$\sigma_{\lcal}: \PP \tilde{\ecal}_N \to \PP H^0(X, \lcal_{N +
g})$,
$$\sigma_{\lcal} (P_1 + \cdots + P_g, s) = \prod_{j = 1}^g {\bf 1}_{\ocal(P_j)} s, \;\;\;P_1
 + \cdots + P_g = \rho(s) $$   is an analytic branched  ${(N +
g) \choose N}$-fold cover
taking the fiber of $\PP \tilde{\ecal}_N $ over $P_1 + \cdots +
P_g \in X^{(g)}$   to the subspace
$$
\PP \{S \in H^0(X, \ocal(\lcal_{N + g}): \dcal(S) \geq P_1 +
\cdots + P_g \} \subset \PP H^0(X, \lcal_{N + g}).
$$
The branch locus $\bcal^{(N)}_{N + g} \subset \tilde{X}^{(N)}$
consists of $\{(\zeta_1+ \cdots + \zeta_N, P_1 + \cdots+ P_g)\}$
with  multiple points.
\end{mainprop}

\begin{proof}

 We first show that $D \sigma_{\lcal} $ is an isomorphism on
  the open dense subset $\tilde{X}^{(N)} \backslash \bcal^{(N)}_{N + g}$.    Since a  holomorphic map of complex manifolds of the same
dimension is surjective if its differential is surjective at one
point, this will prove that $\sigma_{\lcal} $ is surjective and
that it is   a branched covering map.

Let $s_t$ be a curve in $\tilde{\ecal}_N$, and let $
\sigma_{\lcal}(s_t)$ be the image curve in $H^0(X, \lcal_{N +
g})$. Then
$$D \sigma_{\lcal} (\dot{s}) = \frac{d}{dt} |_{t = 0} s(t)
\sigma_{\lcal}(t) = \dot{s}_0  \prod_{j = 1}^g {\bf
1}_{\ocal(P_j)} + s_0 \frac{d}{dt} |_{t = 0}  \prod_{j = 1}^g {\bf
1}_{\ocal(P_j)}. $$ If the left side is zero, then
$$\frac{\dot{s}}{s} = - \sum_{j = 1}^g  \frac{ \dot{{\bf
1}}_{\ocal(P_j)}}{ {\bf 1}_{\ocal(P_j)}}. $$ Since $s  = C
\prod_{j = 1}^N {\bf 1}_{\ocal(\zeta_j)}, $ this implies
$$ \sum_{k = 1}^N  \frac{ \dot{{\bf
1}}_{\ocal(\zeta_k)}}{ {\bf 1}_{\ocal(\zeta_k)}} + \sum_{j = 1}^g
\frac{ \dot{{\bf 1}}_{\ocal(P_j)}}{ {\bf 1}_{\ocal(P_j)}} = 0. $$
If the $\{\zeta_k, P_j\}$ are all distinct then this equation
cannot hold since the poles on the left side occur at some of the
$\{\zeta_k\}$ and on the right at some of the $\{P_j\}$. Hence $D
\sigma_{\lcal}$ is injective (and therefore an isomorphism) away
from the large diagonals.

Further, it  is  ${ N + g \choose N}$ to $1$ on the same set.
  Indeed, given $[S]
\in \PP H^0(X, \lcal_{N + g})$, we split up its $N + g$ zeros into
two groups, one of $N$ points $\zeta_1 + \cdots + \zeta_N$ and one
of $g$ points $P_1 + \cdots + P_g$. There does not appear to be
any preferred way to do this, so we consider all possible ways.
Since $s$ is a section, $\prod_{j = 1}^N \ocal(\zeta_j) \otimes
\prod_{j = 1}^g \ocal(P_j) = \lcal_{N + g}$,  it follows that
$A_{\lcal_{N + g}} (\zeta_1 + \cdots + \zeta_N) = P_1 + \cdots +
P_g.$ Hence $[S]$ is the image of the section $\frac{S}{ \prod_{j
= 1}^g {\bf 1}_{\ocal(P_j)}} $ under the canonical map. When the
points $\{\zeta_1, \dots, \zeta_N, P_1, \dots,P_g\}$  are
distinct, there are ${ N + g \choose N}$ to split up the zeros
into a subset of $N$ elements and a subset of $g$ elements. Hence
the line through $s$ is the image of ${ N + g \choose N}$
sections, and it is clear that these are the only ways that $S$ is
the image of an element. On the hypersurfaces where $\{\zeta_j,
P_k\}$ has multiple points, there are fewer ways to split up the
set of zeros; this is the branch locus.

\end{proof}

 Proposition \ref{CANONICAL} indicates why it is simpler to
 work with $\PP \tilde{\ecal}_N$ than $\PP \ecal_N$. In the latter
 case we would need to puncture out $X^{(N)}_{N + g}$ in order to
 defined the map to $\PP H^0(X, \lcal_{N + g})$, since $P_1 +
 \cdots + P_g$ would not be uniquely defined on the Wirtinger
 subvariety. In the case of $\PP \tilde{\ecal}$, the the line
 bundle is duplicated at points $P_1 + \cdots + P_g, P_1' + \cdots
 + P_g'$ so that $\lcal_{N + g} \otimes \ocal(- (P_1 + \cdots +
 P_g)) \simeq \lcal_{N + g} \otimes \ocal(- (P_1' + \cdots +
 P_g'))$. This redundancy makes it possible to define the map to
 $\PP \tilde{\ecal}$ at all points of $\PP \tilde{\ecal}_N$.

%

\subsection{\label{ADMHERM} $\omega$-admissible family of Hermitian metrics on $\tilde{\ecal}_N$}

Having reviewed the relevant Abel-Jacobi theory, we now return to
the question of defining   Hermitian metrics on the vector bundle.
$\tilde{\ecal}_N$, i.e. a smooth family of Hermitian inner
products on the spaces $H^0(X, \xi)$ for $\xi \in \mbox{Pic}^N$.
As mentioned in the introduction, we define them by fixing a
Hermitian metric $h_0$ on $\ocal(P_0)$ and the associated
Hermitian metrics $h_0^{N + g}$ on $\ocal((N + g) P_0) = \lcal_{N
+ g}$ with Chern form $\omega$. Together with the Bernstein-Markov
measure $\nu$ we obtain an inner product $G_{N + g}(h, \nu)$ as in
(\ref{DEFGN}). We then define inner products $G_N(h, \nu, L)$ on
$H^0(X, L)$ by specifying that the  maps $\sigma_{\lcal}$ are
isometric. That is, we restrict $G_{N + g}(h, \nu)$ to each
embedded subspace and thus induce an inner product on each $H^0(X,
L)$. We refer to the family of such inner products as  the
$\omega$- {\it admissible Hermitian metrics}. The Gaussian
measures induced Fubini-Study measures on the associated
projective spaces of sections.

\subsection{\label{OPT} $\omega_0$-admissible metrics and admissible Hermitian
inner products}

As mentioned in the introduction, there is  another natural way to
define a family of $\omega_0$-admissible Hermitian inner products
on the line bundles $\xi \in \mbox{Pic}^N(X)$ and an associated
family of $\omega_0$-admissible Hermitian inner products on the
spaces $H^0(X, \xi)$. Namely, we equip the line bundles $\xi$ with
$\omega_0$ admissible metrics and then use (\ref{DEFGN}) to define
associated Hermitian inner products.

This approach involves the complication that the admissible line
bundle metrics are only unique up to a constant, and therefore the
family of metrics as $\xi$ ranges over $\mbox{Pic}^N$ is only
unique up to a function on $\mbox{Pic}^N$. This constant can be
fixed up to an overall constant $C_N$ by using the Faltings metric
on the associated determinant line bundle of $\ecal_N \to
\mbox{Pic}^N$, i.e. a Hermitian metric on $\bigwedge^{\mbox{top}}
\ecal_N$.

A further complication is that the Hermitian inner products on
$H^0(X, \xi)$ differ from the inner products defined by the
canonical embedding by the factor of $\sigma_{\lcal} = \prod_{j =
1}^g {\bf 1}_{\ocal(P_j)}(z)$. That is, the canonical embeddings
would not be isometric if we used $\omega_0$-admissible metrics to
define admissible Hermitian inner products on $H^0(X, \xi)$. To
deal with this complication, one would need to express to directly
evaluate the Fubini-Study-Haar ensemble with these inner products
in terms of zeros coordinates.

Although the approach in terms of $\omega_0$-admissible Hermitian
metrics and inner products seems very natural and attractive, we
opt for the large vector space (and projective linear ensemble)
for simplicity of exposition.

\section{\label{JPCCALCULATION}  Proof of  Theorem \ref{JPDHG} (I)}

In this section, we prove the first formula for the JPC for the
projective linear ensemble in terms of the prime form and Bergman
kernel determinant. Starting from the  general formula for the
Fubini-Study volume form with respect to an inner product on the
large projective space $\PP H^0(X, \lcal_N)$ (\S \ref{FSVOL}), we
pull back this volume form by a generalized Newton-Vieta map from
zeros to sections (\S \ref{VDSECT}) to obtain a volume form on
$X^{(N)}$. We express it in the natural configuration space
coordinates (zeros coordinates)  as the quotient of a
`Vandermonde' and an $L^2$ factor. In Proposition
\ref{BASICFORMULA}, we express the Vandermonde determinant
 in terms of the prime form. Finally we write the denominator in terms of the prime form
to complete the proof.

We apply  formula (\ref{FSAFFINE}) to the projective space $\PP
H^0(X, \lcal_{N + g})$, where as above $\lcal_{N + g}$ is a line
bundle of degree $N + g$, endowed with the inner product
(\ref{DEFGN}) with $h$ an admissible metric. To obtain an
identification with $\CP^{d_N}$, we need introduce a basis of
$H^0(X, \lcal_{N +g})$.

\subsection{\label{VIETASECT} Coordinates relative to a basis of $H^0(X, \lcal_{N + g})$}

We now define a  special orthonormal basis of $H^0(X, \lcal_{N +
g})$ and (at the same time)   an affine chart $\C^{N + g - 1}$ for
$\PP H^0(X, \lcal_{N + g})$. We follow an analogy with the genus
zero case, and we begin by explaining that case in a form suitable
for generalization. When  $X = \CP^1$, we often use the basis
$\{z^j\}$ on the right side of (\ref{VIETA}).  This basis of
polynomials represents local holomorphic coefficients relative to
the frame $e^N(z) $ of $\ocal(N) \to \CP^1$ corresponding to the
homogeneous polynomial $z_0^N$ in coordinates $(z_0, z_1)$ on
$\C^2$. In these coordinates $z = \frac{z_1}{z_0}$ and $z^j e^N =
z_1^j z_0^{N - j}$.  The section $z_1^N$ is distinguished in this
basis because the coefficient of $z^N$ of the product $\prod_{j =
1}^N (z - \zeta_j) = z^N + \cdots$ always equals one. Hence an
affine chart for $\PP H^0(X, \ocal(N))$ is given by  the affine
space of monic polynomials of this form; the associated  affine
coordinates $w_j$ are the coefficients relative to $1, \dots,
z^{N-1}$.

The element $z^N$ has a natural generalization to the line bundle
$\lcal_N$:  Namely, it is the coherent state
$\Phi_{h_{FS}^N}^{\infty}$ for the Fubini-Study inner product
centered at the point $\infty \in \CP^1$ (see \S
\ref{COHERENTSTATE}). Indeed, for any holomorphic section
(polynomial) $s$, $\langle s, \Phi_{h_{FS}^N}^{\infty} \rangle =
s(\infty)$, while $\langle s, z^N \rangle = a_N$, where $s =
\sum_{j = 0}^N a_j z^j$. So we need to see that $a_N = s(\infty)$.
But in homogeneous coordinates, $z_0 = 0$ defines $\infty$, so
all monomials $z^j e^N = z_1^j z_0^{N - j}$ with $j \not= N$
vanish at $\infty$.

 Given the inner product $G_N(h, \nu)$ on $H^0(X, \lcal_{N
+g})$, the  \szego projector $\Pi_{\lcal_{N + g}}$ is defined to
be the orthogonal projection from all $L^2$ sections of $\lcal_{N
+g}$ onto the space $H^0(X, \lcal_{N +_g})$ with respect to
$G_N(h, \nu)$. For simplicity, we do not include the data $(h,
\nu)$ in the notation for $\Pi_{\lcal_{N +g}}$. Given a point $P
\in X$, the associated coherent state is defined as in \S
\ref{COHERENTSTATE}, i.e.
\begin{equation} \label{CSb} \Phi_{N
+g}^{P}: =
  \Pi_{\lcal_{N
+g}}(\cdot, P). \end{equation}
  Its important property is that
  \begin{equation} s(P) = \langle s, \Phi_{N
+g}^{P} \rangle_{G_N(h,
  \nu)}, \;\;\; s \in H^0(X, \lcal_{N
+g}). \end{equation} Strictly speaking,  $\Phi^P(z)$ is a section
with values in $\lcal_P$, and we need to tensor with a covector in
$\lcal_P^*$ to cancel this factor.

 To generalize $z^N$ we therefore pick a base point $P_0$ and
use the element $$\hat{\psi}_0 : = \frac{\Phi_{N +
g}^{P_0}}{||\Phi_{N + g}^{P_0}||}$$ as a distinguished basis
element of $H^0(X, \lcal_{N + g})$.  We also use it as a local
frame for $\lcal_{N + g} \to X$. We then pick a $G_{N + g}(\omega,
\nu)$-orthonormal basis $\{\hat{\psi}_j\} $ for $H^0_{P_0}(X,
\lcal_{N + g}) \simeq H^0(X, E_{N + g - 1})$, and write them
locally as $\hat{\psi}_j = \psi_j \Phi_{N + g}^{P_0}\}$,  where
$\psi_j$ are local coefficient functions in the frame (An
advantage of working with the large vector space is that we can
fix a convenient basis for it.) We then have the orthogonal
decomposition,
\begin{equation} H^0(X,
\lcal_{N + g}) = H^0_{P_0} (X, \lcal_{N + g}) \oplus \C \Phi_{N
+g}^{P_0}, \;\;\; \mbox{where} \;\; H^0_{P_0} (X, \lcal_{N + g}) =
\{s: s(P_0) = 0\}.
\end{equation}
We further define the auxiliary line bundle  \begin{equation}
\label{EDEF} E_{N + g - 1} = \lcal_{N + g} \otimes \ocal(- P_0)
\end{equation} and define the isomorphism
\begin{equation} \otimes {\bf 1}_{P_0} : H^0(X, E_{N + g -1}) \to
H^0_{P_0}(X, \lcal_{N + g}). \end{equation} Equipped with the
admissible metrics, this isomorphism is an isometry. We only
introduce $E_{N + g -1}$ to quote relevant facts from Abel-Jacobi
theory from the literature, and to be able to speak of Bergman
kernels rather than the conditional Bergman kernels  for
$H^0_{P_0}(X, \lcal_{N + g})$; the latter are the principal
objects.

We let $Z_j$ denote coordinates with respect to
$\{\hat{\psi}_j\}$. We view $H^0_{P_0} = \{Z: Z_0 = 0\}$ as the
`hyperplane at infinity' and define  the affine coordinates $w_j$
on $\PP H^0(X, \lcal_{N + g}) \backslash H^0_{P_0}$ by $w_j =
\frac{Z_j}{Z_0}$. Thus, the projective coordinates are ratios of
the coefficient functions $f_j$ of the section in the frame.

\begin{maindefin}
 We
then put
\begin{equation} \label{ECALNDEF} \ecal_{N + 1 - j}(\zeta_1, \dots, \zeta_N) = w_j (S_{\zeta}),
\;\; (j = 0, \dots, N). \end{equation} Thus, \begin{equation}
\label{SZETA1} \tilde{S}_{\zeta_1, \dots, \zeta_N} : = ||\Phi_{N +
g}^{P_0}|| \frac{S_{\zeta_1, \dots, \zeta_N}}{\langle S_{\zeta_1,
\dots, \zeta_N}, \Phi_{N + g}^{P_0} \rangle }: = \cdot \sum_{j =
0}^{N } \ecal_{N - j}(\zeta_1, \dots, \zeta_N) \psi_j(z).
\end{equation}
\end{maindefin}

 By definition,
\begin{equation} \label{ECAL01} \ecal_0(\zeta_1, \dots, \zeta_N) =1, \end{equation}
generalizing the affine space of monic polynomials in the basis
$\{z^j\}$. In that case, $\langle S_{\zeta_1, \dots, \zeta_N},
z^{N } \rangle = 1$.

Since $S_{\zeta_1, \dots, \zeta_N}$ is not well-defined if
$\zeta_1 + \cdots + \zeta_N \in X^{(N)}_{N + g}$, we regard it as
defined only on the complement. Alternatively, it is a
well-defined map from extended configuration space
$\tilde{X}^{(N)}$.

\subsection{Pull back of the Fubini-Study volume form on $\PP H^0(X, \lcal_{N +
g})$}

The first step in the proof of Theorem \ref{JPDHG} (I) is the
following  preliminary version of the formula for the JPC $\vec
K^N_{PL}$ of the projective linear ensemble:

\begin{mainprop} \label{VOLAd} The pullback to $X^{(N)} \backslash X^{(N)}_{N + g}$ under $\psi_{\lcal}$
(or to $\tilde{X}^{(N)}$ under $\sigma_{\lcal}$) of the
Fubini-Study volume form on $\PP H^0(X, \lcal_{N + g})$ with
respect to the inner product $G_{N + g}(\omega, \nu)$ is given by,

\begin{eqnarray}
 \vec K^N_{PL} = ||\Phi_{N + g}^{P_0}||^{- 2(N + 1)} \left|\langle
S_{\zeta_1, \dots, \zeta_N}, \Phi_{N + g}^{P_0} \rangle \right|^{2
N + 2 } \;\; \frac{\prod_{j = 1}^{N } d\ecal_j \wedge d
\bar{\ecal}_j}{||S_{\zeta_1, \dots, \zeta_N}||_{L^2(G_{N +
g}(\omega, \nu))}^{2(N + 1)}}.
 \end{eqnarray}
 \end{mainprop}

 \begin{proof}

We  evaluate the Fubini-Study volume form (\ref{VOLAc}) on $\PP
H^0(X, \lcal_{N + g})$ with respect to the inner product $G_{N +
g}(\omega, \nu)$. The homogeneous coordinates $w_j$  on the chart
where $Z_0 = \langle S_{\zeta_1, \dots, \zeta_N}, \Phi_{N +
g}^{P_0} \rangle  \not= 0$ are defined in (\ref{ECALNDEF}), or
equivalently, $Z_j  = \langle S_{\zeta_1, \dots, \zeta_N},
\frac{\Phi_{N + g}^{P_0}}{||\Phi_{N + g}^{P_0}|| }\rangle  \not
\ecal_{N + 1 - j}(\zeta_1, \dots, \zeta_N)$, and
$$S_{\zeta_1, \dots, \zeta_N} = \prod_{j =
1}^g {\bf 1}_{\ocal(P_j)}(z) \cdot \prod_{j = 1}^N {\bf
1}_{\ocal(\zeta_j)}(z), = \sum_{k=0}^N Z_k \psi_k. $$ We then pull
back the form (\ref{VOLAc}) under the section $(w, 1) = ||\Phi_{N
+ g}^{P_0}||^2  \frac{S_{\zeta_1, \dots, \zeta_N}}{\langle
S_{\zeta_1, \dots, \zeta_N}, \Phi_{N + g}^{P_0} \rangle }$. Since
we have picked an orthonormal basis for $H^0(X, \lcal_{N + g})$,
the matrix $A$ in (\ref{VOLAc}) is the identity matrix.  The
formula for the denominator follows from
$$(1 +
||W||^2)^{N + 1} = ||\Phi_{N + g}^{P_0}||^{2(N + 1)} ||\langle
S_{\zeta_1, \dots, \zeta_N}, \Phi_{N + g}^{P_0} \rangle ||^{- 2 (N
+ 1)}\;\; ||S_{\zeta}||^{2 (N + 1)}. $$

\end{proof}


\subsection{\label{VDSECT} Vandermonde in higher genus}

The next step is to simplify the form in the numerator in
Proposition \ref{VOLAd}.  By the higher genus Vandermonde
determinant we mean Jacobian determinant $ \det
\begin{pmatrix} \frac{\partial \ecal_n}{\partial \zeta_j}
\end{pmatrix} $  defined by
\begin{equation} \label{DECAL}   \; d \ecal_1 \ \wedge
\cdots \wedge d \ecal_N  = \det
\begin{pmatrix} \frac{\partial \ecal_n}{\partial \zeta_j}
\end{pmatrix} \prod_{j = 1}^N d \zeta_j.
\end{equation}
Here, we assume that $\zeta_1 \cdots + \zeta_N$ does not lie in
the branch locus of $p_N: X^N \to X^{(N)}$ so that we can use
$\zeta_j$ as local coordinates (see \S \ref{XN}). That is, we fix
a trivializing chart $U$ for $\ocal(P_0)$ centered at $P_0$, and a
trivializing frame $e$ for $\ocal(P_0)$. We let $\zeta$ denote a
local holomorphic coordinate in $U$ which vanishes at $P_0$ and we
denote the associated coordinates on $U^{(N)}$ by $\{\zeta_1,
\dots, \zeta_N\}$.  We express each section of $\lcal_{N + g}$ as
a local holomorphic function times this frame. For simplicity, we
use the same notation for sections and their local holomorphic
functions relative to this frame. We now prove the formula alluded
to in (\ref{DELTAG}).

\begin{mainprop} \label{BASICFORMULA} Let $\ecal_{N - j}(\zeta_1, \dots, \zeta_N)$ be defined by
(\ref{SZETA1}). Then $$\det \left( \begin{pmatrix} \frac{\partial
\ecal_{N - j}}{\partial \zeta_k}
\end{pmatrix}_{j, k = 1}^{N } \right)   = || \Phi_{N + g}^{P_0}||^{ N} \frac{\prod_{k = 1}^N
\prod_{j = 1}^g {\bf 1}_{\ocal(P_j)}(\zeta_k) }{\langle
S_{\zeta_1, \dots, \zeta_N}, \Phi_{N + g}^{P_0} \rangle^N \; \det
\begin{pmatrix} \psi_n(\zeta_j) \end{pmatrix}} \;
\left(\prod_{k = 1}^N \prod_{j: k \not= j} {\bf
1}_{\ocal(\zeta_j)} (\zeta_k)\right).
$$
 \end{mainprop}

The determinant omits $j = 0$, in which case  $\ecal_N \equiv 1$
and the derivatives vanish.  We also observe that when taking the
norm square of this expression, the factor $\left|\langle
S_{\zeta_1, \dots, \zeta_N}, \Phi_{N + g}^{P_0} \rangle \right|^{2
N +2}$ in Lemma \ref{VOLAd} cancels all but two powers in  the
norm-square of $\langle S_{\zeta_1, \dots, \zeta_N}, \Phi_{N +
g}^{P_0} \rangle^N $ in the denominator of (\ref{BASICFORMULA}).

\begin{rem}
The Slater determinant  $\det
\begin{pmatrix} \psi_n(\zeta_j) \end{pmatrix}$  is a section of the highest power of the exterior
tensor product
$$\pi_1^* \lcal_{N + g }  \boxtimes \cdots \boxtimes \pi_N^*
\lcal_{N + g }  \to X^N$$ where $\pi_j : X^N \to X$ is the
projection to the $j$th factor. On the other hand,
$$\left( \prod_{j = 1}^g {\bf
1}_{\ocal(P_j)}(\zeta_k)   \cdot \prod_{j: k \not= j} {\bf
1}_{\ocal(\zeta_j)} (\zeta_k)\right) d \zeta_k = \pi_k^*
\partial  \left( \prod_{j = 1}^g {\bf
1}_{\ocal(P_j)}   \cdot \prod_{j= 1}^N {\bf 1}_{\ocal(\zeta_j)}
\right) (\zeta_k) \in \pi_k^* \lcal_{N + g} \otimes K_X. $$ Taking
the exterior tensor product $\prod_{k = 1}^N$ of these $(1,0)$
forms and taking the ratio with the Slater determinant produces a
well-defined $(N, 0)$ form on $X^{(N)}$, i.e. a section of
$\pi_1^* K_X  \boxtimes \cdots \boxtimes \pi_N^* K_X \to X^N$.
Thus, as mentioned after the statement of Theorem \ref{JPDHG}, the
ratio is well defined without the choice of a Hermitian metric.
Note also that we only obtain special sections since  $\psi_j \in
H_{P_0}^0(X, \lcal_{N + g})$.

\end{rem}

 The proof of Proposition \ref{BASICFORMULA} consists of
two Lemmas.

\begin{lem} \label{SLATER1} We have the following identity on determinants of $N \times N$ matrices,
 $$\det \begin{pmatrix}\sum_{n = 1}^N \frac{\partial
\ecal_{N - n} }{\partial \zeta_1} \psi_n(\zeta_1), & &  \sum_{n =
1}^N \frac{\partial
\ecal_{N - n}}{\partial \zeta_1} \psi_n(\zeta_N)\\  && \\
\sum_{n = 1}^N \frac{\partial \ecal_{N - n}}{\partial \zeta_N}
\psi_n(\zeta_1) && \sum_{n = 1}^N \frac{\partial \ecal_{N -
n}}{\partial \zeta_N} \psi_n(\zeta_N)
\end{pmatrix} = \det \begin{pmatrix} \frac{\partial
\ecal_{N - n}}{\partial \zeta_j} \end{pmatrix}_{j,n = 1}^N \det
\begin{pmatrix} \psi_n(\zeta_j) \end{pmatrix}_{j,n = 1}^N
$$
\end{lem}

As above, the sums omit $n = 0$ since $\ecal_0 = 1$.

\begin{proof} For each $n$ we consider the row vector $\Psi_n(\zeta) : = \left[
\psi_n(\zeta_1), \dots, \psi_n(\zeta_N) \right]$. Then the $j$th
row of our matrix is $\sum_{n = 1}^N \frac{\partial \ecal_{N -
n}}{\partial \zeta_j} \Psi_n(\zeta)$, so we are calculating
$$\sum_{n = 1}^N \frac{\partial \ecal_{N - n}}{\partial
\zeta_1} \Psi_n(\zeta) \bigwedge \cdots \bigwedge \sum_n
\frac{\partial \ecal_{N - n}}{\partial \zeta_N} \Psi_n(\zeta). $$
Clearly this gives the  sum
$$\left(\sum_{\sigma \in \Sigma_N} \epsilon(\sigma)
\frac{\partial \ecal_1}{\partial \zeta_{\sigma(1)}} \cdots
\frac{\partial \ecal_N}{\partial \zeta_{\sigma(N )}} \right)
\Psi_1(\zeta) \wedge \Psi_2(\zeta) \wedge \cdots \wedge
\Psi_N(\zeta) $$ stated in the Proposition.

\end{proof}

We now calculate the left side in Lemma \ref{SLATER1} in a
different way:

\begin{lem}\label{SLATER2G}  With the above notation and
conventions, we have
 $$\begin{array}{l} \det \begin{pmatrix}\sum_{n = 1}^N \frac{\partial
\ecal_{N - n} }{\partial \zeta_1} \psi_n(\zeta_1), & &  \sum_{n =
1}^N \frac{\partial
\ecal_{N - n}}{\partial \zeta_1} \psi_n(\zeta_N)\\  && \\
\sum_{n = 1}^N \frac{\partial \ecal_{N - n}}{\partial \zeta_N}
\psi_n(\zeta_1) && \sum_{n = 1}^N \frac{\partial \ecal_{N -
n}}{\partial \zeta_N} \psi_n(\zeta_N)
\end{pmatrix}\\ \\= || \Phi_{h^N}^{P_0}||^{N} \langle S_{\zeta_1, \dots,
\zeta_N}, \Phi_{h^N}^{P_0} \rangle^{- N }\; \prod_{k = 1}^N
\left(\prod_{j = 1}^g {\bf 1}_{\ocal(P_j)} (\zeta_k) \cdot
\prod_{j: k \not= j} {\bf 1}_{\ocal(\zeta_j)}(\zeta_k)\right).
\end{array}$$
\end{lem}

\begin{proof} By definition of $S_{\zeta_1, \dots, \zeta_N}$ (Definition \ref{SZETA}) and by (\ref{SZETA1}),  we have
$$|| \Phi_{h^N}^{P_0}|| \langle S_{\zeta_1, \dots,
\zeta_N}, \Phi_{h^N}^{P_0} \rangle^{-1} \prod_{j = 1}^g  {\bf
1}_{\ocal(P_j)} (\zeta_k) \cdot \prod_{j = 1}^N {\bf
1}_{\ocal(\zeta_j)}(z)= \sum_{j = 0}^{N } \ecal_{N  - j}(\zeta_1,
\dots, \zeta_r) \psi_j(z).
$$

Recall that we are working in the chart where $\langle S_{\zeta_1,
\dots, \zeta_N}, \Phi_{h^N}^{P_0} \rangle \not= 0$. If we
differentiate this factor or $\prod_{j = 1}^g {\bf
1}_{\ocal(P_j)}$ and set $z = \zeta_k$, the second factor
vanishes. So we only need to differentiate the factor $\prod_{j =
1}^N {\bf 1}_{\ocal(\zeta_j)}(z)$ and multiply it by the other two
factors with $z = \zeta_k$.

The factor ${\bf 1}_{\ocal(\zeta_j)} (\zeta_k)$ vanishes if $j =
k$, so differentiation in $\zeta_n$ (in the local coordinate) must
remove this factor to get a non-zero result on the left side  when
we set
 $z = \zeta_k$. Using that ${\bf 1}_P(Q) \sim P - Q$ near the
 diagonal, we get
$$\begin{array}{lll}  \frac{\partial}{\partial \zeta_n}
  \cdot \prod_{j = 1}^N {\bf 1}_{\ocal(\zeta_j)}(z)
|_{z = \zeta_k} & = &      \cdot \prod_{j \not= n} {\bf
1}_{\ocal(\zeta_j)}(\zeta_n)
 \\&& \\
&= & \delta_{n k} \;   \prod_{j \not= n} {\bf
1}_{\ocal(\zeta_j)}(\zeta_k).
\end{array} $$

 Here we differentiate the local expression for ${\bf 1}_{\ocal(P)}(z)$  as if it were
 a function rather than a section. To be more precise, we implicitly use the Chern  connection for
 the admissible metric on  $\ocal(1)$. But the connection term
 vanishes when we evaluate at $z = \zeta_k$,  so the covariant derivative produces
 the same result as the local derivative.

 It follows that
$$\det \begin{pmatrix} \frac{\partial}{\partial \zeta_j} S_{\zeta_1, \dots, \zeta_N}(z) |_{z = \zeta_k} \end{pmatrix}
=\; \prod_{k = 1}^N \prod_{j = 1}^g {\bf 1}_{\ocal(P_j)}(\zeta_k)
\left(  \prod_{j: j \not= k} {\bf
1}_{\ocal(\zeta_j)}(\zeta_k)\right). $$

Multiplying by the other factors completes the proof   of the
Lemma.
\end{proof}

Combining the two Lemmas we obtain Proposition \ref{BASICFORMULA}.

\begin{rem}

To put these calculations into perspective, we recall that in
genus zero, with $\psi_j(z) = z^j$ for $j = 0, \dots, N - 1$,  the
Vandermonde determinant arises  in the determinant formula:
 \begin{equation} \label{g0}  \det \begin{pmatrix}\sum_{n = 1}^N \frac{\partial
\ecal_{N - n} }{\partial \zeta_1} \psi_n(\zeta_1), & &  \sum_{n =
1}^N \frac{\partial
\ecal_{N - n}}{\partial \zeta_1} \psi_n(\zeta_N)\\  && \\
\sum_{n = 1}^N \frac{\partial \ecal_{N - n}}{\partial \zeta_N}
\psi_n(\zeta_1) && \sum_{n = 1}^N \frac{\partial \ecal_{N -
n}}{\partial \zeta_N} \psi_n(\zeta_N)
\end{pmatrix}  =
 \prod_{j: k \not= j} (\zeta_j - \zeta_k) = \Delta^2(\zeta).
 \end{equation}
The proof is that
$$\begin{array}{lll} \frac{\partial}{\partial \zeta_j}  \prod_{n = 1}^N (z - \zeta_j)|_{z = \zeta_k}
& = &  \prod_{n \not= j} (\zeta_n - \zeta_k) \\ && \\
& = & \delta_{j k}   \prod_{n \not= k} (\zeta_n - \zeta_k).
\end{array} $$
 hence
$$\det \begin{pmatrix} \frac{\partial}{\partial \zeta_j} \prod_{j = 1}^N (z - \zeta_j) |_{z = \zeta_k} \end{pmatrix} =
\;  \prod_{j: k \not= j} (\zeta_j - \zeta_k). $$ Further, the
Slater determinant $\det (\psi_j(\zeta_k)) = \Delta(\zeta)$. So
the two factors in Proposition \ref{BASICFORMULA} cancel to leave
$\Delta(\zeta)$.

\end{rem}

\subsection{\label{Slatersection} Slater determinant and Bergman determinant}

To complete the proof of Theorem \ref{JPDHG} (I), we relate the
Slater determinants $ \det
\begin{pmatrix} \psi_n(\zeta_j)
\end{pmatrix}_{j,n = 1}^N$ in Lemma \ref{SLATER1} to Bergman
determinants. The following Lemma is a general fact for any
Hilbert space (it also applies to $H^0_{P_0}(X, \lcal_{N + g})$).

 \begin{lem}\label{SLATERa} Let $E$ be a line bundle of degree $n + g - 1\geq 2g - 1$, and let  $G$ be an inner product on $H^0(X,
 E)$. Let $\{\psi_j\} = \{f_j e_E\}_{j = 1}^{n }$ be a basis for $H^0(X, E)$ Then
 any $(\zeta_1, \dots, \zeta_n) \in X^{(n)}$,
$$\frac{\left| \det \left( f_j(\zeta_k) \right)_{j, k = 1}^n  \right|^2}{ \det
\left( \langle \psi_j, \psi_k \rangle_{G} \right)} = \;\; \det
\left(B_G(\zeta_j, \zeta_k) \right)_{j, k = 1}^n  . $$
\end{lem}

If we include the frame $e_E$ we would have the additional factor
of $\prod_{j, k = 1}^n e_E(\zeta_j) \otimes e_E(\zeta_k)^*$ on
both sides. If we then contract with the Hermitian metric $h =
e^{- \phi}$ we would multiply both sides by $ e^{- \sum_{j = 1}^n
\phi(\zeta_j)}$.

\begin{proof}

This follows from the general formula that if $\{v_j\}$ is a basis
of an inner product space $V$, then
\begin{equation} \label{INNPROD} \left| v_1 \wedge \cdots
 \wedge v_n \right|^2 = \det \left( \langle v_j, v_k \rangle \right). \end{equation}

We consider a basis of coherent states $\Phi^{\zeta_k}$ for
$H^0(X, E)$ ($k = 1, \dots, n)$ with respect to $G$. Then
$\psi_j(\zeta_k) = \langle \psi_j,  \Phi^{\zeta_k} \rangle_G$. We
define the matrices,
\begin{equation} M := \left( \langle \Phi^{\zeta_k},
\Phi^{\zeta_{\ell}} \rangle \right) = \left( B_G(\zeta_j,
\zeta_{\ell} \right), \;\;\; A = \left( A_j^k \right), \;\; \psi_j
= \sum A_j^k \Phi^{\zeta_k}.
\end{equation} Then
$$\det \left( \langle \psi_j, \psi_k \rangle_G \right) = \det A^*
A \det M = |\det A|^2 \det M. $$ Also,
$$\left(\langle \psi_j,  \Phi^{\zeta_k} \rangle_G \right) = A M.
$$
It follows that
$$\frac{|\det \left(\langle \psi_j,  \Phi^{\zeta_k} \rangle_G
\right)|^2}{\det \left( \langle \psi_j, \psi_k \rangle_G \right)}
= \frac{|\det A|^2 |\det M|^2}{|\det A|^2 \det M} = \det M. $$

  \end{proof}

We now consider the determinant  $ \det \begin{pmatrix}
\psi_n(\zeta_j)
\end{pmatrix}_{j,n = 1}^N$ in Lemma \ref{SLATER1}.  Note that this
matrix omits the column $n = 0$. Hence, the relevant Bergman
kernel is the conditional Bergman kernel for $H^0_{P_0}(X,
\lcal_{N + g})$ with respect to the admissible metric and measure
$d\nu$ (see  \S \ref{BERGMANSECTION}). Since this is not a
standard object, we use the isomorphism  of this conditional space
with  $H^0(X, E_{N + g - 1})$ with the induced inner product. Then
we can use the Bergman kernel for this space of sections. On the
other hand, its basis $\psi_j^E$ of sections differ from those
$\psi_j$  of $\lcal_{N + g}$ by the factor ${\bf 1}_{\ocal(P_0)}$,
which will show up in its Slater determinant.

 \begin{cor}\label{SLATEREa} Let $\{\hat{\psi}_j\}$ be the orthonormal basis of $H^0_{P_0}(X, \lcal_{N + g})$
 and let $B_N$ be the conditional Bergman kernel for the admissible inner product. Let $G$ be the isometric inner product on $H^0(X,
 E_{N  + g - 1}) \simeq H^0_{P_0}(X, \lcal_{N - g})$ and let
 $B_{E_{N + g -1}}$ be its Bergman kernel.
 Then, the Slater determinant in the denominator of Proposition
 \ref{BASICFORMULA} is given by,
$$\begin{array}{lll} \left| \det \left( \psi_j(\zeta_k) \right)_{j, k = 0}^{N - 1} \right|^2
& = & \left| \prod_{k = 1}^N {\bf 1}_{P_0}(\zeta_k) \right|^2
 \left| \det \left( \psi_j^E(\zeta_k) \right)_{j, k = 0}^{N - 1}  \right|^2 \\ && \\
 &=&
\;\; \left| \prod_{k = 1}^N {\bf 1}_{P_0}(\zeta_k) \right|^2
 \det \left(B_{E_{N + g -1}}(\zeta_j, \zeta_k) \right)_{j, k = 1}^N  \det \left( \langle \psi_j, \psi_k \rangle_{G} \right).
 \end{array} $$
\end{cor}

\subsection{\label{PFREM} Projective linear ensemble: Proof of I of Theorem \ref{JPDHG}}

Put:
\begin{equation} \label{FCALNDEF} \fcal_N(\zeta_1, \dots, \zeta_N,
P_0) = ||\Phi_{h^N}^{P_0}||^{- 2} \left|\langle S_{\zeta_1, \dots,
\zeta_N}, \Phi_{h^N}^{P_0} \rangle  \right|^{ 2 }.
\end{equation}
We now state I of Theorem \ref{JPDHG} in a more precise way and
complete the proof:

\begin{maintheo} The pull-back under $\psi_{\lcal}: X^{(N)}
\backslash X^{(N)}_{N + g} \to \PP H^0(X, \lcal_{N + g})$ (resp.
$\sigma_{\lcal}: \tilde{X}^{(N)} \to \PP H^0(X, \lcal_{N + g}))$
of the Fubini-Study volume form, is given in uniformizing
coordinates on the cover $X^N$ by,
$$\begin{array}{lll} (I)\;\; \vec K^N_{PL}(\zeta_1, \dots, \zeta_N)  & = & \frac{1}{Z_N(h)} \frac{\fcal_N(\zeta_1, \dots, \zeta_N, P_0) \prod_{k = 1}^N \left|
\; \prod_{j = 1}^g E(P_j, \zeta_k)   \cdot \prod_{j: k \not= j}
  E(\zeta_j, \zeta_k)\right|^2}{ \; \det
  \left(B_N (\zeta_j, \zeta_k) \right)_{j, k = 1}^N} \prod_{j = 1}^N d\zeta_j \wedge d \bar{\zeta}_j  \\ && \\
  && \times \left(\int_X \left| \prod_{j = 1}^g E (P_j, z) \right|_{h_g}^2 \cdot \left| \prod_{j = 1}^N E (\zeta_j, z)\right|^2_{h^N} d\nu(z) \right)^{- N - 1}
  \end{array},$$
  where $ P_1 + \cdots + P_g = A_{\lcal_{N + g}} (\zeta_1 + \cdots + \zeta_N)$.
(defined in \S \ref{AJSECT}) and $Z_N(h)$ is a normalizing
constant so that $\vec{K}^N$ has mass one). It extends to a smooth
form on extended configuration space $\tilde{X}^{(N)}$.
\end{maintheo}

The proof   follows   Proposition  \ref{VOLAd},
  Proposition \ref{BASICFORMULA} and Corollary \ref{SLATEREa}.
That is, we first use \begin{eqnarray} \vec{K}^N_{PL}  & = &
||\Phi_{h^N}^{P_0}||^{- 2(N + 1)} \left|\langle S_{\zeta_1, \dots,
\zeta_N}, \Phi_{h^N}^{P_0} \rangle  \right|^{2 N + 2 } \;\;
\frac{\prod_{j = 1}^{N } d\ecal_j \wedge d
\bar{\ecal}_j}{||\prod_{j = 1}^g {\bf 1}_{\ocal(P_j)}(z) \cdot
\prod_{j = 1}^N {\bf 1}_{\ocal(\zeta_j)}||_{L^2(G_{N + g}(\omega,
\nu))}^{2(N + 1)}}.
 \end{eqnarray}

 Second we use
 $$\det \left( \begin{pmatrix} \frac{\partial
\ecal_{N - j}}{\partial \zeta_k}
\end{pmatrix}_{j, k = 1}^{N } \right)   = || \Phi_{h^N}^{P_0}||^{ N} \frac{\prod_{k = 1}^N
\prod_{j = 1}^g {\bf 1}_{\ocal(P_j)}(\zeta_k) }{\langle
S_{\zeta_1, \dots, \zeta_N}, \Phi_{h^N}^{P_0} \rangle^N \; \det
\begin{pmatrix} \psi_n(\zeta_j) \end{pmatrix}} \;
\left(\prod_{k = 1}^N \prod_{j: k \not= j} {\bf
1}_{\ocal(\zeta_j)} (\zeta_k)\right).
$$
Here it is assumed that $\{\psi_n\}$ is orthonormal. Finally, we
take the norm square and simplify.

It is a smooth form on extended
  configuration space since it is the pullback of a smooth form
  under a smooth map.

\subsection{\label{BERGZEROS} Analysis of zeros}

Since $\vec K_{PL}^N$ is a smooth form,  the zeros of the
denominator must be cancelled by the zeros of the numerator in the
expression (I) of Theorem \ref{JPDHG}. We now verify this as a
check on the calculation. We begin with the Bergman determinant:

\begin{lem} \label{BERGZEROSLEM} The zero set of  $\det \left(B_{E_{N + g -1}}(\zeta_j, \zeta_k) \right)_{j, k = 1}^N$ on $X^{(N)}$
consists of the `diagonals' $\zeta_j = \zeta_k$, together with the
$\zeta_1 + \cdots + \zeta_N$ such that $[E_{N + g - 1}] - (\zeta_1
+ \zeta_2 + \cdots + \zeta_N) \in \Theta = W^1_{g -1}$, i.e. has a
one-dimensional space of holomorphic sections. Multiplication by
${\bf 1}_{P_0}$ maps this subspace to a one-dimensional subspace
of $H_{P_0}(X, \lcal_{N + g})$ vanishing at $P_0, \zeta_1, \dots,
\zeta_N$.

\end{lem}

We will denote this set by $X^{(N)}_{N + g - 1}$, analogously to
$X^{(N)}_{N + g}$. Thus we now have three `bad' sets:

\begin{itemize}

\item The branch locus  $\bcal^{(N)}_{N + g}$  of the map $\PP
\tilde{\ecal}_N \to \PP H^0(X, \lcal_{N + g})$.

\item   $X^{(N)}_{N + g}$, the set  where the representation
$\lcal - \zeta_1 + \cdots + \zeta_N = P_1 + \cdots + P_g$ fails to
be unique.

\item  $X^{(N)}_{N + g -1}$ where $E_{N + g -1} - \sum_{j = 1}^N
\zeta_j = Q_1 + \cdots + Q_{g-1}$, or equivalently $\lcal -
\zeta_1 + \cdots + \zeta_N = P_0 + Q_1 + \cdots + Q_{g - 1}$. The
`extra zeros' of the Bergman determinant are in $X^{(N)}_{N + g
-1}$.

\end{itemize}

\begin{proof}


 Let $E$ be a line bundle of degree $N + g -
1$ and recall the Abel type maps,
\begin{equation} \label{FALAR} A_{E, g - 1}: X^{(N)} \to \mbox{Pic}^{g - 1}(X), \;\;
A_{E, g-1} (P_1 + \cdots +  P_N) : = \ocal (E) (\ocal (P_1 +
\cdots + P_N))^{-1}.
\end{equation} As mentioned above, in  $\mbox{Pic}^{g-1}$ there is the divisor $ \Theta$
of bundles which have global holomorphic sections. Given a base
point it can be identified with the Wirtinger variety $W_{g -1}
\subset \mbox{Jac}(X)$ under the further tensor product by
$\ocal(- (g-1) P_0)$.  If $\zeta_1 + \cdots + \zeta_N \notin A_{E,
g-1}^{-1} (\Theta)$ then the map $s \in H^0(X, E) \to \bigoplus_{j
= 1}^N s( \zeta_j)$ defines an isomorphism
\begin{equation} \label{ISOMORPHISM} H^0(X, E) \simeq \bigoplus_{j = 1}^N E [\zeta_j].
\end{equation}
Indeed, the  map fails to be an isomorphism if and only if there
exists a non-zero section $s$ such that $\dcal(s) \geq \zeta_1 +
\cdots + \zeta_N$. Then $\frac{s}{\Pi_{j = 1}^N {\bf
1}_{\ocal(\zeta_j)}} \in H^0(X, \ocal (E) (\ocal (\zeta_1 + \cdots
+ \zeta_N))^{-1})$ is a non-zero section and $\ocal (E) (\ocal
(\zeta_1 + \cdots + \zeta_N))^{-1}) \in \Theta$.  Conversely, if
$\in H^0(X, \ocal (E) (\ocal (\zeta_1 + \cdots + \zeta_N))^{-1})$
has a non-trivial section $e_{g -1}$ then $e_{g-1} \prod_{j = 1}^N
{\bf 1}_{\ocal(\zeta_j)} \in H^0(X, E_{N + g -1})$ is a section
vanishing at all $\zeta_j$.

Now let $E = E_{N + g - 1} = (N + g - 1) P_0$ and let
$\{\Phi_{G}^{\zeta_j}\}_{j = 1}^N$ be the set of coherent states
in $H^0(X, E_{N + g-1}) \simeq H^0_{P_0}(X, \lcal_{N + g})$
centered at the points $\{\zeta_j\}$. Then the Gram matrix of
inner products of these coherent states is
$$( \langle \Phi_{G}^{\zeta_j}, \Phi_{G}^{\zeta_k} \rangle_G ) = (
B_G(\zeta_j, \zeta_k) ). $$ It follows that the zeros of the
Bergman determinant are the points $\{\zeta_1, \dots, \zeta_N\}$
such that the coherent states $\Phi_G^{\zeta_j}$ fail to be
linearly independent or equivalently such that the evaluation map
in
 (\ref{ISOMORPHISM}) fails to be an isomorphism, and so
 $\{\zeta_1, \dots \zeta_N\} \in A_{E_{N + g -1}, g}^{-1}
 (\Theta)$.

 In this case, $H^0(X, E_{N + g-1}) \simeq H^0_{P_0}(X, \lcal_{N + g})$ is
 one-dimensional, hence there exists a one-dimensional space of
 $s \in H^0(X, E_{N + g - 1})$ vanishing at all $\zeta_j$.
The last statement is then obvious.

\end{proof}

\begin{lem} \label{ALLZEROSCANCELLED} As functions on  extended configuration space $\tilde{X}^{(N)} $, the zeros of $\det \left(B_{E_{N + g -1}}(\zeta_j,
\zeta_k) \right)$ are cancelled by  those of the numerator
$$\left|\prod_{k = 1}^N \prod_{j = 1}^g {\bf
1}_{\ocal(P_j)}(\zeta_k) \left(\prod_{k = 1}^N \prod_{j: k \not=
j} {\bf 1}_{\ocal(\zeta_j)} (\zeta_k)\right)\right|^2. $$

\end{lem}

\begin{proof}

 Given   $\zeta_2, \dots, \zeta_{N}$,
$\det
  \left(B_{E_{N + g -1}} (\zeta_j, \zeta_k) \right)_{j, k = 1}^N$,
    vanishes at $\zeta_1 = \zeta_2, \dots, \zeta_{N }$ and at $ g  $ further
points $\zeta_1 +  Q_1 + \cdots + Q_{g -1}$ so that $[E_{N + g
-1}] = \sum_{j = 2}^N \zeta_j + \zeta_1 +  Q_1 + \cdots + Q_{g
-1}. $ This follows from Lemma \ref{BERGZEROSLEM}. Vanishing of
the Bergman determinant is equivalent to vanishing of the Slater
determinant $(S_k(\zeta_j)$ for some (hence any) basis $\{S_k\}$
of $H^0(X, E_{N + g - 1})$. Suppose there exists a section $S$
vanishing at $\zeta_2 \cdots + \zeta_N + \zeta + Q_1 + \cdots +
Q_{g - 1}$. Putting $S_1 = S$ shows that the Slater determinant
vanishes at any $N$-element subset of this configuration.

Viewed as a function of $\zeta_1$ (or any other index), the
product $$ \left(\prod_{k = 1}^N \prod_{j: k \not= j} {\bf
1}_{\ocal(\zeta_j)} (\zeta_k)\right) \prod_{k = 1}^N \prod_{j =
1}^g {\bf 1}_{\ocal(P_j)}(\zeta_k) $$ also vanishes at any
$N$-element subset of this configuration. Indeed, adding $P_0$ to
a  configuration from $X^{(N)}_{N + g - 1}$ produces a special
type of configuration from $X^{(N)}_{N + g}$ with $Q_1 + \cdots +
Q_{g -1 } + P_0 = P_1 + \cdots + P_g$.  This representation is
unique since $\dim H^0(X, \ocal(Q_1 + \cdots + Q_{g-1})) = 1$. It
follows that $[\lcal_{N + g}] = \sum_{j = 1}^N \zeta_j + P_0  +
Q_1 + \cdots + Q_{g - 1}. $.
\end{proof}

\begin{rem} The factor $ \left(\prod_{k = 1}^N \prod_{j: k \not= j} {\bf
1}_{\ocal(\zeta_j)} (\zeta_k)\right)$ cancels all of the
`coincidence zeros' and indeed it has two factors of $\zeta_j -
\zeta_k$ for each $j \not= k$ while the Bergman determinant has
one. This effectively leaves a product $\prod_{j < k} {\bf
1}_{\ocal(\zeta_j)} (\zeta_k)$.

The factor $ \prod_{k = 1}^N \prod_{j = 1}^g {\bf
1}_{\ocal(P_j)}(\zeta_k)$ always vanishes when $\zeta_k = P_j$ for
some $j, k$, including when $\zeta_1 + \cdots + \zeta_{N } \notin
X_{N + g -1}^{(N)}$, i.e.  when $(N + g) P_0 - \sum_{j = 1}^N
\zeta_j = P_1 + \cdots + P_g$ but there does not exist $Q_1 +
\cdots + Q_{g -1}$ such that $(N + g ) P_0 - \sum_{j = 1}^N
\zeta_j = P_0 + Q_1 + \cdots + Q_{g-1}$, i.e. when no $P_j$ equals
$P_0$.   In this sense it has `more zeros' than necessary to
cancel the zeros of the Bergman determinant. One may view the
extra zeros as arising from the degeneracy of the map $D
\psi_{\lcal_{N + g}}$ in Proposition \ref{CANONICAL}.

\end{rem}

\section{\label{JPCCALCULATIONII} JPC and  Green's functions:
Proof of II of Theorem \ref{JPDHG}}

In this section, we rewrite the formula for the JPC of the
projective linear ensemble  in terms of  Green's functions. In
particular, we rewrite the Bergman (or Slater) determinant as  a
product of values of the Green's function. This step is crucial to
obtain the large deviations rate functional. The product formula
for the Slater/Bergman determinant follows from the bosonization
formula of \cite{ABMNV,VV,Fal,F}.

To state the results, we need some further notation. We put
$\frac{\partial}{\partial z} = \frac{1}{2}
(\frac{\partial}{\partial x} - i \frac{\partial}{\partial y} ),
\frac{\partial}{\partial \bar{z}} = \frac{1}{2}
(\frac{\partial}{\partial x} + i \frac{\partial}{\partial y} ), $
and  $\partial f = \frac{\partial f}{\partial z} d z$;  similarly
for $\dbar f$. Then $\ddbar = \frac{\partial^2}{\partial z
\partial \bar{z}} d z \wedge d \bar{z}.$
We also put  $d =
\partial + \dbar, d^c := \frac{i}{4 \pi} (\dbar -
 \partial)$ and   $dd^c = \frac{i}{2\pi}
 \ddbar$. Thus, $dd^c f = \frac{ i}{8 \pi} \Delta f dz \wedge d\bar{z} = \frac{1}{4 \pi} \Delta f dx \wedge dy, $
where $\Delta$ is the standard Euclidean Laplacian, and
\begin{equation} \label{FUNDSOL} \Delta \left(\frac{1}{2 \pi} \log
|z|\right)  = \delta_0 \iff dd^c (2 \log |z|) = \delta_0 dx \wedge
dy.
\end{equation}

\subsection{\label{GRF} Green's function with respect to a
Hermitian metric}

Given a real $(1,1)$ form $\omega$,  we  define
 the Green's
function $G_{\omega}$ relative to $\omega$ to be  the unique
solution $G_{\omega}(z, \cdot) \in \dcal'(X)$ of
\begin{equation} \label{GHDEF} \left\{ \begin{array}{ll} (i) &  dd^c_w G_{\omega} (z,w)   =  \delta_z(w) -
\omega_w, \\ \\ (ii) & G_{\omega}(z,w) = G_{\omega}(w, z), \\ &
\\(iii) & \int_{X} G_{\omega}(z,w) \omega_w = 0, \end{array} \right.
\end{equation}
where the equality in the top line is in the sense of $(1,1)$
forms. We refer to \cite{ZZ} for background  (see the proof of
Proposition 10); uniqueness follows from condition (iii).

 We note that in \cite{ZZ} the Green's function was denoted $G_h$
 with respect to a Hermitian metric $h$ on $\ocal(1)$. But the
 Green's function depends only on the curvature $(1,1)$ form of
 $h$, so we denote it here by $G_{\omega}$.

The Green's potential of a measure in higher genus is defined
precisely as in genus zero in  \cite{ZZ}. The Green's potential of
a measure  is uniquely characterized as the solution of
\begin{equation*} \left\{ \begin{array}{l}
dd^c U^{\mu}_{\omega} = \mu - \omega; \\ \\
 \int_{X} U^{\mu}_{\omega} \omega = 0.
\end{array} \right. \end{equation*}

As  in the genus zero case of \cite{ZZ}, the Green's function may
be expressed in
 terms of local holomorphic coordinates and a local potential $\phi$ for $\omega$ (i.e.  $\omega= dd^c
 \phi$).
A more invariant expression for $G_{\omega}$  is in terms of the
prime form (see \S \ref{POINT}) for background).

\begin{prop} \label{GREEN}  Let $(X, \omega)$ be a compact Riemann surface equipped a $(1,1)$ form $\omega$.
Let $\ocal(w)$ be the point line bundle and let ${\bf
1}_{\ocal(w)}(z)$ be its canonical section. Also, let $h_{w}$ be
an $\omega$-admissible metric on $\ocal(w)$.
 Then the Green's
function $G_{\omega}$ is given by
$$\begin{array}{lll} G_{\omega}(z,w) & = &  \log ||{\bf 1}_{\ocal(w)} (z)||_{h(\omega,
w)}^2
  - \frac{1}{\int_X \omega} \int_X \log || {\bf 1}_{\ocal(w)}
(z)||^2_{h_w(z)} \omega_z \\ && \\
& = & \log ||E(z, w)||_{h(z) \boxtimes h(w)}.

\end{array}$$
\end{prop}

\begin{proof}

We apply $ i \ddbar = 4 dd^c$ in the $z$ variable to both sides of
the formula.
 On the left, we obtain $\delta_w(z) \omega - \frac{1}{A} \omega$,
where $A = \int_X \omega$. On the right side, we write ${\bf
1}_{\ocal (w)} = (z - w) e$ relative to  a local holomorphic
function $e$ of $\ocal(w)$ near $w$. Then in the $z$ variable,
$$dd^c \log ||{\bf 1}_{\ocal(w)} (z)||_{h(\omega, w)}^2 =
\delta_w(z) - dd^c ||e(z)||_{h(\omega, w)}^2 = \delta_w(z) -
\omega,
$$ since  $\phi_w(z) = \log ||e(z)||_{h(\omega, w)}^2$ is a potential for
$\omega$. Note that  $\omega_{h(z; \omega, w)} = \frac{1}{A}
\omega$, since $\omega_{h(z; \omega, w)} $ is the curvature of a
line bundle of Chern class one and is harmonic with respect to
$\omega$. Hence after taking $dd^c$ the integral vanishes and  we
get $\delta_w(z) \omega - \frac{1}{A} \omega$ on the right side.

It follows that $dd^c_z \left(G_{\omega}(z,w) - \log ||{\bf
1}_{\ocal(w)} (z)||_{h(\omega, w)}^2 \right) = 0$ and since both
are globally well-defined their difference must be a constant,
possibly depending on $w$. The constant  is determined by the
condition that $\int_X G(z,w) \omega_w(z) = 0$ for all $w$.

\end{proof}

\begin{rem}  (i) It is not obvious that $G(z,w) = G(w, z)$ from this
formula, but this must be the case since the Green's function
satisfying (i) and (iii) is unique.
\medskip

(ii)  The admissible Hermitian metric is only defined up to a
multiplicative constant. However, the right side of the formula is
independent of the choice of the constant. \medskip

(iii) Since the genus $g \geq  1$, we may also  express $X =
\tilde{X}/\Gamma$ where $\tilde{X}$ is the universal holomorphic
cover. Then $G$ is a $\Gamma$-invariant solution of (i) - (ii) on
$\tilde{X}$, and (iii) holds for the integral over a fundamental
domain $\fcal$ for $\Gamma$. Indeed, $dd^c (G_{\omega} - 2 \log |z
- w|) = \omega$ where derivatives are taken in either the $z$ or
$w$ variable, and so $G_{\omega} - 2 \log |z - w|$ is a smooth
potential for $\omega$. If we subtract $\phi(z) + \phi(w)$, then
$dd^c$ of the result equals zero in each variable. Hence the
difference must be a constant and the constant is determined by
(iii).\end{rem}

\subsection{Mean value of Green's function}

Henceforth, we define
\begin{equation} \label{rho} \rho_{\omega}(w): = \frac{1}{\int_X \omega} \int_X \log || {\bf 1}_{\ocal(w)}
(z)||^2_{h_w(z)} \omega_z. \end{equation}

Let us relate this to the analogous `constant' when $g = 0$. In
\cite{ZZ}, we defined $ E(h): =  \left( \int_{\CP^1} \phi(z)
\omega_h + 4\pi \rho_{\phi}(\infty) \right),$ where  $\rho_{\phi}$
was  a certain Robin constant. In Lemma 8  of \cite{ZZ} we showed
that in each trivializing affine chart of $\ocal(1) \to \CP^1$,
and relative to the frame $e(z)$ over the affine chart $\C$ in
which
 $h = e^{- \phi}$ and $\omega_h = dd^c \phi$, the Green's function has the local expression,
$
    G_{j}(z,w) = 2 \log
|z - w| - \phi_j(z) - \phi_j(w) +  E(h),$  and $\int_{\C} G_j(z,w)
dd^c \phi_j = 0$. In particular, we showed that $ \int_{\CP^1}
\log || z - w||^2_{h_z \boxtimes h_w} \omega_h $ is a constant in
$z$, equal to $E(h)$.

In higher genus, we also have $ \log || {\bf 1}_{\ocal(w)}
(z)||^2_{h_w(z)} = \log |z - w| - \phi(z) - \phi(w)$ relative to
the canonical frame of $\ocal(w)$. We could use Green's formula on
a fundamental domain $\fcal$   to simplify the formula for
$\rho_{\omega}(w)$ and also express it in terms of the potential.
But  we do not need the expression, so we leave the details to the
reader.

\subsection{Norms of canonical sections and Green's functions}

We now tie together the Vieta maps with the Green's potentials of
the discrete measures
 $\mu = \mu_{\zeta}$.

\begin{lem} \label{NORMSUPHI}  Let $h$ be an $\omega$-admissible
metric. Then with $\rho_{\omega}$ as in (\ref{rho}),

\begin{itemize}

\item (i)\; $ \frac{1}{N} \log ||{\bf 1}_{\zeta_1 + \cdots +
\zeta_N}(z)||_{h^N} - \frac{1}{N} \int_{X} \log ||{\bf 1}_{\zeta_1
+ \cdots +
\zeta_N}||_{h^N}^2 \omega = U_{\omega}^{\mu_{\zeta}}(z). $ Hence,\\
\\$ ||{\bf 1}_{\zeta_1 + \cdots +
\zeta_N}(z)||_{h^N}^{\frac{1}{N}} e^{ - \frac{1}{N} \int_{X} \log
||{\bf 1}_{\zeta_1 + \cdots + \zeta_N}||_{h^N}^2 \omega} =
e^{U_{\omega}^{\mu_{\zeta}}}. $\\

\item (ii) \; $   \int_X \log ||{\bf 1}_{\zeta_1 + \cdots +
\zeta_N} (w)||_{h^N}^2  \omega = \sum_{j = 1}^N
\rho_{\omega}(\zeta_j). $\\

\item (iii)\; $ \frac{1}{N} \log ||S_\zeta(z)||_{h^N} -
\frac{1}{N} \int_{X} \log ||S_\zeta||_{h^N}^2 \omega =
U_{\omega}^{\mu_{\zeta}}(z) + \frac{1}{N}
U_{\omega}^{\mu_{P(\zeta)}}(z). $ Hence,\\ \\$
||S_\zeta(z)||_{h^N}^{\frac{1}{N}} e^{ - \frac{1}{N} \int_{X} \log
||S_\zeta||_{h^N}^2 \omega_h} = e^{U_{h}^{\mu_{\zeta}} +
\frac{1}{N}
U_{\omega}^{\mu_{P(\zeta)}}}. $\\

\item (iv) \; $   \int_X \log ||S_{\zeta} (w)||_{h^N}^2  \omega_h
= \sum_{j = 1}^N \rho_{\omega}(\zeta_j) + \sum_{j = 1}^g
\rho_{\omega}(P_j(\zeta). $

\end{itemize}
\end{lem}

\begin{proof}  Since $\log ||{\bf 1}_{\zeta_1 + \cdots
+ \zeta_N}||_{h^N}^2 = \sum_{j = 1}^N \log ||{\bf 1}_{\zeta_j}
||_{h}^2, $ it follows from  Proposition \ref{GREEN} that
$$d\mu_{\zeta} = dd^c \frac{1}{N}  \log ||{\bf 1}_{\zeta_1 + \cdots
+ \zeta_N}||_{h^N}^2 + \omega, $$ and that
$$d\mu_{\zeta} + \frac{1}{N} d\mu_{P(\zeta)} = dd^c \frac{1}{N}  \log ||{\bf 1}_{\zeta_1 + \cdots
+ \zeta_N} {\bf 1}_{P_1 + \cdots + P_g}||_{h^{N + g}}^2 + (1 +
\frac{g}{N}) \omega,
$$
it follows by integration against $G_{\omega}(z,w)$ that
$$\begin{array}{lll}    U_{\omega}^{\mu_{\zeta}}(z) &: =&
    \int_{X} G_{\omega}(z,w) d\mu_{\zeta}(w)\\ && \\ & = &
\int_{X} G_{\omega}(z,w)  \left( dd^c \frac{1}{N}  \log ||{\bf
1}_{\zeta_1 + \cdots + \zeta_N}||_{h^N}^2 + \omega \right) \\ && \\
& = & \int_{X} G_{\omega}(z,w) \left( dd^c \frac{1}{N}  \log
||{\bf 1}_{\zeta_1 + \cdots
+ \zeta_N}||_{h^N}^2  \right) \\ && \\
& = & \frac{1}{N} \log ||{\bf 1}_{\zeta_1 + \cdots +
\zeta_N}||_{h^N}^2(z) - \frac{1}{N} \int_{X} \log ||{\bf
1}_{\zeta_1 + \cdots + \zeta_N}||_{h^N}^2(z) \omega.
\end{array}$$

\end{proof}

\subsection{\label{BOSON} Bosonization formulae}

Bosonization formulae equate the determinantal correlation
functions of a free fermionic field theory in two dimensions to
the symmetric  correlation functions of a corresponding bosonic
field theory. For our purposes, the important point is that the
bosonization formulae express Slater or Bergman determinants in
terms of products of prime forms (or exponentials of Green's
functions). We give further background in the Appendix \S
\ref{APPENDIX}.

With some adjustments to make the notation consistent with this
article, formula (5.4) of \cite{ABMNV} states the following:

\begin{maintheo} \label{BOSONa} Let $L \to X $ be the line bundle of degree $d$  of the divisor
$\sum_{j = 1}^d a_j$. Let $\{\psi_j\}_{j = 1}^p$ be a basis of
$H^0(X, L)$ where $p = \dim H^0(X, L) = d + 1 - g$.  Then there
exists a constant $A_N(\omega, g)$   such that

\begin{equation}  \begin{array}{ll}  \left| \det
\begin{pmatrix}
\psi_1(\zeta_1) & \cdots & \psi_{p } (\zeta_1)  \\  & \\
\;\psi_1(\zeta_p) & \cdots & \psi_{p } (\zeta_p)  \end{pmatrix}
\right|^2 & = A_N(g, \omega) \;\;\prod_{i < j} |E(\zeta_i,
\zeta_j)|^2\\ & \\ & \cdot |\det (\langle \psi_i, \psi_j
\rangle)|^2 \ncal(L \otimes \ocal(- \sum_{j = 1}^p \zeta_j)
\otimes L_{\acal}^{-1} )
\end{array}\end{equation}

\end{maintheo}

   The constant $A(g, \omega)$ is a quite complicated
constant,  independent of $\{\zeta_j\}$ and the choice of
$\{\psi_k\}$, involving a ratio of determinants of two Laplacians
multiplied by some extra factors involving  $\theta$-functions, a
metric $g$ on $X$ and spin structures. Also, $L_{\acal}$ is a spin
structure (a bundle of degree $g - 1$) depending on a choice of
homology basis $\acal$ of $X$, and $\ncal(L \otimes \ocal(- \sum_j
\zeta_j) \otimes L_{\acal}^{-1})$ is defined by
\begin{equation} \label{NCAL} \ncal_{\acal}(z) = e^{- 4 \pi i \langle Y y, y \rangle}
|\vartheta( z | \tau)|^2,
\end{equation}
where $z$ is the divisor class $z = L \otimes \ocal(- \sum_j
\zeta_j) \otimes L_{\acal}^{-1}$. In addition, $\tau$ is the
period matrix  and $Y$ is the matrix of inner products of the
holomorphic one forms (we refer to \cite{ABMNV} for further
details).

The $\vartheta$-factor is quite important since it cancels the
extra  zeros. The rest of the  details of the formula are mostly
irrelevant for our purposes, since they are of lower logarithmic
order than the LDP rate function.

We apply the bosonization formula to the line bundle $E_{N + g -
1}$ of (\ref{EDEF}) of degree $d = N + g -1$. Hence $p = N$ and
the relevant theta function is
$$\theta ((N + g -1) P_0 - \sum_{j = 1}^N \zeta_j - \Delta). $$
Here we suppress $\tau$ since the complex structure is fixed in
our discussion.

\subsection{\label{JPCCALCULATIONIIsub} Completion of the
proof of II of Theorem \ref{JPDHG}}

To state the next result, we need some further notation.
 Motivated by formula I of Theorem \ref{JPDHG} and by  the detailed
bosonization formula (see \S \ref{BOSON}, Theorem \ref{BOSONa},
and the Appendix \S \ref{APPENDIX}), we put
\begin{equation} \label{FNDEF} \begin{array}{lll} F_N(\zeta_1, \dots, \zeta_N) : & = &
A(N, g)   e^{ -(N + 1) \sum_{j = 1}^g \rho_{\omega}(P_j(\zeta))}
 \fcal_N(\zeta_1, \dots, \zeta_N, P_0) \left| \prod_{k = 1}^N {\bf 1}_{P_0}(\zeta_k) \right|^{-2}  \\
&& \\ &&  \left |\left(\prod_{k = 1}^N \; \prod_{j = 1}^g E(P_j,
\zeta_k)\right) \right|^2 \left|  \left(\; \theta ((N + g - 1) P_0
- \sum_{i = 1}^N \zeta_i - \Delta)\right)\right|^{-2},
\end{array}
\end{equation}
where $\fcal_N$ is defined in (\ref{FCALNDEF}). Here we use that
$|\det (\langle \psi_i, \psi_j \rangle)|^2 = 1$. We view $\fcal_N$
as defined on $X^{(N)}  \backslash X^{(N)}_{N + g}$ or on extended
configuration space $\tilde{X}^{(N)}$. In the latter case we write
$F(\zeta_1 + \cdots + \zeta_N, P_1 + \cdots + P_g)$.

\begin{lem} \label{JPDINV}
    Let $h = e^{- \phi}$ be a
    smooth Hermitian metric on $\ocal(P_0)$, and let
    $\omega_h, G_{\omega}$ be as above.
    Then, the  joint probability current for the projective linear ensemble  is given by:
$$\begin{array}{lll}   \vec K^N(\zeta_1, \dots, \zeta_N) & = &
 \frac{1}{Z_N(\omega)}
    \frac{\exp \left(\frac{1}{2} \sum_{i \not=
j} G_{\omega}(\zeta_i, \zeta_j) \right) }{\left(\int_{X} e^{ N
\int_{X} G_{\omega}(z,w) d\mu_{\zeta}(w)} e^{ \int_{X}
G_{\omega}(z,w) d\mu_{P(\zeta)}(w)} d\nu(z) \right)^{N+1}} \\ &&
\\ && \; \cdot  F_N(\zeta_1, \dots, \zeta_N)  e^{- 2 N \int
\rho_{\omega}(w) d\mu_{\zeta}(w)},
\end{array}$$
where $\rho_{\omega}$ is defined in (\ref{rho}), where $F_N$ is
defined in (\ref{FNDEF}), and where $Z_N$ is a normalizing
constant
 (see \S \ref{APPENDIX}).

\end{lem}

\begin{proof}

We first rewrite  the numerator \begin{equation} \label{NUM}
\frac{ \prod_{k = 1}^N \left| \; \prod_{j = 1}^g E(P_j, \zeta_k)
\cdot \prod_{j: k \not= j}
  E(\zeta_j, \zeta_k)\right|^2}{\det \left(B_N (\zeta_j, \zeta_k) \right)_{j, k =
  1}^N} \end{equation}
of the expression I in Theorem \ref{JPDHG} in terms of the Green's
function.  It is a product of three factors:
\begin{itemize}

\item (i)\; $\prod_{k = 1}^N \left| \;    \prod_{j: k \not= j}
E(\zeta_j, \zeta_k)\right|$;
\medskip

\item (ii) \; $\prod_{k = 1}^N \left| \; \prod_{j = 1}^g E(P_j,
\zeta_k) \right|^2$;
\medskip

\item (iii) \; $(\det \left(B_N (\zeta_j, \zeta_k) \right)_{j, k =
  1}^N)^{-1}$.

\end{itemize}

By Proposition \ref{GREEN} and Lemma \ref{NORMSUPHI},  (i) equals

  \begin{equation} \label{EG} \left|\prod_{j: k \not= j = 1}^N \;
 {\bf 1}_{\ocal(\zeta_j)}(\zeta_k))\right|^2 =  \exp \left( \sum_{i < j}^N G_{\omega}(\zeta_i,
\zeta_j) \right) \exp\left(   -
 (N - 1) \sum_{j = 1}^N \rho_{\omega}(\zeta_j) \right). \end{equation}
 Hence,
\begin{equation}\begin{array}{lll} 2  \sum_{i < j}^N\; \log \;\left|{\bf 1}_{\ocal(\zeta_j)}(\zeta_k))\right|^2 & = &  \sum_{i < j}^N G_{h}( \zeta_i,  \zeta_j)
   +
 (N - 1) \sum_{j = 1}^N \rho_{\omega}(\zeta_j))\\&& \\
& = & N^2 \int_{X \times X \backslash D}   G_{h}(z,
w)d\mu_{\zeta}(z) d\mu_{\zeta}(w) - N (N - 1) \int
\rho_{\omega}(w) d\mu_{\zeta}(w).
\end{array}
\end{equation}

To express (iii) in terms of Green's functions, we use the
bosonization formula (\ref{BOSON}) for the  Bergman determinant or
of  the Hermitian line bundle $E = E_{N + g -1}$. It gives:
\begin{equation} \begin{array}{lll} \left|\det \begin{pmatrix}
\psi_1^E(\zeta_1) & \cdots & \psi^E_{N } (\zeta_1)  \\  & \\
\;\psi_1^E(\zeta_N) & \cdots & \psi^E_{N } (\zeta_N)
\end{pmatrix} \right|^2 & = A_N(g, \omega) \;\;\prod_{i < j}
|E(\zeta_i, \zeta_j)|^2\\ & \\ & \cdot |\det (\langle \psi^E_i,
\psi^E_j \rangle)|^2 \ncal(E_{N + g -1} \otimes \ocal(- \sum_{j =
1}^N \zeta_j) \otimes L_{\acal}^{-1} )
\end{array}\end{equation}
Here, by (\ref{NCAL}),  up to an $N$-independent positive
constant,
$$\ncal(E_{N + g -1} \otimes \ocal(- \sum_{j =
1}^N \zeta_j) \otimes L_{\acal}^{-1} ) = || \theta ((N + g - 1)
P_0 - \sum_{i = 1}^N \zeta_i - \Delta)||^2. $$

Since
$$|\det \begin{pmatrix} \psi_n(\zeta_j) \end{pmatrix}|^2 = \left| \prod_{k = 1}^N {\bf 1}_{P_0}(\zeta_k) \right|^2
\det \left(B_{E_{N + g -1}}(\zeta_j, \zeta_k) \right)_{j, k =
1}^N,
 $$
we need to multiply the product side of the bosonization formula
by $\left| \prod_{k = 1}^N {\bf 1}_{P_0}(\zeta_k) \right|^{2}$.

Again using  Proposition \ref{GREEN} and Lemma \ref{NORMSUPHI}, we
then convert the denominator

$$\left(\int_X \left| \prod_{j = 1}^g E (P_j, z) \right|_{h_g}^2
\cdot \left| \prod_{j = 1}^N E (\zeta_j, z)\right|^2_{h^N} d\nu(z)
\right)^{- N - 1}$$ of  expression I in Theorem \ref{JPDHG} into
the Green's function expression by the identities
\begin{equation} \label{GRID}
\begin{array}{l}\int_X   e^{N \int_X \log  \left|{\bf 1}_{\ocal(w)}(z)\right|^2_{h_w} d\mu_{\zeta}(w)}
e^{\int_X \log  \left|{\bf 1}_{\ocal(w)}(z)\right|^2_{h_w}
d\mu_{P(\zeta)(w)}}
   d\nu(z)
\\ \\
= \left( \int_{X} e^{N \int_{X} G_{h}(z,w) d\mu_{\zeta}(w)}  e^{
\int_{X} G_{h}(z,w) d\mu_{P(\zeta)}(w)} d\nu(z) \right) e^{ N \int
\rho_{\omega}(w) d\mu_{\zeta}(w)   + \sum_{j = 1}^g
\rho_{\omega}(P_j(\zeta))}
\end{array} \end{equation}

We now raise the denominator (\ref{GRID}) to the power $ (N + 1)$
and combine with the numerator calculation so that the full
expression I of Theorem \ref{JPDHG} becomes
$$\frac{\exp \left( \sum_{i < j}
G_{h}(\zeta_i, \zeta_j) \right) }{\left(\int_{X} e^{ N \int_{X}
G_{h}(z,w) d\mu_{\zeta}(w) } e^{  \int_{X} G_{h}(z,w)
d\mu_{P(\zeta)}(w) } d\nu(z) \right)^{N+1}} \;\; F_N
$$ multiplied by the exponential of
$$ \begin{array}{l}  (N - 1) N \int \rho_{\omega} d\mu_{\zeta}-
  N (N + 1)  (\int \rho_{\omega}(w) d\mu_{\zeta}).
\end{array}$$
We  note the   cancellation in the $N^2$ term of $\int
\rho_{\omega} d\mu_{\zeta}$, leaving $- 2 N \int \rho_{\omega}
d\mu_{\zeta}$.  This gives the stated result.

\end{proof}

\subsection{\label{MOREZEROS} Analysis of zeros}

We now update \S \ref{BERGZEROS} to take into account the above
results using bosonization. We first re-consider Lemma
\ref{ALLZEROSCANCELLED} in the light of the bosonization formula.

\begin{lem} \label{ALLZEROSCANCELLEDb}  $F_N$ is holomorphic on $\tilde{X}^{(N)} $.  That is, $$\left(\prod_{k = 1}^N \; \prod_{j =
1}^g E(P_j, \zeta_k)\right)   \left(\; \theta ((N + g - 1) P_0 -
\sum_{i = 1}^N \zeta_i - \Delta)\right)^{-1}$$ has no poles.
Moreover, the same is true on extended configuration space
$\tilde{X}^{(N)}$.
\end{lem}

\begin{proof} We first note that the poles of $\left| \prod_{k = 1}^N {\bf 1}_{P_0}(\zeta_k)
\right|^{-2}$ are cancelled by the zeros of $\fcal_N$. Hence it
suffices to prove the second statement, which comes down to the
equivalent statement in Lemma \ref{ALLZEROSCANCELLED}.

 By
Riemann's vanishing theorem, the zeros of $\theta ((N + g -1) P_0
- \sum_{i = 1}^{N } \zeta_i - \Delta)$ occur at $\{\zeta_1, \dots,
\zeta_N\}$ such that   $[(N + g -1) P_0 - \sum_{i = 1}^{N }
\zeta_i ] = Q_1 + \cdots + Q_{g-1}$ is the divisor class of a line
bundle of degree $g -1 $ with a one-dimensional space of
holomorphic sections. Equivalently, $[(N + g) P_0 - \sum_{i =
1}^{N } \zeta_i ] = P_0 + Q_1 + \cdots + Q_{g-1}$ is the divisor
class of a line bundle of degree $g $ with a one-dimensional space
of holomorphic sections vanishing at $P_0$. But this implies that
$P_0 + Q_1 + \cdots + Q_{g -1} = P_1 + \cdots + P_g$ when the
representation is unique, so that $\left(\prod_{k = 1}^N \;
\prod_{j = 1}^g E(P_j, \zeta_k)\right)   $ vanishes when some
$\zeta_k$ equals some   $Q_j$.


\end{proof}

For $\zeta_1 + \cdots + \zeta_N \notin X^{(N)}_{N + g}$,
$A_N(\zeta_1 + \cdots + \zeta_N) = P_1 + \cdots + P_g$ and
$$(N + g - 1) P_0 - \sum_{j = 1}^N \zeta_j = (N + g) P_0 - \sum_{j
= 1}^N \zeta_j - P_0 = P_1 + \cdots + P_g - P_0,
$$ hence $$ \theta( (N + g - 1) P_0  - \sum_{j = 1}^N
\zeta_j - \Delta) =  \theta(P_1 + \cdots + P_g - P_0 - \Delta).
$$
The same formula is valid on $\tilde{X}^{(N)}$.

%
%
%
%

We sum up in the

\begin{mainprop} \label{FNFINAL}
There exists a constant $B(N, g)$ such that, as holomorphic
sections on $\tilde{X}^{(N)}$,
$$\begin{array}{lll} F_N(\zeta_1, \dots, \zeta_N, P_1  + \cdots + P_g) : & = &
B(N, g)  e^{ -(N + 1) \sum_{j = 1}^g \rho_{\omega}(P_j(\zeta))}
\left |\left(\prod_{k = 1 }^N \;\prod_{j = 1}^g E(P_j,
\zeta_k)\right) \right|^2 \\ && \\
& \cdot & ||  \theta(P_1 + \cdots + P_g - P_0 - \Delta)||^{-2}
\;\; \prod_{j = 1}^g |E(P_j, P_0)|^2 \; ||\Phi_{h^N}^{P_0}||^{- 2}

\end{array} $$ \end{mainprop}

We have used  Lemma \ref{BERGZEROS}, together with the
cancellation,

$$\begin{array}{lll}  \fcal_N(\zeta_1, \dots, \zeta_N)|\prod_{j = 1}^N E(P_0, \zeta_j)|^{-2}
& = & ||\Phi_{h^N}^{P_0}||^{- 2} \prod_{j = 1}^g |E(P_j, P_0)|^2.
\end{array}$$

\section{\label{JPCPLE} The JPC as a Chern form }

So far, we have concentrated in $\vec K_{PL}^N$. We now give
analogous formulae for the more complicated $\vec K^N_{FSH}$. We
do this by expressing $\vec K^N_{FSH}$ in terms of $\vec
K_{PL}^N$. To do this, it is helpful  to relate both forms to the
representative Chern form of $c_1(\zcal_N)^{\mbox{top}}$ of a
natural line bundle $\zcal_N \to X^{(N)}$.

\subsection{Line bundles over $X^{(N)}$}

We first make

\begin{maindefin} \label{ZCALNDEF} At $\{\zeta_1, \dots, \zeta_N\} \in X^{(N)}$, the
fibre of $\zcal_N$ is the line $\C {\bf 1}_{\ocal(\zeta_1+ \cdots+
\zeta_N)}$ of holomorphic sections of the line bundle
$\ocal(\zeta_1 + \cdots + \zeta_N)$,
\begin{equation} \label{ZCALN} \zcal_N \to X^{(N)}, \;\; (\zcal_N)_{\zeta} = \{[s] \in \ecal_N: \dcal(s) = \zeta \}, \end{equation}
\end{maindefin}

Recalling that $X^{(N)} \simeq \PP \ecal_N $, we may identify
$\zcal_N = \ocal_{\PP}(1) \to \PP (\ecal_N)$. We recall from \S
\ref{PICBUN} that this line bundle is ample.
Another line bundle which plays a role in computation of the JPC
is the twist
\begin{equation}  \zcal_N  \otimes  A_{N}^* \Theta \to X^{(N)},
\end{equation}
where $A_N: X^{(N)} \to \mbox{Jac}(X)$ is the Abelian sums map and
$\Theta \to \mbox{Jac}(X)$ is the standard theta-line bundle (see
\S \ref{THETASTUFF}).

\subsection{Formulae for $d \tau_N^{FSH}$ and $d \lambda_{PL}$}

The   Hermitian inner products $G_N(h, \nu)$ of (\ref{DEFGN})
induce  a Hermitian metric on $\zcal_N$,  corresponding to the
choice of $\omega_0$-admissible metrics on $\xi \in
\mbox{Pic}^N(X)$ (\S \ref{OPT}).  That is,  at a point $\zeta \in
X^{(N)}$, the norm of the vector ${\bf 1}_{\ocal(\zeta)} \in
\zcal_{\zeta}$ is $||{\b 1}_{\ocal(\zeta)}||_{G_N(h, \nu)}$. The
curvature $(1,1)$ form of the Chern connection is given by the
$(1,1)$ form,
\begin{equation} \label{OMEGAZCAL} \omega_{\zcal} (\zeta)= \frac{i}{2} \ddbar \log ||
{\bf 1}_{\ocal(\zeta)}||_{G(h, \nu)} \;\; \mbox{on}\;\; X^{(N)},
\end{equation} where $\ddbar$ is the operator on $X^{(N)}$, with
K\"ahler potential given by the pluri-subharmonic function  $\log
|| {\bf 1}_{\ocal(\zeta)}||_{G_N(h, \nu)}. $
 This choice is most natural geometrically,
but as in \S \ref{OPT}, we opt for the somewhat more complicated
Hermitian metric on $\zcal_N$ in Definition \ref{SZETA}, i.e.  the
Hermitian metric $|| S_{\zeta_1, \dots, \zeta_N}||_{G_{N + g}(h,
\nu)} = ||{\bf 1}_{\ocal(\zeta)} {\bf 1}_{\ocal(P(\zeta))}||_{G_{N
+ g}(h, \nu)}$. Its curvature $(1,1)$ form is given by,
\begin{equation} \label{OMEGAZCALb} \tilde{\omega}_{\zcal} (\zeta)= \frac{i}{2} \ddbar \log ||
S_{\zeta_1, \dots, \zeta_N}||_{G_{N + g}(h, \nu)}.
\end{equation}

Further, if we equip $\Theta \to \mbox{Jac}(X)$ with its standard
Hermitian metric $h_{\Theta}$ with curvature the Euclidean $(1,1)$
form $\omega_{\Theta} : = \sum dz_i \otimes d \bar{z}_i$ on
$Jac(X)$, then $\omega_{\zcal} + A_N^* \omega_{\Theta}$ is the
curvature form of $\zcal_N \otimes A_N^* \omega_{\Theta}$ (and
similarly for the $\tilde{\omega_{\zcal}}$ version).

\begin{mainprop}\label{JPCPROP} We have,

\begin{itemize}

\item The  Fubini-Study-Haar probability measure (Definition
\ref{FSHAARDEF}) is given by
\begin{equation} d\tau_N^{FSH} =  \tilde{\omega}_{\zcal}^{N-g} \wedge
A_N^* \omega_{\Theta}. \end{equation}

\medskip

\item The probability measure of the Fubini-Study-$X^{(g)}$
\begin{equation} d\tau_N^{FS X^{(g)}} =  \tilde{\omega}_{\zcal}^{N-g} \wedge
\rho_N^* d \sigma, \;\;\;\;\; (cf. (\ref{DIA}))
\end{equation}

\item The projective linear ensemble probability measure
$d\lambda_{PL}$  (Definition \ref{LARGEVS}) is given by
$\tilde{\omega}_{\zcal}^{N}$.
\end{itemize}

\end{mainprop}

\begin{proof}

The Proposition follows easily from the following

\begin{lem} \label{POS}

\begin{itemize}


\item  $\tilde{\omega}_{\zcal} = \psi_{\lcal_{N + g}}^*
\omega_{FS, N + g}, $ where $\omega_{FS, N + g}$ is the
Fubini-Study metric on $\PP H^0(X, \lcal_{N + g})$ corresponding
to the choice of Hermitian inner product $G_{N + g}(h, \nu)$. It
is a semi-positive $(1,1)$ form.

\item $\tilde{\omega}_{\zcal} + \omega_{\Theta}$ is a strictly
positive $(1,1)$ form on $X^{(N)}$.

\end{itemize}

\end{lem}

We consider  the following diagram.

\begin{equation}\label{DIB}  \begin{array}{llll} \zcal_N & \stackrel{\simeq}{\rightarrow}
 &\sigma_{\lcal}^* \ocal(1) & \;\;\;\stackrel{\sigma^*_{\lcal}}{ \rightarrow}\;\; \ocal(1) \\ \\  \pi \downarrow & & \downarrow \rho
 & \;\;\;\;\;\;\;\;\;\;\;\;\;\;\; \downarrow \pi
 \\ \\
(X^{(N)}, \tilde{\omega}_{\zcal}) & \stackrel{\dcal
\simeq}{\leftarrow} & (\PP \tilde{\ecal}_N, \sigma_{\lcal}^*
\omega_{FS, N + g}) & \stackrel{\sigma_{\lcal}}{
\rightarrow} \;\; (\PP H^0(X, \lcal_{N + g}), \omega_{FS}) \\ \\
A_N \downarrow & & \downarrow \mbox{Definition \ref{ETILDEDEF} } &
 \\ \\ \mbox{Jac}(X) & \stackrel{\simeq}{\leftarrow} &
X^{(g)} &
\end{array}
\end{equation}
where $\simeq$ is the identification, and $\sigma_{\lcal}$ is the
branched covering map, discussed in \S \ref{TILDEECALN}.

The inner product $G_{N + g}(h_0, \nu)$ on $H^0(X, \lcal_{N  +g})$
induces a Fubini-Study Hermitian metric $G_{N + g}(h, \nu)$ on
$\ocal(1) \to \PP H^0(X, \lcal_{N + g})$ and the Fubini-Study form
$\omega_{FS, N + g}$ on $\PP H^0(X, \lcal_{N + g})$ is its
curvature $(1,1)$ form. Under the holomorphic map $\psi_{\lcal_{N
+ g}}$ and its lift to $\ocal(1)$, the Hermitian metric and
curvature form pull back to a Hermitian metric and its curvature
form on $\psi^*_{\lcal_{N + g}} \ocal(1)$. The pulled back
Hermitian metric is the natural Hermitian metric on $\ocal(1) \to
\PP \ecal_N$ defined by the Hermitian inner products. Since
$\psi_{\lcal_{N + g}}$ is holomorphic, $\psi_{\lcal_{N + g}}^*
\omega_{FS, N + g} = \omega_{\PP \ecal_N}$, where the latter is
$\tilde{\omega}_{\zcal}$ transported to $\PP \ecal_N$.

It follows that $\tilde{\omega}_{\zcal_N}$ is semi-positive and
strictly positive away from the singular set of $\psi_{\lcal_{N +
g}} $ determined in Proposition \ref{CANONICAL}.  It is also
strictly positive along the fibers of $\PP \ecal_N \to X^{g}$.
Since $A_N^* \omega_{\Theta}$ is semi-positive and strictly
positive along the `horizontal' directions transverse to the
fibers, the sum of the two is strictly positive.

By the definitions  (\ref{FSHAARDEF}) and  (\ref{JPCDEF}),
\begin{equation} \left\{ \begin{array}{lll}    \psi_{\lcal_{N +
g}}^*(\tilde{\omega}_{FS, N + g}^{N - g}) \wedge  A_N^* d
\omega_{\Theta}^g & = & \tilde{\omega}_{\zcal}^{N - g} \wedge
A_N^* \omega_{\Theta}^g,
\\ && \\ \psi_{\lcal_{N +
g}}^*(\omega_{FS, N + g}^{N})  & = & \tilde{\omega}_{\zcal}^{N }.
\end{array} \right.\end{equation}

\end{proof}

Since $\psi_{\lcal_{N + g}}$ is singular along the Wirtinger
subvariety of $\PP \ecal_N$, or equivalently along $X^N_{N + g}$,
it is helpful to factor the pullbacks as follows:

\begin{cor}\label{JPCTILDE}

\begin{itemize}

\item The JPC for the Fubini-Study-Haar ensemble is given by
$\vec{K}_N = \dcal_* \left(\iota_{\lcal}^* \omega^{N-g}_{FS, N +
g} \right) \wedge A_N^*\omega_{\Theta}^g;$

\item The JPC for the projective linear ensemble is given by
$\vec{K}_N =  \dcal_* \left(\iota_{\lcal}^* \omega^{N}_{FS, N + g}
\right)$

\end{itemize}

\end{cor}

\section{\label{PLVSFSH} Proof of Theorem \ref{JPDFSH}}

In Proposition \ref{JPCPROP} we describe the two probability
measures as volume forms on $\PP \ecal_N$ or equivalently on
$X^{(N)}$. We now use this relation to give an explicit formula
for $\vec K^N_{FSH}$:

\begin{maintheo}\label{MAINFSHPROP}  We have
$$\vec K^N_{FSH} (\zeta_1, \dots, \zeta_N) = \jcal_N (\zeta_1, \dots, \zeta_N)  \frac{1}{Z_N(\omega)}
    \frac{\exp \left(\frac{1}{2} \sum_{i \not=
j} G_{\omega}(\zeta_i, \zeta_j) \right) \prod_{j = 1}^N d^2
\zeta_j}{\left(\int_{X} e^{ N \int_{X} G_{\omega}(z,w)
d\mu_{\zeta}(w)} e^{ \int_{X} G_{\omega}(z,w) d\mu_{P(\zeta)}(w)}
d\nu(z) \right)^{N+1}}, $$ where (for a certain constant $B(N,
g)$),
\begin{equation} \label{JN} \begin{array}{lll} \jcal_N (\zeta_1, \dots, \zeta_N)
 & = &
  B(N, g)  e^{ -(N + 1) \sum_{j = 1}^g
\rho_{\omega}(P_j(\zeta))} ||\Phi_{h^N}^{P_0}||^{- 2}\\ && \\
& \cdot & ||  \theta(P_1 + \cdots + P_g - P_0 - \Delta)||^{-2}
\;\; \prod_{j = 1}^g |E(P_j, P_0)|^2 \;
 (\det \left( \langle \Phi_{N + g}^{P_j}, \Phi_{N + g}^{P_k} \rangle \right) \\ && \\
&& \left| \left(  \prod_{j \not= n} E(P_n, P_j)\right)
 \right|^{-2}  |\det D A_g|^2,\;\;
\end{array} \end{equation}
where $D A_g$ is the derivative of the Abel map.

Here,
$$ ||  \theta(P_1 + \cdots + P_g - P_0 - \Delta)||^{-2} \;  \prod_{j = 1}^g |E(P_j, P_0)|^2 \;  |\det D A_g|^2$$ is a smooth positive
function, and
$$\det \left(\Phi_{N +
g}^{P_j}(P_k) \right)   \left|  \prod_{j \not= n} E(P_n,
P_j)\right|^{-2}.
$$ is  a smooth nowhere vanishing function on $X^{(g)}$ times
$|\prod_{j < n} E(P_j, P_n)|^{-2}. $

\end{maintheo}

\begin{proof}

 It suffices  to calculate the ratio of the volume
forms, i.e. the Jacobian
\begin{equation} \label{JNa} J_N : = \frac{\tilde{\omega}_{\zcal}^{N}}{\tilde{\omega}_{\zcal}^{N-g} \wedge A_N^*
\omega_{\Theta}^g }. \end{equation} The discussion is similar for
the ensembles of Definition \ref{FSTILDEDEF}, with forms pulled
back from $X^{(g)}$. In fact, the proof involves passing back and
forth between $\tilde{X}^{(N)}$ and $X^{(N)}$ and between
$\mbox{Jac}(X)$ and $X^{(g)}$.

 We do this by
evaluating numerator and denominator on  frames for the vertical
bundle  $V \PP \ecal_N$ and for the horizontal space $H \PP
\ecal_N$ defined by the connection:

\begin{equation} H_{(L, [s])} \PP \ecal_N : =; V_{(L, [s])}^{\perp} \;\; \mbox{with respect to the metric}\;\;
  \tilde{\omega}_{\zcal_N} + A_N^*\omega_{\Theta}. \end{equation}
  In fact, the same connection is defined by the condition that
  the horizontal space is $\tilde{\omega}_{\zcal_N}$-orthogonal to
  the vertical space. Indeed,
 $A_N^* \omega_{\Theta}$ is a `horizontal' $(1,1)$
  form, i.e. $A_N^* \omega_{\Theta}(V, W) = 0$ for all $W \in
  T_{(L, [s])} \PP \ecal_N$ if $V$ is vertical. It follows that  $ \tilde{\omega}_{\zcal_N}(V, H) =   (\tilde{\omega}_{\zcal_N} +
  A_N^*\omega_{\Theta})(H, V)   $. The difference in the metrics
  lies in the fact that $ (\tilde{\omega}_{\zcal_N} +
  A_N^*\omega_{\Theta})$ is always non-degenerate in the
  horizontal subspace.

\subsection{Calculation of the Jacobian $J_N$}

We  let $V_1, \dots, V_N$ be a vertical orthonormal frame for $V$
with respect to $\tilde{\omega}_{\zcal_N} + A_N^*\omega_{\Theta}$,
or equivalently,  $\tilde{\omega}_{\zcal_N}$; this makes sense,
since they are the same metric along the fibers.


\begin{lem}\label{FSHsPL}
The Jacobian (\ref{JNa}) is given by
$$J_N = \frac{\det \langle H_i, \bar{H}_j
\rangle_{\tilde{\omega}_{\zcal}}}{\det (A_N^* \omega_{\Theta}^g
(H_i, \bar{H}_j))}
$$

\end{lem}

\begin{proof}

For   $\tilde{\omega}_{\zcal}$, the Gram matrix of inner products
of the elements of the frame is the $N\times N$ matrix
$$ \gcal_N : =  \begin{pmatrix} \langle V_i, \bar{V}_j \rangle & \langle V_i,
\bar{H}_j\rangle
\\ & \\
\langle \bar{V}_i, H_j \rangle  &  \langle H_i, \bar{H}_j \rangle
\end{pmatrix} = \begin{pmatrix} I_{(N-g) \times (N-g)} & 0_{(N - g) \times g}
\\ & \\
0_{g \times (N -g)} &  \langle H_i, \bar{H}_j
\rangle_{\tilde{\omega}_{\zcal}}
\end{pmatrix}. $$ Here, $\langle, \rangle$ is short for the Hermitian semi-metric
defined by  $\tilde{\omega}_{\zcal}$. Then,
$$\tilde{\omega}_{\zcal_N}^N (V_1, \bar{V}_1, \dots, V_{N - g},
\bar{V}_{N - g}, H_1, \bar{H}_1, \dots, H_g, \bar{H}_g) \; = \;
\det \gcal_N. $$

 In the case of $\tilde{\omega}_{\zcal}^{N-g}
\wedge A_N^* \omega_{\Theta}^g$, the volume of the frame is the
determinant of the  matrix
$$ \begin{pmatrix} \langle V_i, \bar{V}_j \rangle & 0
\\ & \\
0  &   A_N^* \omega_{\Theta}^g(H_i, \bar{H}_j)
\end{pmatrix} = \begin{pmatrix} I_{(N-g) \times (N-g)} & 0_{(N - g) \times g}
\\ & \\
0_{g \times (N - g)} & A_N^* \omega_{\Theta}^g(H_i, \bar{H}_j)
\end{pmatrix}. $$

\end{proof}

We now evaluate each horizontal block determinant by carrying the
connection over to $X^{(N)}$ under the identification with $\PP
\ecal_N$ and horizontally lift a curve from $\mbox{Jac}(X)$ to
$X^{(N)}$ starting at a point $\vec \zeta = \zeta_1 + \cdots +
\zeta_N$. The point $\vec \zeta$ corresponds to a line of sections
$[s] \in \PP \ecal_N$.

We fix the standard basis of holomorphic $\frac{\partial}{\partial
z_j}$ on $\mbox{Jac}(X)$ and the associated basis of   real vector
fields $\frac{\partial}{\partial x_j}, \frac{\partial}{\partial
y_j}$. We let $H_j$ be the horizontal lift of
$\frac{\partial}{\partial z_j}$.

Further,  denote the integral curves by $x_j(t), y_j(t)$. We now
consider the horizontal lifts $[s_k(t)]$ of the curves $x_j(t),
y_j(t)$ to $\PP \ecal_N$, or alternatively the lifted  curves
$\zeta_1^k(t) + \cdots + \zeta_N^k(t)$ in $X^{(N)}$. We want to
calculate the matrix of inner products of their tangent vectors
with respect to $\tilde{\omega}_{\zcal}$. By definition, it is the
same to calculate the matrix of inner products of the tangent
vectors of the image curves in $\PP H^0(X, \lcal_{N + g})$ under
$\sigma_{\lcal}$. We thus need to consider the associated curves
$\vec P^k (t) = P_1^k(t) + \cdots + P_g^k(t)$ in $X^{(g)}$ which
are the images of $\zeta_1^k(t) + \cdots + \zeta_N^k(t)$ under
$A_{\lcal_{N + g}}$. The map is only well-defined away from
$X^{(N)}_{N + g}$. (Here, it would have been preferable to work on
$\tilde{X}^{(N)}$, but then the connection would degenerate).

 The curves $P^k_1(t) + \cdots +
P^k_g(t)$ in $X^{(g)}$ are the same as the images of the $x_k(t),
y_k(t)$ under the  inverse Abelian sums map $A_g^{-1}: \mbox{Jac}
\to X^{(g)}$.  The inverse map is not well-defined on $W_g^1$ (\S
\ref{AJSECT}). We are calculating volume forms so it is sufficient
to work on the complement of this set, but the degeneraracy set of
$A_g$ makes itself felt in the zeros and singularities of the
forms.

 Since  $\tilde{\omega}_{\zcal_N} = \psi_{\lcal_{N + g}}^* \omega_{FS, N +
g}$ (Lemma \ref{POS}), or $\sigma_{\lcal}^* \omega_{FS, N + g}$ on
$\tilde{X}^{(N)}$,  the  horizontal space $H_{([s], P_1 + \cdots +
P_g))} \PP \tilde{\ecal}_N$
  maps under $D \sigma_{\lcal}$ to the
orthogonal complement of the tangent space to the canonically
embedded image of  $\PP H^0(X, L)$. This complement may be
identified with the orthogonal complement in $H^0(X, \lcal_{N +
g})$ to the subspace $\{s \in H^0(X, \lcal_{N + g}): D(s) \geq P_1
+ \cdots + P_g \} $ with respect to the inner product $G_{N +
g}(h, \nu)$. Indeed, the natural projection $\pi: \C^{d + 1} -
\{0\} \to \CP^d$ is a Riemannian submersion when the spaces are
equipped with compatible Euclidean, resp. Fubini-Study, metrics.
It is a principal $\C^*$ bundle and carries a natural (Hopf)
connection.

We denote by $[s_k(t)]$ the horizontal curve in $\PP
\tilde{\ecal}_N$ whose initial
 tangent vector is $H_k$.  Under $\sigma_{\lcal}$ it goes to a curve
 in $\PP H^0(X, \lcal_{N + g})$. We use the well-known
 identification of the projective space $\PP H$ associated to a
 Hilbert space with the set of rank one Hermitian orthogonal projections
 in the space of Hermitian operators on $H$. Thus, to  $[s_k(t)]$
 we associate the curve $\pi_k(t) = \frac{s_k(t) \otimes
 s_k(t)^*}{||s_k(t)||_{L^2}^2}$ of projections. Then the
 Fubini-Study inner product of tangent vectors is given by
$\langle \dot{\pi}_k, \dot{\pi}_m \rangle = Tr \dot{\pi}_k \circ
\dot{\pi}_m^*$.

We then put $$s_k(t) = \prod_{j = 1}^g E(\cdot,
 P_j^k(t)) \prod_{j = 1}^N E(\cdot, \zeta_j^k(t)), \;\; S_k(t) =
e^{i \theta_k(t)}  \frac{s_k(t)}{||s_k(t)||}, $$ where the  phase
$e^{i \theta_k(t)}$ is chosen so that $\langle
 \dot{S}_k(t), S_k(t) \rangle = 0$. With this phase condition the
 map $\psi_t \to |\psi_t \rangle \langle \psi_t |$ is an isometry,
 i.e.
 \begin{equation} \label{TR} Tr \dot{\pi}_k(0) \dot{\pi}^*_m(0) = \langle \dot{S}_k(0),
 \dot{S}_m(0) \rangle. \end{equation}
Horizontality is the condition that $\dot{S}_k(0) \bot H^0(X,
\ocal(\zeta_1 + \cdots + \zeta_N))$. As   long as the $P_j$ are
 distinct, the coherent states  $\{\Phi_{N + g}^{P_j}\}$ form  a basis of the ortho-complement.  Hence,
there exists a matrix  $C = (C_j^k)$ of coefficients so that
\begin{equation} \label{HORTANa} \frac{d}{dt}|_{t = 0} S_k(t) = \sum_{j = 1}^g C_j^k \Phi_{N +
g}^{P_j}. \end{equation}

\begin{rem} Abbreviate $\lcal = \lcal_{N + g}$. Since $\Phi_{N +
g}^{P_j}(z) \in \lcal_{z} \otimes \lcal_{P_j}^*$, the left side of
(\ref{HORTANa}) is a section of $\lcal$ while the right side takes
values in $\lcal \boxtimes \lcal^*$. Hence, $C_j^k $ must take
values in $\lcal_{P_j}$ and the product $ C_j^k \Phi_{N +
g}^{P_j}$ implicitly involves a contraction. It will be seen below
that these extra factors cancel, so rather than introduce new
notation we will keep track of them in a series of remarks.
\end{rem}

The next step is to calculate the matrix $C$. To this end, we
introduce the matrix
 $$M_{\vec P} = M = (M_{i j}), \;\;\; \mbox{with}\;\; M_{ij} : = \langle
\Phi^{P_i}_{N + g}, \Phi^{P_j}_{N + g} \rangle =  \det (\Phi_{N +
g}^{P_j}(P_k)) . $$

\begin{rem} Since $\Phi_{N +
g}^{P_j}(P_k) = \Pi_{N + g}(P_k, P_j)$, this is matrix is also
$\lcal \boxtimes \lcal^*$-valued. As discussed in \S
\ref{BERGMANSECTION}, we could write it as $B_{N + g}(P_k, P_j)$
times a frame for $\lcal \boxtimes \lcal^*$. As noted above, the
frames will cancel later.
\end{rem}

We further define the matrix   $Q = Q_{s, \vec P}$ by
\begin{equation} \label{Q}  (Q_{n k}) = \frac{d}{dt}|_{t = 0} S_k(t)(P_n).   \end{equation}

\begin{rem} This matrix is $\lcal$-valued. ince the curve has  been lifted to $H^0(X, \lcal_{N + g})$,
 the tangent vectors may be identified  with elements of this vector
 space.
  Hence $Q_{n k} \in \lcal_{N + g}(P_n). $ \end{rem}

 It is clear that the $\frac{d}{dt}|_{t = 0}$ derivative only
 produces a non-zero term when the factor $E(P_n, P_n^k(t))$ is
 differentiated.  Thus, we have
\begin{equation} Q_{k n}  = \frac{ \left(\prod_{j \not= n} E(P_n,
 P_j) \prod_{j = 1}^N E(P_n, \zeta_j) \frac{d}{dt}_{t = 0} E(P_n, P_n^k(t)) \right)}{||\left(\prod_{j } E(\cdot,
 P_j) \prod_{j = 1}^N E(\cdot \zeta_j) \right)||_{L^2}}  e^{i \theta_k(t)}. \end{equation}
Since $E$ is a section of a line bundle, it
 should require a connection to differentiate it in the second
 component. However, since $E(z,w)$  vanishes on the diagonal, all
 connections give the same result, i.e.  the derivative of the
 coefficient $P_n - P_n^k(t)$ relative to a frame. Since the
 derivative does not touch the frame, it produces another
 holomorphic section of the same line bundle. We then identify
$  \frac{d}{dt}_{t = 0} E(P_n, P_n^k(t))$ with the scalar $-
\dot{P}_n^k(0)$, the coordinate of the $n$-component of the $k$th
curve in our choice of local coordinates. We thus have,
\begin{equation} \label{detQ} \det Q_{s, \vec P} =\frac{ \left(\prod_{\ell = 1}^g \prod_{j =
1}^N E(P_{\ell},\zeta_j)\right) \left(  \prod_{j \not= n}  E(P_n,
P_j)\right)}{||\left(\prod_{j } E(\cdot,
 P_j) \prod_{j = 1}^N E(\cdot \zeta_j) \right)||_{L^2}}  e^{i \theta_k(t)}. \left(\det  (- \dot{P}_n^k(0)) \right) .\end{equation}

To put the last determinant in an invariant form, we note that
$$ \left(- \dot{P}_n^k(0)) \right)^*  \left(- \dot{P}_n^k(0)) \right) = D
A_g^{-1 *} D A_g^{-1}. $$ Indeed, $(- \dot{P}_n^k(0))$ is the
matrix of $D A_g^{-1}$ relative to the standard basis
$\frac{\partial}{\partial z_j}$ of $\mbox{Jac}(X)$ and the
coordinate basis of $X^{(g)}$.

We now claim that

\begin{lem} \label{JNCLAIM} Let $H_j$  be the horizontal lifts of the coordinate vector fields
$\frac{\partial}{\partial z_j}$ on $\mbox{Jac}(X)$ to $H_{([s],
P_1 + \cdots + P_g)} \PP \ecal_N$.  Then
\begin{equation} \label{MAINFORM}  \det
 (\langle
H_i, \bar{H}_j \rangle_{\tilde{\omega}_{\zcal_N}}) = \det (M^{
-1}) |\det Q|^2,
\end{equation} and therefore, $$J_N = (\det M^{ -1}) |\det Q|^2.$$
\end{lem}

\begin{rem} We observe that both $\det M$ and  $|\det Q|^2$ take values in $\lcal
\boxtimes \lcal^*$. Hence if we expressed them relative to a frame
$e_{\lcal}$ for $\lcal$, the frames would cancel in the quotient.
If we we write a section $\dot{S}_k(P_n) = f_k(P_n)
e_{\lcal}(P_n)$ then the ratio leaves a Slater determinant $\|
\det f_k(P_n)|^2$ in the numerator and the Bergman determinant
$B_{N + g}(P_k, P_j)$ in the denominator. Thus the ratio is a
positive scalar function.

\end{rem}

We now prove the Lemma:

\begin{proof}

In view of
 Lemma \ref{FSHsPL}, the second statement follows from the first.
 We now prove the first statement. Taking inner products and using
 (\ref{HORTANa}), we have
$$(\langle H_i, \bar{H}_j
\rangle_{\tilde{\omega}_{\zcal_N}}) = C M C^* $$ and
\begin{equation} \label{STEP1} \det (\langle H_i, \bar{H}_j
\rangle_{\tilde{\omega}_{\zcal_N}}) = \det \left(\langle
\Phi^{P_i}_{N + g}, \Phi^{P_j}_{N + g} \rangle \right) \det
\left(C^* C \right) = \det M \det (C^* C). \end{equation}

 It also follows by setting $z = P_n$ in  (\ref{HORTANa}) that
\begin{equation} \label{HORTANn} Q_{k n} = \sum_{j = 1}^g C_j^k \Phi_{N +
g}^{P_j}(P_n). \end{equation} Thus, we have
\begin{equation} Q = C M \iff C = Q M^{-1}. \end{equation}
Here, we are  assuming that the $P_j$ are distinct, hence that $M$
is invertible (Lemma \ref{M}).

 It follows that
\begin{equation} \det (\langle H_i, \bar{H}_j \rangle) = \det (Q^*
Q) \det M^{ -1}. \end{equation}

\end{proof}

\subsection{Configuration spaces and lifts of $\vec K^N_{FSH}$}

Using that $$\vec K^N_{FSH} (\zeta_1, \dots, \zeta_N) =
\tilde{\omega}_{\zcal}^{N-g} \wedge A_N^* \omega_{\Theta}^g  =
J_N^{-1} \tilde{\omega}_{\zcal}^{N} =  J_N^{-1}
\tilde{\omega}_{\zcal}^{N} , $$  it follows from Lemma
\ref{JNCLAIM} that the lift of $\vec K^N_{FSH}$ to the Cartesian
product $X^N$ is given by
$$\vec K^N_{FSH} (\zeta_1, \dots, \zeta_N) = \jcal_N (\zeta_1, \dots, \zeta_N)  \frac{1}{Z_N(\omega)}
    \frac{\exp \left(\frac{1}{2} \sum_{i \not=
j} G_{\omega}(\zeta_i, \zeta_j) \right) \prod_{j = 1}^N d^2
\zeta_j}{\left(\int_{X} e^{ N \int_{X} G_{\omega}(z,w)
d\mu_{\zeta}(w)} e^{ \int_{X} G_{\omega}(z,w) d\mu_{P(\zeta)}(w)}
d\nu(z) \right)^{N+1}}, $$ where (for a constant $B(N, g)$),
\begin{equation} \label{JN-} \begin{array}{lll} \jcal_N (\zeta_1, \dots, \zeta_N)
 & = &
 \; \cdot  F_N(\zeta_1, \dots, \zeta_N)  e^{- 2 N \int
\rho_{\omega}(w) d\mu_{\zeta}(w)} (\det M) |\det Q|^{-2} \;\;
\\ && \\ & = & B(N, g)  e^{ -(N + 1) \sum_{j = 1}^g
\rho_{\omega}(P_j(\zeta))}  ||\Phi_{h^N}^{P_0}||^{- 2} \left
|\left(\prod_{k = 1 }^N \;\prod_{j = 1}^g E(P_j,
\zeta_k)\right) \right|^2\\ && \\
& \cdot & ||  \theta(P_1 + \cdots + P_g - P_0 - \Delta)||^{-2}
\;\; \prod_{j = 1}^g |E(P_j, P_0)|^2 \;
 (\det \left( \langle \Phi_{N + g}^{P_j}, \Phi_{N + g}^{P_k} \rangle \right) \\ && \\
&& \left|\left(\prod_{\ell = 1}^g \prod_{j = 1}^N
E(P_{\ell},\zeta_j)\right) \left(  \prod_{j \not= n}  E(P_n,
P_j)\right) \right|^{-2} |\det D A_g|^2\;\;
\\ && \\ & = & B(N, g)  e^{ -(N + 1) \sum_{j = 1}^g
\rho_{\omega}(P_j(\zeta))} ||\Phi_{h^N}^{P_0}||^{- 2}\\ && \\
& \cdot & ||  \theta(P_1 + \cdots + P_g - P_0 - \Delta)||^{-2}
\;\; \prod_{j = 1}^g |E(P_j, P_0)|^2 \;
 (\det \left( \langle \Phi_{N + g}^{P_j}, \Phi_{N + g}^{P_k} \rangle \right) \\ && \\
&& \left| \left(  \prod_{j \not= n} E(P_n, P_j)\right)
 \right|^{-2}  |\det D A_g|^2\;\;

\end{array} \end{equation}

We next claim that
$$ ||  \theta(P_1 + \cdots + P_g - P_0 - \Delta)||^{-2} \;  \prod_{j = 1}^g |E(P_j, P_0)|^2 \;  |\det D A_g|^2$$ is a smooth positive function. To see this, we
recall that $\theta(P_1 + \cdots + P_g - P_0 - \Delta) = 0$
whenever there exists $Q_1 + \cdots + Q_{g -1} \in X^{(g-1)}$ so
that $P_1 + \cdots + P_g - P_0 = Q_1 + \cdots + Q_{g -1}. $ This
can occur if some $P_j = P_0$ and the $\{Q_j\}$ are the remaining
$\{P_k\}_{k \not= j}$, and such poles are cancelled by $\prod_{j =
1}^g |E(P_j, P_0)|^2$.  It can also happen at other points on the
Wirtinger variety $W^1_g$ where the representation fails to be
unique, and these poles are cancelled by $|\det D A_g|^2$.

We further claim that the factor
$$\det \left(\Phi_{N +
g}^{P_j}(P_k) \right)   \left|  \prod_{j \not= n} E(P_n,
P_j)\right|^{-2}.
$$ is  a smooth nowhere vanishing function on $X^{(g)}$ times
$|\prod_{j < n} E(P_j, P_n)|^{-2}. $ Thus, the $\det M$ factor
cancels `half' the poles of the second factor.

To explain this, we note that $ \det
 (\langle
H_i, \bar{H}_j \rangle_{\tilde{\omega}_{\zcal_N}}) = 0$ if and
only if $([s], P_1 + \cdots + P_g) \in \bcal^{(N)}_{N + g}$, the
branch locus of the map $\sigma_{\lcal}$. Indeed, $\{H_1, \dots,
H_g\}$ is a basis of the horizontal space at all points, so the
determinant can only vanish when $\tilde{\omega}_{\zcal_N}$
degenerates. Since it is the pullback of a non-degenerate form
under $\sigma_{\lcal}$, it only degenerates on the branch locus.
We recall that this is the locus where $\zeta_1 + \cdots + \zeta_N
+ P_1 + \cdots + P_g$ has multiplicity (i.e. at least two terms
coincide). We observe that also

\begin{equation}  \label{M}  \det M_{\vec P} = 0 \iff \exists
j \not= k: \;\;P_j = P_k. \end{equation}

Indeed,   $\det M = 0 $  if and only if the map
$$\{s \in H^0(X, \lcal_{N + g}): \dcal(s) \geq P_1 + \cdots +
P_g\}^{\perp} \to  \lcal_{N + g}[P_1] \oplus \cdots \oplus
\lcal_{N + g}[P_g] $$ on the given ortho-complement, sending $s
\to (s(P_1), \dots, s(P_g))$, has a kernel. The kernel is trivial
if the $P_j$ are distinct,  since $s(P_1) = \cdots = s(P_g) = 0$
implies that $s$ lies both in the given subspace and its
orthocomplement. When there are multiplicities, then there is a
non-trivial kernel;  one needs to supplement the map with the
derivatives of $s$ at the multiple points.

This completes the proof of Theorem \ref{MAINFSHPROP}.

\end{proof}

\section{\label{LDSECTIOIN} LDP for the projective linear ensemble: Proof of Theorem \ref{LD}}

In this section we prove Theorem \ref{LD} for the projective
linear ensemble. More precisely,  we reduce the proof to the
results of \cite{ZZ}.

For the sake of completeness, we  recall the definition of the
LDP: if $B(\sigma, \delta)$ denotes the ball of radius $\delta$
around $\sigma \in \mcal(\CP^1)$ in the Wasserstein metric, and
$B^o(\sigma,\delta)$ (respectively, $\overline{B(\sigma,\delta)}$)
denote its interior (respectively, its closure), then
\begin{equation} \label{LDPROP}
\begin{array}{lll} - \inf_{\mu \in B^o(\sigma, \delta)} \tilde I^{\omega, K}(\mu)
&\leq & \liminf_{N \to \infty}
\frac{1}{N^2} \log {\bf Prob} _N(B(\sigma, \delta)) \\ && \\
&\leq&  \limsup_{N \to \infty} \frac{1}{N^2} \log {\bf Prob}
_N(B(\sigma, \delta)) \leq    - \inf_{\mu \in \overline{B(\sigma,
\delta)}} \tilde I^{\omega, K} (\mu). \end{array}
\end{equation}

Once we have found the approximate rate functional, and have
expressed it in terms of Green's functions,  we can take its limit
precisely as in \cite{ZZ} and obtain the LDP.

For any ensemble, we express the lift of $\vec K^N$ to the
Cartesian product as
\begin{equation} \vec K^N(\zeta_1, \dots, \zeta_N) = D_N(\zeta_1, \dots,
\zeta_{N})\prod_{j = 1}^N d^2 \zeta_j. \end{equation}

\subsection{An approximate rate function: Proof of Proposition \ref{APPROXRATEa}}

The Proof of Proposition \ref{APPROXRATEa} is similar to that of
Lemma 18 of \cite{ZZ}, the principal new feature being the
coefficient function $F_N$.

We  express the JPC of Theorem \ref{JPDHG} in terms of the
empirical measures $\mu_{\zeta}$. The following Lemma proves
Proposition  \ref{APPROXRATEa}.

\begin{lem} \label{APPROXRATE} We have
$$\vec K_n^N(\zeta_1, \dots, \zeta_N) = \frac{1}{\hat{Z}_N(h)} e^{- N^2 \left( -\frac{1}{2}
\ecal^{\omega}_N(\mu_{\zeta}) + \frac{N+1}{N}  J_N^{\omega, \nu}
(\mu_{\zeta})\right)} \;\; \kappa_N,
$$
where (cf. Proposition \ref{FNFINAL}),
$$\begin{array}{lll} \kappa_N & = &
F_N(\zeta_1, \dots, \zeta_N) \prod_j d^2 \zeta_j.
\end{array}$$

\end{lem}

\begin{proof} We first observe that the main factor in (II) of Theorem \ref{JPDHG} can be rewritten in terms
of the empirical measure to get,
$$ \frac{\exp \left(\frac{1}{2} \sum_{i \not=
j} G_{\omega}(\zeta_i, \zeta_j) \right) }{\left(\int_{X} e^{ N
\int_{X} G_{\omega}(z,w) d\mu_{\zeta}(w)} e^{ \int_{X}
G_{\omega}(z,w) d\mu_{P(\zeta)}(w)} d\nu(z) \right)^{N+1}} = e^{-
N^2 I_N^{\omega, \nu} (\mu_{\zeta})}. $$ Indeed, by taking the
right side as the definition of $I_N^{\omega, \nu}$, we get
\begin{equation*} \label{INV2}
    \begin{array}{lll}
   I_N^{\omega, \nu}
& = &  - \frac{1}{N^2}\sum_{i \not= } \frac{1}{2}
G_{\omega}(\zeta_i, \zeta_j) + \frac{N+1}{N^2} \log \left(\int_{X}
e^{N \int_{X} G_{\omega}(z,w) d\mu_{\zeta} } e^{ \int_{X}
G_{\omega}(z,w) d\mu_{P(\zeta)}(w)}  d\nu(z) \right)\nonumber \\
&& \quad
\quad \\
 & = &
 - \frac{1}{N^2}  \frac{1}{2} \int_{X \times X
\backslash D} G_{\omega}(z,w) d\mu_{\zeta}(z) d\mu_{\zeta}(w) +
\frac{N+1}{N^2} \log \left(\int_{X} e^{N
U^{\mu_{\zeta}}_{\omega}(z) } e^{U^{\mu_{P(\zeta)}(z)}} d\nu(z)
\right) \nonumber  \\ && \\ & = & - \frac{1}{N^2} \left(
 - \frac{1}{2}  \ecal_N^{\omega}(\mu_{\zeta}) +  \frac{N(N + 1)}{N^2}   J_N^{\omega, \nu}
 (\mu_{\zeta})\right),
\end{array}
\end{equation*}
as one sees by comparing with Definition \ref{APPROXRF}. We then
combine the rest of the factors into $\kappa_N$.
  \end{proof}

\subsection{Properties of the rate function} The properties of the
rate function in higher genus are similar to those proved in
\cite{ZZ} Section 6 in genus zero, and the proofs are the same, so
we state them rapidly and refer to \cite{ZZ} for the proof.

\begin{prop} \label{GOODRATEF} (see \cite{ZZ}, Proposition 24)  The function $I^{\omega, K}$ of
(\ref{IGREEN}) has the following properties:

\begin{enumerate}
\item It is a lower-semicontinuous functional. \item It is
    strictly convex.
\item Its unique minimizer is the equilibrium measure $\nu_{h,
K}$. \item Its minimum value equals $\frac{1}{2}  \log \mbox{\rm
Cap}_{h} (K).$
\end{enumerate}

\end{prop}

 Set
\begin{equation}
    \label{eq-ofer1a}
    E_0(h)=\inf_{\mu\in \mcal(X)} I^{\omega,K}(\mu), \quad
    \tilde I^{\omega,K}=I^{\omega,K}-E_0(h)\,.
\end{equation}
The infimum $\inf_{\mu\in \mcal(X)} I^{\omega,K}(\mu)$ is achieved
at the Green's equilibrium measure $\nu_{\omega, K}$ with respect
to $(\omega, K)$, and   $E_0(h)= \frac12 \log \mbox{Cap}_{\omega}
(K)$, where (as above) $\mbox{Cap}_{\omega}(K)$ is the Green's
capacity with respect to $\omega$.  For background we refer to
\cite{ZZ} and its references.

\subsection{\label{LDSECT} Completion of proof of Theorem \ref{LD}}

%
%

Given Theorem \ref{MAINFSHPROP},  Lemma \ref{APPROXRATE} and the
arguments of \cite{ZZ}, the main remaining complication in proving
Theorem \ref{LD} is to deal with the singular factor  $\left|
\prod_{j < n} E(P_n, P_j)\right|^{-2}$ described in the statement
of  Theorem \ref{MAINFSHPROP}. Of course it must be cancelled by
the other factors since the form is smooth, but we need to make
the necessary estimates to take the limit as $N \to \infty$.  We
recall that this factor arises since the pullback of the
Fubini-Study volume form on $\PP H^0(X, \lcal_{N + g})$ to $\PP
\ecal_N$ has degeneracies on the branch locus, while those of the
Fubini-Study-fiber volume forms do not.

The cancellation of this factor comes from the fact (discussed in
\S \ref{XN}) that $\vec K^N_{FSH}$ is a smooth non-degenerate form
on $\PP \ecal_N \simeq X^{(N)}$ which is bundle-like and which
contains the lift of a factor on $X^{(g)}$. This factor can be
expressed in coordinates $P_1 + \cdots + P_g$ in the image of the
Abel map and its lift to $X^g$ contains a Vandermonde type factor
 $\left| \prod_{j < n} E(P_n, P_j)\right|^{2}$ cancelling the
 singular one above. To determine the ratio, we need to change
 coordinates again from $\zeta_1 + \cdots + \zeta_N$ to $P_1 +
 \cdots  + P_g$ and $N - g$ remaining coordinates along the fiber,
 $\PP H^0(X, \ocal( (N + g)P_0 - (P_1 + \cdots + P_g))$. To prove
 Theorem \ref{MAINFSHPROP}, i.e. to study empirical measures, we
 need the remaining $N - g$ coordinates to be divisor coordinates
 rather than coefficients relative to a basis. Indeed, this was
 the main reason for introducing $\lcal_{N + g}$ in the first
 place.

 We therefore seek an $X^{N - g}$-valued  fiber coordinate $\eta_1 +  \cdots + \eta_{N - g}$
 which are coordinates of $N - g$ zeros of sections in the fiber $\PP H^0(X, \ocal( (N + g)P_0 - (P_1 + \cdots +
 P_g))$. Since $\dim \PP H^0(X, \ocal( (N + g)P_0 - (P_1 + \cdots +
 P_g)) = N - g$, a section is specified up to scalar multiples by
 $N - g$ zeros.  Equating $\PP \ecal_N = X^{(N)}$, it is equivalent  to
define (almost everywhere) an analytic function
$$\vec \zeta (\vec \eta, \vec P)  \in  X^{(N)}: A_N(\zeta) = P_1 + \cdots + P_g $$
whose image is an open dense (indeed, Zariski open) subset of
$X^{(N)}$. There are of course ${N \choose g}$ ways to select
 s subset of
 $N - g$ zeros from $\zeta_1 + \cdots + \zeta_N$, and what we are
 claiming is that there exists a well-defined analytic
 branch of the correspondence $X^{(N)} \to X^{(N - g)} \times
 X^{(g)}$ whose graph is given by
 $$\{(\zeta_1 + \cdots + \zeta_N; (\zeta_{j_1} + \cdots + \zeta_{j_{N-g}}: P_1 + \cdots + P_g) \} \subset X^{(N)} \times (X^{(N - g) } \times X^{(g)}),
 $$
where as usual $P_1 + \cdots + P_g = A_N(\zeta_1 + \cdots +
\zeta_N).$ Existence of such a branch follows from the fact that
correspondence is a covering map on the complement of the branch
locus, and the complement is  a Zariski open set (see also
\cite{Mat}).

We recall that the expression for $\vec K^N_{FSH}$ in Theorem
\ref{MAINFSHPROP} is for the pull back of this form to $X^N$ in
the local coordinates $\zeta_1, \dots, \zeta_N$. We now use the
local coordinates $\eta_1, \eta_2, \dots, \eta_{N-g}, P_1, \dots,
P_g$ instead. By definition of a Fubini-Study-fiber bundle
probability measure, $\vec K^N_{FSH}$ has the wedge product of a
form $d \sigma$ lifted from $X^{(g)}$ and a smooth Fubini-Study
volume form along the fibers. The former is a smooth positive
multiple of $|\prod_{j \not= k = 1}^g E(P_j, P_k)|^2 \prod_{j =
1}^g d P_j \wedge d \bar{P}_j$ and a smooth form in $\eta$.
 It follows that the factor  $\left| \prod_{j < n} E(P_n,
 P_j)\right|^{2}$ in the formula of Theorem
\ref{MAINFSHPROP} is cancelled by the same Vandermonde factor
arising from the expression of $\vec K^N_{FSH}$ in the coordinates
$(\vec \eta, \vec P)$. This changes the definition of $\kappa_N$
in Lemma \ref{APPROXRATE} and Theorem \ref{MAINFSHPROP} to
\begin{equation} \label{KAPPATILDE} \tilde{\kappa}_N =
\tilde{\jcal}_N \prod_{j = 1}^g d P_j \wedge d \bar{P}_j \wedge
\prod d_{j = 1}^{N - g} d \eta_j \wedge d \bar{\eta}_j,
\end{equation}  with
\begin{equation} \label{JNTILDE} \begin{array}{lll} \tilde{\jcal_N} (\eta_1, \dots, \eta_{N-g}, P_1 \cdots, P_g)
 & = &
  B(N, g)  e^{ -(N + 1) \sum_{j = 1}^g
\rho_{\omega}(P_j(\zeta))} ||\Phi_{h^N}^{P_0}||^{- 2}\\ && \\
& \cdot & ||  \theta(P_1 + \cdots + P_g - P_0 - \Delta)||^{-2}
\;\; \prod_{j = 1}^g |E(P_j, P_0)|^2 \;
\\ && \\ && (\det \left( \langle \Phi_{N + g}^{P_j}, \Phi_{N + g}^{P_k} \rangle \right) \left| \left(  \prod_{j < n} E(P_n, P_j)\right)
 \right|^{-2}  |\det D A_g|^2,\;\;
\end{array} \end{equation}

\subsubsection{LD upper bound}

We first sketch the proof of the  upper bound part of the large
deviation principle from \cite{ZZ}: \begin{lem}
\begin{equation} \label{UBTOPROVE}
    \lim_{\delta \downarrow 0} \limsup_N \frac{1}{N^2}  \log \; {\bf
Prob}_N (B(\sigma, \delta)) \leq - \tilde I^{\omega,K}(\sigma).
\end{equation}
\end{lem}

\begin{proof}

The first step is to prove the following:  let $\epsilon>0$ and
let $K = \mbox{supp} \nu$.  If $\nu$ satisfies
    the Bernstein-Markov condition (\ref{BM}),
then there exists a $N_0=N_0(\epsilon)$ such that for all $N>N_0$
and all $\mu_\zeta\in \mcal(X)$, (with $\rho_{\omega}$ as in
(\ref{rho}))
$$ \log ||e^{U^{\mu_\zeta}_{\omega}} e^{\frac{1}{N} U^{\mu_{P(\zeta)}}} ||_{L^N(\nu) }
\geq \sup_{z \in K} (U^{\mu_{\zeta}}_{\omega}) - \epsilon \,.$$

This is similar to Lemma 30 of \cite{ZZ}, with two modifications.
First, there is the new factor $\frac{1}{N} U^{\mu_{P(\zeta)}}$
(which is absorbed in the $\epsilon$).  Second, the
 Bernstein-Markov assumption is now a uniform estimate comparing $L^2$ norms
    and sup norms of sections $s \in H^0(\C, \xi)$ as $\xi $
    varies over $\mbox{Pic}^N$:
\begin{equation} \sup_{z \in K} |s(z)|_{h_{\xi}} \leq C_{\epsilon}
e^{\epsilon N} \left(\int_K |s(z)|^2_{h_{\xi(\zeta)}} d \nu(z)
\right)^{1/2}, \;\;\; \forall \xi \in \mbox{Pic}^N, \;\; s \in
H^0(X, \xi).
\end{equation}
Here, $h_{\xi(\zeta)}$ is the admissible metric on the line bundle
$\xi(\zeta)$ where $s_{\zeta} \in H^0(X, \xi(\zeta))$.  By Lemma
\ref{NORMSUPHI} we may write,
$$|s_\zeta(z)|^2_{h_{\xi}} =
e^{N (U_{\omega}^{\mu_\zeta}(z) - \frac{1}{N}
U_{\omega}^{\mu_{P(\zeta)}} } , \;\; \forall \zeta \in X^{(N)}.
$$ Hence,  \begin{equation*}
    \begin{array}{lll} ||e^{U^{\mu_\zeta}_{\omega}} e^{\frac{1}{N} U^{\mu_{P(\zeta)}}} ||_{L^N(\nu)}
    & = &
        \left(\int_K |s_\zeta(z)|^2_{h_{\xi(\zeta)}} d \nu(z)
\right)^{1/N}\\ && \\
& \geq & \left( C^{-1}_{\epsilon} e^{- N \epsilon} \sup_{z \in K}
|s_\zeta(z)|_{h_{\xi(\zeta)}}^2  \right)^{\frac{1}{N}},
\end{array} \end{equation*} and
$$\log ||e^{U^{\mu_\zeta}_{\omega}} e^{\frac{1}{N}
U^{\mu_{P(\zeta)}}}||_{L^N(\nu)} \geq \sup_{z \in K}
U^{\mu_{\zeta}}_{\omega}   - \epsilon + \frac1N\log C_\epsilon,
$$ for all $\epsilon > 0$.

Write \begin{equation} \label{THETAN} \Theta_N=-\frac{1}{N^2}
\log\hat{Z}_N(h)\,.\end{equation} In \S \ref{CONSTANT} we show
that (as in \cite{ZZ}), $\Theta_N\to_{N\to\infty} \log
\mbox{Cap}_{\omega}(K)$.

By  Lemma \ref{JPDINV} and Lemma \ref{APPROXRATE},
\begin{equation}
    \label{eq-180209b}
    \frac{1}{N^2}  \log \; {\bf Prob}_N (B(\sigma, \delta))
= \frac{1}{N^2} \log \int_{\zeta \in X^{(N)} : \mu_{\zeta} \in
B(\sigma, \delta)} e^{- N^2 (I_N(\mu_{\zeta} ) } \tilde{\kappa}_N
+\Theta_N,
\end{equation}
where $\tilde{\kappa}_N$ is the smooth,  non-negative   $(N, N)$
form defined in (\ref{KAPPATILDE}), and $I_N$ is the approximate
rate function.

  Fix $M\in \R$ and let $G_{\omega}^M=G_{\omega}\vee (-M)$
be the truncated Green function and let $\ecal_{\omega}^M$ be the
Green's energy associated to the truncated Green's function. As in
\cite{ZZ}, $G_{\omega}^M$ is continuous on $X\times X$ and
\begin{eqnarray*}
    -\frac{1}{N^2}\sum_{i<j}G_{\omega}(\zeta_i,\zeta_j) &\geq&
 \ecal_{\omega}^M(\mu_{\zeta}) -\frac{C(M)}{N} ,
\end{eqnarray*}
where the constant $C(M)$ does not depend on $\xi$.

It follows that, for any $\epsilon>0$ and all $N>N_0(\epsilon)$,
\begin{equation*}\begin{array}{lll}
 \frac{1}{N^2}  \log \; {\bf Prob}_N (B(\sigma, \delta))
 &\leq & \frac{1}{N^2}\log\int_{\xi\in X^{(N)}: \mu_{\zeta}\in B(\sigma,\delta)}
 e^{\frac{N^2}{2}\ecal_{\omega}^M(\mu_\xi)-N^2J_{\omega}^K(\mu_{\zeta}) }
\tilde{\kappa}_N
\\ && \\ &+&
  (
  \Theta_N
 +\frac{C'(M)}{N}+\epsilon),
\end{array} \end{equation*}
 for some constant $C'(M)$.

 It follows that
$$ \begin{array}{lll}   \limsup_N \frac{1}{N^2}  \log \; {\bf
Prob}_N (B(\sigma, \delta)) &\leq & \limsup_{N\to\infty} \Theta_N
\\ && \\ && + \limsup_{\delta \downarrow 0} \sup_{\mu \in B(\sigma, \delta)}
- \left( - \frac12\ecal_{\omega}^M(\sigma) + J_{\omega}^K(\sigma)
\right) \\ && \\ && + \left| \limsup_N \frac{1}{N^2} \log
\int_{X^{(N)}} \tilde{\kappa}_N \right|\,.
\end{array}$$
 We now claim:
 \begin{lem} \label{ZEROa} Let $\tilde{\kappa}_N$ be the smooth $(N, N)$ form defined in (\ref{KAPPATILDE}). Then
$$ \frac{1}{N^2} \left| \log \int_{\tilde{X}^{(N)}} \tilde{\kappa}_N \right|=
 O(\frac{\log N}{N}).
$$
\end{lem}

\begin{proof} Omitting the constant $B(N, g) \left(\det ( \nabla_k P_n) \right), ||\Phi_{h^N}^{P_0}||^{- 2}$   (which may be absorbed into the overall
normalizing constant $Z_N$), the statement comes down to showing
that
$$\begin{array}{l}  \frac{1}{N^2} | \log  \int_{\tilde{X}^{N}}
\;  e^{ -(N + 1) \sum_{j = 1}^g \rho_{\omega}(P_j(\zeta))}\\ \\
 ||  \theta(P_1 + \cdots + P_g - P_0 -
\Delta)||^{-2} \;\; \prod_{j = 1}^g |E(P_j, P_0)|^2  |\det D A_g|^2 \\ \\
(\det \left( \langle \Phi_{N + g}^{P_j}, \Phi_{N + g}^{P_k}
\rangle \right)^*
 \left| \left(  \prod_{j < n} E(P_n, P_j)\right) \right|^{-2}  \prod_{j = 1}^{N - g} d^2
\eta_j \prod_{j = 1}^g d P_j | = O(\frac{\log N}{N}).
\end{array}$$
We observe that the integrand is a smooth function only of $P_1 +
\cdots + P_g \in X^{(g)}$ and only the factor $e^{ -(N + 1)
\sum_{j = 1}^g \rho_{\omega}(P_j(\zeta))} (\det \left( \langle
\Phi_{N + g}^{P_j}, \Phi_{N + g}^{P_k} \rangle \right)$ depends on
$N$.

We first integrate out the $\eta_j$ variables and obtain the
Lebesgue volume of $X^{(N - g)}$ in the $\eta_j$ coordinates. It
lifts back to $X^{N - g}$ as a product measure and therefore (as
in \cite{ZZ}), we have  $\frac{1}{N^2} \log \int_{X^{(N - g)}}
\prod_{j = 1}^{N - g} d^2 \eta_j = O(\frac{\log N}{N}). $

This reduces us to studying the remaining integral over $X^g$.
Since the factor $ || \theta(P_1 + \cdots + P_g - P_0 -
\Delta)||^{-2} \;\; \prod_{j = 1}^g |E(P_j, P_0)|^2 |\det D A_g|^2
$ is smooth, positive function on $X^{(g)}$ which is independent
of $N$, it is bounded above by a constant $C_g > 0$ and below by
another constant $c_g > 0$. Hence it may be removed from the
integral at the cost of a remainder $O(\frac{1}{N^2})$. Also,  the
factor $ e^{ -(N + 1) \sum_{j = 1}^g \rho_{\omega}(P_j(\zeta))} $
may be removed at the cost of a remainder $O(\frac{1}{N})$.
Consequently, it suffices to show that
$$\begin{array}{l}  \frac{1}{N^2} | \log  \int_{X^{g}}
\; (\det \left( \langle \Phi_{N + g}^{P_j}, \Phi_{N + g}^{P_k}
\rangle \right)^*
 \left| \left(  \prod_{j < n} E(P_n, P_j)\right) \right|^{-2}  \prod_j d^2 P_j  | = O(\frac{\log N}{N}).
\end{array}$$

Denote by $X^g(N)$ the `well-separated'  set of $(P_1, \dots,
P_g)$ so that $d(P_j, P_k) \geq \frac{\log N}{\sqrt{N}}$ for all
$j \not= k$, and then put
$$\int_{X^{g}} = \int_{X^g(N)} + \int_{X^g
\backslash X^g(N)}. $$

On $X^g \backslash X^g(N)$, the off-diagonal elements are of order
$N^{- p}$ for any desired $p > 0$. The diagonal elements are of
order $N$. Hence
$$\det \left( \langle \Phi_{N + g}^{P_j}, \Phi_{N + g}^{P_k}
\rangle \right) \geq N^g - O(N^{- p}), \;\;\; \mbox{on}\; X^g
\backslash X^g(N). $$ Since also $ \left| \left(  \prod_{j < n}
E(P_n, P_j)\right) \right|^{-2} \geq \epsilon_0 > 0$ for some
constant $\epsilon_0$ depending  independent of $N$, the integrand
of the $X^g \backslash X^g(N)$ integral is bounded below by
$\epsilon_0 N^g$ and therefore the integral is bounded below by a
constant (depending only on the genus) times $N^g$. Since the
integrand is positive, the addition of the integral over $X^g(N)$
only increases the quantity and therefore,
$$\begin{array}{l}  \frac{1}{N^2}  \log  \int_{X^{g}}
\; (\det \left( \langle \Phi_{N + g}^{P_j}, \Phi_{N + g}^{P_k}
\rangle \right)^*
 \left| \left(  \prod_{j < n} E(P_n, P_j)\right) \right|^{-2}  \prod_j d^2 P_j   \geq  \frac{C_g \log
 N}{N}.
\end{array}$$
This proves the lower bound half of the desired estimate.

We then prove the upper bound. The main contribution  comes from
the integral over $X^g(N)$. On $X^g \backslash X^g(N)$,  $|E(P_j,
P_k) \geq \frac{\log N}{\sqrt{N}}$ and the product $ \left| \left(
\prod_{j < n} E(P_n, P_j)\right) \right|$ is bounded below by
$(\frac{\log N}{\sqrt{N}})^g$. Also by the Hadamard inequality,
and the fact that $||\Phi_{N + g}^P||_{L^{\infty}} \leq N$,
$|(\det \left( \langle \Phi_{N + g}^{P_j}, \Phi_{N + g}^{P_k}
\rangle \right)^*|$ is bounded above by $N^g$. Hence,
$\frac{1}{N^2}  \log$ of the integral over $X^g \backslash X^g(N)$
is $O(\frac{\log N}{N})$.

Thus it suffices to give an upper bound for the integral over
$X^g(N)$.
 By a slight
extension of the same argument, we decompose the remaining set
$\{\vec P \in X^g: exists i \not= j: d(P_i, P_j) < \frac{\log
N}{\sqrt{N}}\}$ into sets where there are $r$ clusters of points
each within $<\frac{\log N}{\sqrt{N}}$ of each other and such that
points of each cluster are $\geq \frac{\log N}{\sqrt{N}}$ apart
for distinct clusters. If each cluster contains just one point
then we are back to the case with $ \frac{\log N}{\sqrt{N}}$
separated points. In each cluster, and with $j \not= k$  we write
$\Phi^{P_j}(z) - \Phi^{P_k} = (P_j - P_k) F_N(P_j, P_k)$. We then
multiply by $(E(P_j, P_k))^{-1}$. There are two entries for each
$(P_k, P_j)$ and so the cancellation leaves the smooth matrix
function $\det (F(P_j, P_k))$ and we need an upper bound for
$\frac{1}{N^2 } \log \int_{X^g(N)} \det (F(P_j, P_k)). $ Each
column involves at most one derivative of $\Phi_{N + g}^{P_j}$,
whose norm is then at most $N^2$. By the Hadamard determinant
inequality, $\frac{1}{N^2 } \log \int_{X^g(N)} \det (F(P_j, P_k))
\prod d^2 P_j $ is of order at most $\frac{\log N}{N}$.

\end{proof}

We now complete the proof of the upper bound: As in \cite{ZZ},
$\ecal_{\omega}^M(\sigma)$ is continuous and
$J_{\omega}^K(\sigma)$ is lower semi-continuous with respect to
weak convergence. It follows that
$$  \lim_{\delta \downarrow 0} \limsup_N \frac{1}{N^2}  \log \; {\bf
Prob}_N (B(\sigma, \delta))\leq \lim_{N\to\infty} \Theta_N+
\frac12\ecal_h^M(\sigma)-J_h^K(\sigma) + \epsilon\,.$$ Since
$\ecal_{\omega}^M(\sigma) \to \ecal_{\omega}(\sigma)$ as $M \to
\infty$ by monotone convergence, and since $\epsilon$ is
arbitrary,  we obtain \eqref{UBTOPROVE}. \end{proof}

To complete the proof of Theorem \ref{LD}, we also  need to prove
the lower bound:
\begin{lem}
\begin{equation} \label{LBTOPROVE}
    \lim_{\delta \downarrow 0} \liminf_N \frac{1}{N^2}  \log \; {\bf
Prob}_N (B(\sigma, \delta)) \geq - \tilde I^{\omega,K}(\sigma).
\end{equation}
\end{lem}
Exactly as  in \cite{ZZ} (Lemma 31), it suffices   to prove
\eqref{LBTOPROVE}
 when  $\sigma=f\omega\in \mcal(X)$ with $f$ a strictly
     positive and continuous function on $X$.

 \begin{proof}
We closely follow the  proof of the LDP lower bound  in \cite{ZZ}
until the last step where we apply Lemma \ref{ZEROa} rather than
the product measure argument in \cite{ZZ}.  Under the assumption
that $\sigma = f \omega$, we can construct a sequence of discrete
probability measures
$$d\sigma_N = \frac{1}{N} \sum_{j = 1}^N \delta_{Z_j} \in B(\sigma, \delta)$$
with the following properties:
\begin{enumerate}
\item $\sigma_N\in B(\sigma,\delta/2) $ for all $N$ large; \item
$d(Z_i, Z_j) \geq \frac{C(\sigma,\delta)}{\sqrt{N}} $ for $i \not=
j$.
\end{enumerate}
 Define
$$D_N^{\eta} = \{\zeta \in X^N: d(\zeta_j,  Z_j) \leq
\frac{\eta}{N}, \;\; j = 1, \dots, N \}. $$ Then, for $\eta$ small
enough and all $N$ large, all $\zeta \in D_N^{\eta}$ satisfy that
$\mu_{\zeta} \in B(\sigma, \delta).$ Since $D_I^{\eta} \subset
B(\sigma, \delta),$
\begin{equation} \label{eq-180209c}
    {\bf
    Prob}_N (B(\sigma, \delta))  \geq \int_{D_N^{\eta}}  e^{- N^2
I_N(\mu_{\zeta}) } \tilde{\kappa}_N+\Theta_N,
\end{equation}
where $I_N =  I_N^{\omega, \nu}$ is the approximate rate function.
Following the Green's function estimates up to (62) of \cite{ZZ}),
we get  that for any $\epsilon'>0$ and all $N$ large enough,
\begin{equation}
    \label{ofer-new}
    {\bf  Prob}_N (B(\sigma, \delta))  \geq
e^{-N^2I^{\omega, K}(\sigma)-3\epsilon' N^2} \int_{D_N^{\eta}}
\tilde{\kappa}_N
\end{equation}

By  Lemma \ref{ZEROa}, we have the lower bound (for some $C \geq
0$)
$$ \frac{1}{N^2}  \log \int_{\tilde{X}^{(N)}} \tilde{\kappa}_N
\geq - C \frac{\log N}{N},
$$
and together with (\ref{ofer-new}) it implies the desired lower
bound (\ref{LBTOPROVE}).

\end{proof}

\subsection{\label{CONSTANT} The normalizing constant: Proof of Lemma \ref{NC}}

To complete the proof of Theorem \ref{LD},  we need to determine
the logarithmic asymptotics of the normalizing constant
$\hat{Z}_N(\omega)$, or equivalently of $\Theta_N$ (\ref{THETAN}).
In fact, in the course of the proof we also introduced a constant
$B(N, g)$ and in the proof it was also absorbed into $\Theta_N$.
We now denote the overall constant by $\hat{Z}_N(\omega)$.  We
have proved the LDP for the measure multiplied by this constant.
We then determine $\Theta_N$ from the fact that ${\bf Prob}_N  $
 is a
probability measure. As in \cite{ZZ}, Lemma 4 (see Section 7.4):

\begin{lem} \label{NC}
 We have, $$- \lim_{N \to \infty}
\frac{1}{N^2} \Theta_N =  \lim_{N \to \infty} \frac{1}{N^2} \log
\hat{Z}_N(\omega) = \frac12 \log \mbox{\rm Cap}_{\omega} (K)\,.
$$
    \end{lem}

For the sake of completeness, we include the proof from \cite{ZZ}:
\begin{proof}
 By Proposition \ref{GOODRATEF} and the  proof of the large deviations upper
 bound in Lemma \ref{UBTOPROVE},
\begin{equation*} \begin{array}{lll}  0 & = & \lim_{N \to \infty} \frac{1}{N^2} \log {\bf Prob}
_N(\mcal(X )) \\ && \\ & \leq& \limsup_{N \to \infty}
\frac{-1}{N^2} \log \hat{Z}_N(h)   - \inf_{\mu \in \mcal(X}
I^{\omega, K} (\mu) \\ && \\
& = & \limsup_{N \to \infty} \frac{-1}{N^2} \log \hat{Z}_N(h)   -
I^{\omega, K}(\nu_{\omega, K})  \\ && \\
& = & \limsup_{N \to \infty} \frac{-1}{N^2} \log \hat{Z}_N(h)
-\frac{1}{2} \log \mbox{Cap}_{h} (K).
\end{array}
\end{equation*}
A similar argument using the large deviations lower bound shows
the reverse inequality for $\liminf_{N\to\infty} \frac{-1}{N^2}
\log \hat{Z}_N.$
\end{proof}

\section{\label{APPENDIX} Appendix on determinants and bosonization}

 Up to the constant factor, the bosonization we quote in Lemma
 \ref{BOSONa} is relatively simple to prove (see  Fay \cite{F}).
Following Fay's presentation, we describe
 line bundles and
their sections by their automorphy factors.
 For genus one,  $X$ is expressed
as $\C \backslash \Gamma$ while for $g \geq 2$,  $X = \H
\backslash \Gamma$ with $\Gamma \subset SL(2, \R)$ and $\H$ the
upper half plane.  In either case, $L$ is defined by a factor of
automorphy $\phi_{\gamma}(z)$,
$$\phi_{\gamma}(z) = \chi_{\gamma} \frac{t(\gamma z)}{t(z)} \prod_{i = 1}^d \frac{E(\gamma z, a_i)}{E(z, a_i)}, \;\; (\gamma \in \Gamma),$$
where $t(z)$ is a nowhere vanishing holomorphic function on $\H$.
We also put
$$\sigma(p) = \exp - \left(\sum_{k = 1}^g \int_{A_k} v_k(x) \log E(x, p) \right),\;\;\;\; \sigma(p, p_1) = \frac{\sigma(p)}{\sigma(p_1)}. $$
By Proposition 1.2 of \cite{F}, $\sigma(p)$ is a nowhere vanishing
holomorphic automorphic function on $\H$ with automorphic factors
given in \cite{F} (1.12).  Finally, let $\Delta$ denote the vector
of Riemann constants and let $\theta[\chi]$ denote the theta
function with characteristic $\chi$.

Theorem 1.3 of \cite{F} states the following:  {\it  Let $L$ be
the line bundle $\chi \otimes D$ of degree $d \geq g-1$ with $D =
\sum_{i = 1}^N a_i$ and with $\chi$ a unitary character.
 Then there exists a constant  $f_L$  depending only on the marking of the Riemann
 surface so that, for
any basis $\{\omega_j\}$ of $H^0(L)$ and any points $x_1, \dots,
x_{d + 1 - g} \in X$, so that
 \begin{equation} \label{FAY} \det \begin{pmatrix} \omega_j(x_k) \end{pmatrix}_{j, k =1}^{d + 1 - g} = f_L^{-1} \; \theta[\chi]
  (\sum_{i = 1}^d a_i - \sum_{i = 1}^{d + 1 - g} x_i - \Delta)
 \; \frac{\prod_{i < j}^{d + 1 - g} E(x_i, x_j) \prod_{i= 1}^{d + 1  - g} t(x_i)}
{\prod_{1}^{d + 1 - g} \sigma(x_i, z_0)}, \end{equation}}

Let us sketch the proof of (\ref{FAY}). The main point is to show
that $f_L$ is a constant depending only on the marking. The first
point is  that  $f_L$ is a meromorphic function of $x_i \in X$. As
noted above, the Slater determinant  $\det \begin{pmatrix}
\omega_j(x_k)
\end{pmatrix}$ is a section of  $\pi_1^* L
 \otimes \cdots \otimes \pi_{d + 1 - g}^* L \to X^{(d + 1 - g)}$. With Fay's definition of the prime
 form $E$, the right side is also a section of this bundle. To
 verify the formula is to prove  that $f_L$ has no zeros or  poles.

  For generic $x_2, \dots, x_{d + 1 - g}$,  $\det
\begin{pmatrix} \omega_j(x_k)
\end{pmatrix}$ vanishes at $x_1 = x_2, \dots, x_{d + 1 - g}$ and at
$ g $ further points $\xi$ such that $x_2 + \cdots + x_{d + 1 - g}
+ \xi = [L] \in \mbox{Pic}^d(X)$. Also, $h^1(L \otimes (-\sum_2^{d
+ 1 - g} x_j)) = 0$ for generic $x_2, \dots, x_{d + 1 - g}$. By
Riemann's theorem (\S \ref{THETASTUFF}), the zeros of
$\theta[\chi] (\sum_{i = 1}^d a_i - \sum_{i = 1}^{d + 1 - g} x_i -
\Delta)$ in $x_1$ occur at $g$ points $\eta$ for which
$$[L] - \sum_{2}^{d + 1 - g} x_i = \eta. $$
It follows that $\xi = \eta$ and that $f_L$ has no zeros or poles
in $x_1$. Similarly, $f_L$ is constant in all of the variables
$x_j$. Hence it is constant, completing the proof.

The  missing detail in this formula is an explicit formula for
$f_L$, which in our problem depends on $N$.  It is possible that
this factor is cancelled by the normalizing factor $Z_N(h)$ in
Theorem \ref{JPDHG}. But it is useful  to recall the explicit
formula  for $f_L$.

\subsection{Bosonization formulae}  In view of the number of complicated invariants,
it is useful to compare this to the original bosonization formulae
of \cite{ABMNV} (see also \cite{VV} and \cite{Fal}).  A special
case of (4.15) of \cite{ABMNV} (in the notation of that article)
is the formula,
\begin{equation} \begin{array}{l} \frac{\det' \dbar^*_{\lcal_b}
\dbar_{\lcal_b}}{\det (\psi_i, \psi_j)} \left| \det
\begin{pmatrix}
\psi_1(P_1) & \cdots & \psi_{p } (P_1)  \\ & & \\
\psi_1(P_p) & \cdots & \psi_{p } (P_p)  \end{pmatrix} \right|^2  =
\left( \frac{\det' \dbar^* \dbar}{\det(i Y)^{-1} \cdot A_{X}
}\right)^{-\half} \cdot \ncal(z) \cdot \prod_{i, j = 1}^p G(P_i,
P_j). \end{array} \end{equation} In this formula, a Riemannian
metric is given on $X$,  $\dbar^* \dbar$ is its Laplacian,  $A_X$
is the area of $X$, and $G$ is its `regulated coincident' Green's
function.
 Further $\lcal_b$ is a holomorphic line
bundle of degree $2 \lambda (g - 1)$, $(\cdot, \cdot)$ is an inner
product on $H^0(X, \lcal_b)$ and  $\{\psi_j\}$ is a basis of
$H^0(X, \lcal_b)$. Also, $\dbar_{\lcal_b}: C^{\infty}(X, \lcal_b)
\to C^{\infty}(X, \lcal_b \otimes \overline{K_X})$ is the natural
$\dbar$ operator and $\dbar_{\lcal_b}^*$ is its adjoint in the
inner product induced by the Hermitian metric on $\lcal_b$ and the
Riemannian metric on $X$. Also, $G(P_i, P_j) = E(P_i, P_j)$ in our
notation,  $(i Y)^{-1}$ is the period matrix. and the factor $\det
(i Y)^{-1} \ncal$ is the spin $\half$ determinant. The rest of the
notation is defined in \S \ref{BOSON}. It follows that the
constant $A_N(g, \omega)$ defined there is a ratio twisted Laplace
determinants and some non-vanishing factors independent of $N$.
The proof that the logarithms of  Laplace determinants depend only
linearly on $N$ is given in \cite{BV}.

\end{document}